\documentclass[11pt, a4paper]{article}

\usepackage[utf8]{inputenc}
\usepackage[T1]{fontenc}
\usepackage{lmodern}
\usepackage{microtype}
\usepackage[pdfborder={0 0 0}]{hyperref}
\usepackage{color}
\usepackage{ifthen} 
\usepackage{paralist}
\usepackage{graphicx}
\usepackage[all]{xy}
\xyoption{rotate}
\usepackage{verbatim}

\usepackage[tbtags]{amsmath} 
\usepackage{amssymb}
\usepackage{amsthm}
\usepackage{mathtools}

\usepackage[text={16.5cm, 24.2cm}, centering]{geometry}
\newcommand{\inv}[0]{{-1}}

\newcommand{\oo}[0]{\otimes}

\newcommand{\id}[0]{\mathrm{id}}

\newcommand{\low}[2]{{#1}_{({#2})}}

\newcommand{\st}[0]{{\bf s}}
\newcommand{\ta}[0]{{\bf t}}
\newcommand{\gammad}[0]{{\Gamma_D}}

\newcommand{\prot}[0]{\mathcal H_{\mathrm {pr}}}
\newcommand{\ham}[0]{H_{\mathrm{K}}}
\newcommand{\rhop}[0]{{\boldsymbol \rho}}

\newtheorem{theorem}{Theorem}[section]
\newtheorem*{theorem*}{Theorem}
\newtheorem{example}[theorem]{Example}
\newtheorem{lemma}[theorem]{Lemma}
\newtheorem{proposition}[theorem]{Proposition}

\newtheorem{definition}[theorem]{Definition}
\newtheorem{remark}[theorem]{Remark}

\newcommand*\stdthebibliography{}
\let\stdthebibliography\thebibliography
\renewcommand{\thebibliography}[1]{%
\stdthebibliography{#1}\setlength{\itemsep}{-1pt}}

\def\mytitle{Kitaev lattice models as a Hopf algebra gauge theory}
\def\myauthors{C.~Meusburger}
\hypersetup{pdftitle=\mytitle, pdfauthor=\myauthors}

\begin{document}

\begin{center}
  {\huge\mytitle}

  \vspace{2em}

  {\large
   Catherine Meusburger\footnote{{\tt catherine.meusburger@math.uni-erlangen.de}} 
 }

 Department Mathematik \\
  Friedrich-Alexander-Universit\"at Erlangen-N\"urnberg \\
  Cauerstra\ss e 11, 91058 Erlangen, Germany\\[+2ex]

{July 5, 2016}

  \begin{abstract}
\noindent   We prove that  Kitaev's lattice  model for a finite-dimensional semisimple Hopf algebra $H$ is equivalent to the combinatorial quantisation of Chern-Simons theory for the Drinfeld double $D(H)$. This shows that Kitaev models are a special case of the older and more general combinatorial models.  
This  equivalence  is  an analogue of the relation between Turaev-Viro  and Reshetikhin-Turaev TQFTs and relates
them to the quantisation of moduli spaces of flat connections. 
 
We show that the topological invariants of the two models, the algebra of operators acting on the protected space  of the Kitaev model and the quantum moduli algebra from the combinatorial quantisation formalism, are isomorphic. This is established in a gauge theoretical picture, in which both models appear as  Hopf algebra valued  lattice gauge theories.  
  
  We first prove that the triangle operators of a Kitaev model  form a module algebra over a Hopf algebra of gauge transformations and  that this module algebra  is  isomorphic  to  the lattice algebra in the combinatorial formalism. Both algebras can be viewed as the algebra of  functions on gauge fields in  a Hopf algebra gauge theory.  
 The isomorphism between them induces an algebra isomorphism between their
 subalgebras of invariants,  which are interpreted as  gauge invariant functions or observables. 
It  also relates the curvatures in the two models, which are given as holonomies around  the faces of the lattice.
This yields an isomorphism between the subalgebras obtained by projecting out curvatures, which can be viewed as the algebras of functions on flat gauge fields and are the topological invariants of the two models. 
  \end{abstract}
\end{center}

\section{Introduction}

{\bf Motivation} \newline
Kitaev models \cite{Ki} and Lewin-Wen models \cite{LW1,LW2}, which were shown to be equivalent to Kitaev models in \cite{BA, KMR2},   have  attracted strong interest  in condensed matter physics and  topological quantum computing.   
They 
 assign to 
 an oriented surface $\Sigma$ a finite-dimensional 
 vector
 space, the {\em protected space}, which is a topological invariant of $\Sigma$. They also exhibit topological excitations with braid group statistics, electro-magnetic duality and can be equipped with additional structures such as defects and domain walls \cite{BMD, KK,BBK}, which are a focus of current research.
 
Kitaev models are  also of strong interest from  the perspective of topological quantum field theory  as they are related to Turaev-Viro TQFTs \cite{TVi,BW}.  It was shown in \cite{BK,KKR,KMR,KMR2}  that the protected space  of a Kitaev model for a finite-dimensional semisimple Hopf algebra  $H$  on an oriented surface $\Sigma$ coincides with the vector space $Z_{TV}(\Sigma)$ that the Turaev-Viro TQFT  for the representation category $H$-Mod assigns to $\Sigma$.   
If one interprets Turaev-Viro invariants  as a discretised path integrals \cite{john1,john2}, 
Kitaev models  can  be viewed as their Hamiltonian counterparts. From this point of view,  the linear map $Z_{TV}(M): Z_{TV}(\Sigma)\to Z_{TV}(\Sigma')$ that the Turaev-Viro TQFT assigns to a 3-manifold with boundary $\partial M=\Sigma\coprod \bar\Sigma'$ describes a transition between  Kitaev models on $\Sigma$
and  $\Sigma'$. 

This relation between Kitaev models and TQFTs extends to the case with excitations. 
Excitations in a Kitaev model for  $H$ are labelled by representations of  the Drinfeld double $D(H)$ and correspond to Turaev-Viro TQFTs with line defects \cite{KMR,BK}. These were essential in establishing the equivalence of Turaev-Viro TQFTs  for  $H$-Mod and the Reshetikhin-Turaev TQFTs   for  $D(H)$-Mod in \cite{Ba1,Ba2,KB,TVr}. 

The topological nature of Kitaev models and their relation to Turaev-Viro and Reshetikhin-Turaev TQFTs raises three obvious  questions, which remained open despite their conceptual importance:
\begin{compactenum}
\item  Is there a Hamiltonian analogue of Reshetikhin-Turaev TQFTs  defined along the same lines as Kitaev models?  If yes, can one precisely relate  the 
Kitaev model for a Hopf algebra $H$ and the
Hamiltonian analogue of a Reshetikhin-Turaev TQFT  for its Drinfeld double $D(H)$? 
\\[-1ex]

\item Can a Kitaev model for a Hopf algebra $H$  be viewed as a Hopf algebra analogue of a group-valued lattice gauge theory?  This is strongly suggested by  the structure of Kitaev models, e.~g.~that they are defined in terms of a graph  embedded in a surface $\Sigma$,  that they exhibit  symmetries at each vertex and face of the graph and that these symmetries act trivially on the protected space.  It is also likely due to their relation to 
 Reshetikhin-Turaev  TQFTs, which were  obtained  by quantising  Chern-Simons gauge theories \cite{W,RT}.\\[-1ex]
 
\item  Are there classical analogues of a Kitaev models defined in terms of Poisson or symplectic structures associated with $\Sigma$? Although Kitaev models were defined ad hoc and not by quantising  classical structures, this seems natural due to their interpretation as quantum models. It is also suggested by their relation to Reshetikhin-Turaev TQFTs, which quantise  Chern-Simons gauge theories and hence moduli spaces of flat connections on $\Sigma$.
\end{compactenum}

\medskip
{\bf Main results}\newline
The article addresses these three questions.    More specifically, we show that the algebra of operators that act on the protected space of a Kitaev model  is isomorphic to the quantum moduli algebra obtained in  the combinatorial quantisation  of Chern-Simons gauge theory by Alekseev, Grosse and Schomerus \cite{AGSI,AS} and by Buffenoir and Roche \cite{BR,BR2}.  The protected space of a Kitaev model for a finite-dimensional semisimple Hopf algebra $H$ is therefore given as a representation space of the quantum moduli algebra for its Drinfeld double $D(H)$.
This establishes a complete equivalence between the Kitaev models and the combinatorial quantisation of Chern-Simons theory and shows that Kitaev models are not a new class of models but a special case of the older and more general models in \cite{AGSI,AGSII,AS, BR, BR2}. 

Moreover, it  was   shown in \cite{BR2,AS} that the quantum moduli algebra can be viewed as the Hamiltonian analogue of a Reshetikhin-Turaev TQFT.  Its representation spaces coincide with the vector spaces $Z_{RT}(\Sigma)$ that the Reshetikhin-Turaev TQFT assigns to $\Sigma$, and both give rise to the same action of the mapping class group  $\mathrm{Map}(\Sigma)$. This addresses the first question and establishes a relation $(*)$ between Kitaev model for $H$  and the combinatorial model for $D(H)$ that is analogous to the relation  between a Turaev-Viro TQFT for the representation category  $H$-Mod and the Reshetikhin-Turaev TQFT for the representation category $D(H)$-Mod from \cite{Ba1,Ba2,KB,TVr}:
$$
\xymatrix{
\text{\framebox{Turaev-Viro TQFT for $H$-Mod} }\qquad \ar[d]^{\cite{BK,KKR,KMR}}_{\text{Hamiltonian analogue}} \ar[r]^{\cite{Ba1,Ba2,KB,TVr}\qquad} & \ar[l] \qquad \;\;\;\;\;\text{\framebox{Reshetikhin-Turaev TQFT for $D(H)$-Mod}} \ar[d]_{\cite{AS,BR2}}^{\text{Hamiltonian analogue}}\\
\text{\framebox{Kitaev model for $H$} }  \ar@{-->}[r]^{(*)\qquad\qquad} &  \ar@{-->}[l] \text{ \framebox{combinatorial model for $D(H)$}} 
}
$$
As the quantum moduli algebra was obtained in \cite{BR,AGSI} by canonically quantising  the symplectic structure on the moduli space of flat connections on $\Sigma$ from  \cite{FR,AM}, this also addresses the third question about the corresponding Poisson and symplectic structures. Moreover, all structures that describe  the relation between the two models have Poisson analogues in the theory of Poisson-Lie groups.  
Although this aspect is not developed further here, this defines  Poisson analogues of  Kitaev models  and relates them to symplectic structures on  moduli spaces of flat connections on surfaces. 

Finally,  it was shown in \cite{MW}  that the combinatorial model \cite{AGSI,AGSII, AS,BR,BR2} can  be derived  from a set of simple axioms that generalise  group valued lattice gauge theories to ribbon Hopf algebras.  These axioms encode minimal physics requirements for a local  lattice gauge theory  and imply that the relevant mathematical structures  are module algebras over Hopf algebras. In this sense, the combinatorial model is a generalisation of the notion of a lattice gauge theory from a group to a ribbon Hopf algebra. The equivalence of Kitaev models to combinatorial models  shows that the former  can indeed be interpreted as a Hopf algebra valued lattice gauge theory.  

This viewpoint  is also essential in establishing the correspondence between the two models.  Many of the results in 
\cite{Ki,BMCA} and in \cite{AGSI,AGSII, AS, BR,BR2} are formulated in terms of matrix elements in irreducible representations. While  
 this presents certain computational advantages within the models, it becomes an obstacle when relating them. 
The more algebraic and basis independent formulation  in terms of a Hopf algebra gauge theory in  \cite{MW} allows one to relate the models in a simpler and more conceptual way.

\medskip
{\bf Detailed description of results}\newline
In  this article  we consider Kitaev models for a finite-dimensional semisimple Hopf algebra $H$, as introduced in \cite{BMCA},  and the combinatorial model for the Drinfeld double $D(H)$ from \cite{AGSI,AGSII,BR}, in its formulation  as a Hopf algebra gauge theory  \cite{MW}. Besides the Hopf algebras $H$ and $D(H)$  the input data in both models is a  ribbon graph $\Gamma$, a directed graph with a cyclic ordering of the edge ends at each vertex. The relation between the two models is obtained by thickening $\Gamma$ to a ribbon graph $\gammad$ in which each edge of $\Gamma$ is replaced by a rectangle and each vertex of $\Gamma$ by a polygon. While the Hopf algebra gauge theory  is associated with the ribbon graph $\Gamma$, the natural setting for the Kitaev model is the thickened graph $\gammad$.

 The  concept that is fundamental in relating the two models is {\em holonomy}, which we define as a functor from the path groupoid of  $\Gamma$ or  $\gammad$ into a category constructed from the Hopf algebra data.  For {\em ribbon paths} these holonomies yield  Kitaev's {\em ribbon operators} \cite{Ki}. However, our notion of holonomy is   more general and  defined for  any path in $\gammad$, not only ribbon paths.
 This is essential in relating the two models. 
 By selecting an adjacent face at each vertex - a {\em site} in the language of Kitaev models or a {\em cilium} in the language of Hopf algebra gauge theory -  we associate to each oriented 
  edge $e\in E$ the following two paths $p_{e,\pm}$ in $\gammad$.

 \vspace{-1cm}
\begin{center}
\includegraphics[scale=0.27]{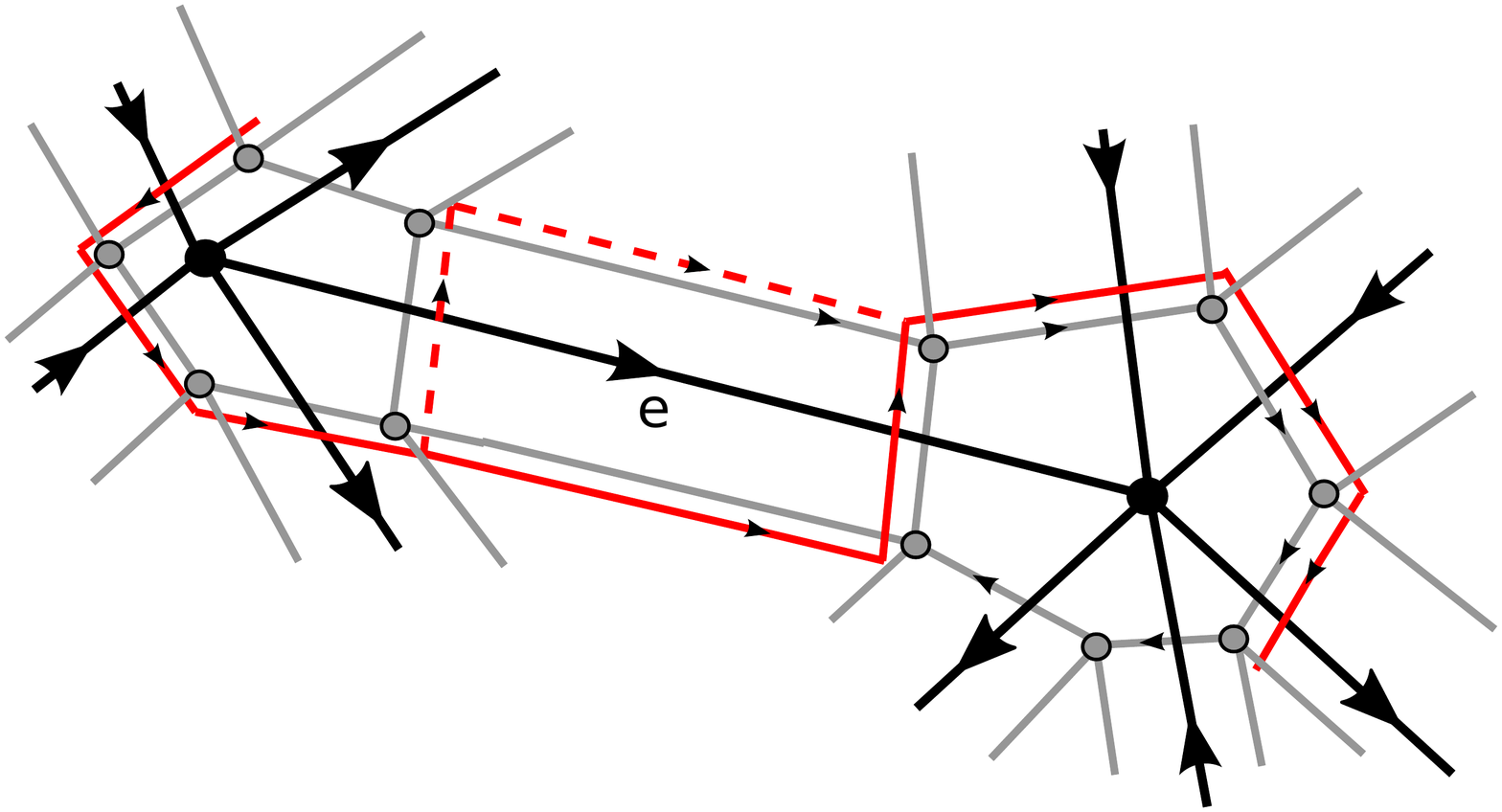}
\end{center}

 \vspace{-1cm}
 The holonomies of the  paths $p_{e,\pm}$, which are {\em not} ribbon paths in the sense of \cite{Ki,BMD,BMCA}, coincide. 
By sending the holonomy of an edge $e$ of $\Gamma$ 
  to the holonomies of the paths $p_{e,\pm}$ in $\gammad$, we then  obtain an explicit relation  between the algebraic structures in a Hopf algebra gauge theory and in the Kitaev models. 
This relation involves three layers:

$\bullet$ {\bf The algebra of functions and the  algebra of triangle operators:} The algebra $\mathcal A^*_\Gamma$  of functions in a Hopf algebra gauge theory for $D(H)$ is the vector space $D(H)^{*\oo E}$ obtained by associating a copy of the dual Hopf algebra $D(H)^*$ to each edge of $\Gamma$. However, its algebra structure  is not the canonical one from the tensor product, but deformed at each vertex by the universal $R$-matrix of $D(H)$. 
The corresponding structure in the Kitaev model for $H$ is the algebra generated by the triangle operators  $L^{h}_{e,\pm}$ and $T^\alpha_{e,\pm}$  for each each edge $e$ of $\Gamma$ and indexed by elements $h\in H$ and $\alpha\in H^*$. These triangle operators act on the vector space
 $H^{\oo E}$  via the left and right regular action of $H$ and  $H^*$ on $H$. They form an  algebra  isomorphic to the $E$-fold tensor product $\mathcal H(H)^{\oo E}$ of the Heisenberg double  of $H$.  The first central result is that under certain assumptions on $\Gamma$ these two algebras are isomorphic (Theorems \ref{th:hdcomb} and \ref{th:chiiso}):

\begin{theorem*} Let $\Gamma$ be a regular ribbon graph. Then the holonomies of the paths $p_{e,\pm}$ induce an algebra isomorphism $\chi: \mathcal A^*_\Gamma\to \mathcal H(H)^{op\oo E}$ from the algebra of functions $\mathcal A^*_\Gamma$ of a Hopf algebra gauge theory for $D(H)$ to the algebra formed by the triangle operators in the Kitaev model for $H$.
\end{theorem*}

$\bullet$ {\bf Gauge symmetries:} The second layer of correspondence concerns the gauge symmetries of the two models.  Gauge transformations in the Hopf algebra gauge theory  are obtained by associating a copy of the Drinfeld double $D(H)$ to each vertex of $\Gamma$ and are given by the Hopf algebra $D(H)^{\oo V}$.  The algebra $\mathcal A^*_\Gamma$ of functions is a module algebra over this Hopf algebra. Consequently, the submodule of invariants  is a subalgebra  $\mathcal A^*_{\Gamma\, inv}\subset \mathcal A^*_\Gamma$, the subalgebra of gauge invariant {\em observables}.

These  gauge symmetries correspond to the symmetries associated with the {\em vertex}  and {\em face operators}  in the Kitaev models. These operators are given as the holonomies of paths in $\gammad$ that go clockwise around the vertices  and faces  of $\Gamma$. For each site of $\Gamma$,  they  define a faithful representation of the Drinfeld double $D(H)$.
Our second main result in Theorem \ref{lem:actedge} states that these representations  give the  algebra of triangle operators  the structure of a module algebra over the Hopf algebra $D(H)^{\oo V}$. Moreover, we find  that the algebra isomorphism that relates it to the algebra of functions in a Hopf algebra gauge theory is compatible with these gauge symmetries:

\begin{theorem*} For each regular ribbon graph $\Gamma$, the vertex and face operators equip the algebra $\mathcal H(H)^{op\oo E}$ of triangle operators with the structure of a $D(H)^{\oo V}$-right module algebra.  The algebra isomorphism $\chi: \mathcal A^*_\Gamma\to \mathcal H(H)^{op\oo E}$ is a module morphism 
and induces an algebra isomorphism  between the subalgebras of invariants $\mathcal A^*_{\Gamma\,inv}\subset \mathcal A^*_\Gamma$ and  $\mathcal H(H)^{op\oo E}_{inv}\subset \mathcal H(H)^{op\oo E}$.
\end{theorem*}

$\bullet$ {\bf Curvature:} The third layer of correspondence between Kitaev models and Hopf algebra gauge theories concerns  {\em curvatures}, which are   the holonomies of the faces of $\Gamma$  and $\gammad$.  In the Hopf algebra gauge theory, the holonomies of the faces of $\Gamma$ give rise to an algebra morphism from the character algebra or the centre of $D(H)$ into the centre $Z(\mathcal A^*_{\Gamma\,inv})$  of the algebra  of gauge invariant observables. By taking the product of these holonomies over all faces of $\Gamma$ and inserting the Haar integral of $D(H)$, one obtains a projector on the {\em quantum moduli algebra} $\mathcal M_\Gamma$.

As the faces of the thickened graph $\gammad$ correspond to either faces, vertices or edges of $\Gamma$ and the holonomies of the latter are trivial, curvatures in Kitaev models are given by the vertex and face operators.  By taking the product of these holonomies for all vertices and faces of $\Gamma$ and inserting the Haar integrals of the Hopf algebras $H$ and $H^*$ one obtains Kitaev's {\em Hamiltonian} $\ham$, which is a projector from  $H^{\oo E}$  on the {\em protected space}.   

We show that  the Hamiltonian  defines a projector on a subalgebra $\mathcal H(H)^{\oo E}_{flat}\subset \mathcal H(H)^{\oo E}_{inv}$, which consists of those elements that satisfy the relation  $\ham\cdot X\cdot \ham=X$. This is  the subalgebra of operators that act on the {\em protected space}. 
We then prove that  for each site  of $\Gamma$ the algebra isomorphism $\chi$ sends the holonomy of the associated face in $\Gamma$ to the product of the associated vertex and face operator of the Kitaev model. 
This leads to our third main result in Theorem \ref{th:modth}:

\begin{theorem*} For each regular ribbon graph $\Gamma$, the  algebra isomorphism  $\chi: \mathcal A^*_\Gamma\to \mathcal H(H)^{op\oo E}$  induces an algebra isomorphism $\chi:\mathcal M_\Gamma\to \mathcal H(H)^{op\oo E}_{flat}$.
\end{theorem*}

Both, the quantum moduli  algebra $\mathcal M_\Gamma$ and Kitaev's protected space were shown to be topological invariants \cite{Ki, BMCA, AGSI,AGSII,BR,MW}. They depend only on the homeomorphism class of the oriented surface obtained by gluing discs to the faces of $\Gamma$. As the same holds for the algebra $\mathcal H(H)^{\oo E}_{flat}$
this theorem establishes a precise equivalence between the topological invariants of the Kitaev model for $H$ and the Hopf algebra gauge theory for $D(H)$.

This equivalence  between Kitaev models and Hopf algebra gauge theory also has a  geometrical interpretation. The thickened ribbon graph $\gammad$ for the Kitaev model can be viewed as a `graph double' of the ribbon graph $\gamma$, because it combines the ribbon graph $\Gamma$ with its Poincar\'e dual $\bar\Gamma$. In  Kitaev models  Poincar\'e duality corresponds to Hopf algebra duality: The edges of $\gammad$ associated with  edges of $\Gamma$ carry the triangle operators $T^\alpha_{e,\pm}$  indexed by elements  $\alpha\in H^*$. The edges of $\gammad$ associated with  edges of $\bar \Gamma$ carry the triangle operators $L^h_{e,\pm}$ indexed by elements  $h\in H$. 
This allows one to view the Kitaev model for the Hopf algebra $H$  as a {\em factorisation} of the Hopf algebra gauge theory for the Drinfeld double $D(H)$, in which  gauge fields on $\Gamma$ with values in $D(H)$ are factorised  into  gauge fields on $\Gamma$  and on $\bar\Gamma$ with values in  $H\subset D(H)$ and  $H^*\subset D(H)$.  Similarly,  the $D(H)$-valued gauge transformations and  curvatures in the Hopf algebra gauge theory  factorise into gauge transformations and curvatures on $\Gamma$ and on $\bar\Gamma$ with values in $H\subset D(H)$ and $H^*\subset D(H)$. The former are associated with vertex and the latter with face operators. 

{\bf Structure  of the article:} In Section \ref{sec:background} we introduce the relevant notation, the background on Hopf algebras  and the graph theoretical background needed in the article. 
Section \ref{subsec:kit} contains a brief  summary of Kitaev models, and  Section \ref{sec:gtheory} summarises the  background on combinatorial quantisation and  Hopf algebra gauge theory. In Section \ref{sec:holkitaev} we introduce a generalised notion of holonomy for Kitaev models, investigate its algebraic properties and show that it reduces to the ribbon operators for ribbon paths. In particular, this applies to the holonomies of  paths  around  the vertices and faces of $\Gamma$, which define the vertex and face operators of the Kitaev model.

In Section \ref{subsec:gsymmflat} we show that Kitaev models exhibit the mathematical structures of a Hopf algebra gauge theory. We prove that  the vertex and face operators give the algebra of triangle operators the structure of a right module algebra  over a Hopf algebra of gauge transformations and investigate its subalgebra of invariants. We show that the Hamiltonian of the Kitaev model arises from the curvatures of the faces of $\gammad$ and projects on a subalgebra  of operators  on the protected space.

Sections \ref{sec:kithopf} and \ref{sec:qmod} contain  the core results of the article. In Section \ref{sec:kithopf} we prove that  under certain assumptions on $\Gamma$ the holonomies of the paths $p_{e,\pm}$ in $\gammad$ induce an algebra isomorphism between the algebra of functions of a Hopf algebra gauge theory and Kitaev's triangle operator algebra.   In Section \ref{sec:qmod} we show that this algebra isomorphism is a morphism of module algebras  and hence induces an isomorphism between their subalgebras of invariants. We then establish that this algebra isomorphism  relates the curvatures of the two models and prove that it induces an isomorphism between the quantum moduli algebra and  the algebra of operators  acting on the protected space.

\section{Background}
\label{sec:background}

\subsection{Notations and conventions}
\label{subsec:conventions}

Throughout the article $\mathbb F$ is a field of characteristic zero.  
For an algebra $A$  and $n\in\mathbb N$ we denote by  $A^{\oo n}$ the $n$-fold tensor product of $A$ with itself, always taken over $\mathbb F$ unless specified otherwise. If $X$ is a finite set of cardinality $|X|$, we write $A^{\oo X}$ instead of $A^{\oo|X|}$ and denote by $\tau: \Pi_X A\to A^{\oo X}$ the canonical $|X|$-linear surjection. Identifying $\Pi_X A$ with the vector space of maps $f: X\to A$, we define
$a_x: X\to A$ for $x\in X$ and $a\in A$ as the map with $a_x(x)=a$ and $a_x(y)=1$ for $y\neq x$ and denote by $(a)_x=\tau(a_x)$ the associated element in $A^{\oo X}$.  This corresponds is the pure tensor in $A^{\oo X}$ that has entry $a$ in the copy of $A$ associated with $x\in X$ and $1$ in all other entries.

Similarly, for $a^1,...,a^n\in A$ and pairwise distinct $x_1,...,x_n\in X$, we  define $(a^1\oo...\oo a^n)_{x_1,...,x_n}$ as the 
 pure tensor in $A^{\oo X}$ that has entry $a^1$ in the copy of $A$ in $A^{\oo X}$ associated with $x_1$, $a^2$ in the copy associated with $x_2$ etc. It is given by $(a^1\oo...\oo a^n)_{x_1...x_n}:=\tau((a^1,...,a^n)_{x_1,...,x_n})$, where
$(a^1,...,a^n)_{x_1...x_n}:=(a^1)_{x_1}\cdot (a^2)_{x_2}\cdots (a^n)_{x_n}: X\to A$ and the product is taken with respect to the pointwise multiplication in $A$. 
We denote by $\iota_{x_1...x_k}: A^{\oo n}\to A^{\oo X}$, $a^1\oo...\oo a^n\mapsto (a^1\oo...\oo a^n)_{x_1,...,x_n}$
the associated inclusion maps.  For linear maps $f_1,...,f_n: A\to A$ we define the linear map
$(f_1\oo...\oo f_n)_{x_1...x_n}: A^{\oo X}\to A^{\oo X}$ by  $(f_1\oo ...\oo f_n)_{x_1,...,x_n}((a)_{x_i})=(f_i(a))_{x_i}$ for all $i\in\{1,...,n\}$ and $(f_1\oo ...\oo f_n)_{x_1,...,x_n}((a)_{y})=(a)_y$ for $y\notin\{x_1,...,x_n\}$.

For Hopf algebras  we use Sweedler notation  without summation signs. We write
$\Delta(h)=\low h 1\oo\low h 2$
for the comultiplication $\Delta: H\to H\oo H$ of a Hopf algebra $H$  and
 also use this notation for elements of $H\oo H$, e.~g.~$R=\low R 1\oo\low R 2$ for a universal $R$-matrix. We denote by
$H^{op}$ and $H^{cop}$, respectively, the Hopf algebra  with the opposite multiplication and comultiplication and by $H^*$ the dual Hopf algebra.
Unless specified otherwise, we use Latin letters for elements of $H$ and Greek letters for  elements of $H^*$. 
 The pairing between $H$ and $H^*$ is denoted 
 $\langle\;,\;\rangle: H^*\oo H\to\mathbb F$, $\alpha\oo h\mapsto \alpha(h)$, and 
the same notation is used for the induced pairing $\langle\,,\,\rangle: H^{*\oo n}\oo H^{\oo n}\to\mathbb F$.

\subsection{Hopf algebras}

In this subsection we summarise  basic facts about Hopf algebras. 
 Unless  specific citations are given,  the results can be found in textbooks on Hopf algebras, such as  \cite{K, majidbook, Mon, rbook}. 

\begin{theorem}[\cite{Dr}] \label{lem:ddouble2}  
Let $H$ be a finite-dimensional Hopf algebra. Then there exists a unique quasitriangular Hopf algebra structure on $H^*\oo H$ for which $H\cong1\oo H$ and $H^{*cop}\cong H^{*cop}\oo 1$ are Hopf subalgebras.  In terms of a basis $\{x_i\}$  of $H$ and  the dual basis $\{\alpha^i\}$ of $H^*$, it is given by
\begin{align}\label{eq:dmult}
&(\alpha\oo h)\cdot (\alpha'\oo h')=
\langle \alpha'_{(3)}, h_{(1)}\rangle \langle S^\inv(\alpha'_{(1)}),h_{(3)}\rangle\, \alpha\alpha'_{(2)}\oo h_{(2)}h' 
& &1=1_{H^*}\oo 1_H\\
&\Delta(\alpha\oo h)=\alpha_{(2)}\oo h_{(1)}\oo\alpha_{(1)}\oo h_{(2)} & &\epsilon(\alpha\oo h)=\epsilon(\alpha)\epsilon(h)\nonumber\\
&S(\alpha\oo h)
=\langle\low\alpha 1,\low h 3\rangle\langle S^\inv(\low\alpha 3),  \low h 1\rangle\; S(\low\alpha 2)\oo S(\low h 2) &  &R=\Sigma_i\,1\oo x_i\oo\alpha^i\oo 1
.\nonumber
\end{align}
This Hopf algebra  is called the  {\bf Drinfeld   double}  of $H$ and denoted $D(H)$.
\end{theorem}

\begin{remark} If $\{x_i\}$  a basis of $H$ and $\{\alpha^i\}$ the dual  basis of $H^*$, 
the dual Hopf algebra $D(H)^*$ of the Drinfeld double $D(H)$ is the vector space $H\oo H^*$ with the following Hopf algebra structure
\begin{align}\label{eq:ddualmult}
&(y\oo\gamma)\cdot (z\oo\delta)=zy\oo\gamma\delta & &1=1_H\oo 1_{H^*}\\
&\Delta(y\oo\gamma)=\Sigma_{i,j}\;\low y 1\oo \alpha^i\low\gamma 1\alpha^j\oo S(x_j)\low y 2x_i\oo\low \gamma 2 & &\epsilon(y\oo\gamma)=\epsilon(y)\epsilon(\gamma)\nonumber\\
&S(y\oo\gamma)= \Sigma_{i,j} \;x_iS^\inv(y)x_j \oo S(\alpha^j)S(\gamma)\alpha^i.\nonumber
\end{align}
\end{remark}

In this article we mostly restrict attention to finite-dimensional semisimple Hopf algebras $H$, since these are the Hopf algebras used in Kitaev models. 
Recall that a finite-dimensional Hopf algebra $H$ over a field $\mathbb F$ of characteristic zero is semisimple if and only if $H^*$ is semisimple if and only if $S^2=\id$ \cite{LR} if and only if $D(H)$ is semisimple \cite{R}.  In this case, $D(H)$ is  a ribbon Hopf algebra with ribbon element given by the inverse of the  Drinfeld element \cite{sgel}.  Note also that any finite-dimensional semisimple Hopf algebra $H$ is equipped with a (normalised)  Haar  integral and that   $\eta\oo \ell$ is a Haar integral for $D(H)$ if and only if  $\eta\in H^*$ and $\ell\in H$ are  Haar integrals of $H^*$ and $H$. 

\begin{definition} Let $H$ be a finite-dimensional Hopf algebra. A  {\bf Haar integral} is an element $\ell\in H$ with $h\cdot\ell=\ell\cdot h=\epsilon(h)\,\ell$ for all $h\in H$ and $\epsilon(\ell)=1$.
\end{definition}

\begin{remark} \label{rem:haar}
The Haar integral of a finite-dimensional semisimple Hopf algebra is unique. If $\ell\in H$ is a Haar integral, then  $S(\ell)=\ell$, the element  $\Delta^{(n)}(\ell)$ is invariant under cyclic permutations for all $n\in\mathbb N$ and  $e=\low\ell 1\oo S(\low \ell 2)$ is a separability idempotent for $H$. Moreover, for all $\alpha\in H^*$ one has $\langle \low\alpha 1,\ell\rangle\, \low\alpha 2=\langle\low\alpha 2,\ell\rangle\, \low\alpha 1=\langle \alpha,\ell\rangle \, 1$.
\end{remark}

A {\bf module} over a Hopf algebra $H$ is a module over the algebra $H$. Important examples that arise in the Kitaev models are the following.

\begin{definition}\label{def:regacts}

Let $H$ be a  Hopf algebra and $H^*$ its dual. \\[-3ex]
\begin{compactenum}

\item The {\bf left regular action}  $\rhd: H\oo H\to H$, $h\rhd k= h\cdot k$ 
 and
 the {\bf right regular action}  $\lhd: H\oo H\to H$, $k\lhd h=k\cdot h$ 
give $H$ the structure of an $H$-left and $H$-right module.\\[-2ex]

\item The {\bf left regular action}  $\rhd^*: H\oo H^*\to H^*$, $h\rhd\alpha=\langle\alpha_{(2)}, h\rangle \,\alpha_{(1)}$ 
 and
 the {\bf right regular action}  $\lhd^*: H^*\oo H\to H^*$, $\alpha\lhd h= \langle \alpha_{(1)}, h\rangle\, \alpha_2$ 
give $H^*$ the structure of an $H$-left and an $H$-right module.\\[-2ex]

\item The {\bf left adjoint action} $\rhd_{ad}: H\oo H\to H$, $h\rhd k= h_{(1)}\cdot k\cdot S(h_{(2)})$ and
the {\bf right adjoint action}  $\lhd_{ad}: H\oo H\to H$, $ k\lhd h= S^\inv(h_{(1)})\cdot k\cdot h_{(2)}$ 
give $H$ the structure of an $H$-left and an $H$-right module.\\[-2ex]

\item The {\bf left coadjoint action} $\rhd^*_{ad}: H\oo H^*\to H$, $\alpha\mapsto \langle S^\inv(\low\alpha 1)\low \alpha 3, h\rangle\, \low\alpha 2$ and the
 {\bf right coadjoint action} $\lhd_{ad}^*:H^*\oo H\to H^*$, $\alpha\mapsto \langle \low\alpha 1 S(\low\alpha 3), h\rangle\, \low\alpha 2$ give $H^*$ the structure of a $H$-left and an $H$-right module.
\end{compactenum}
\end{definition}

An {\bf invariant} of an $H$-module $M$  is  an element
 $m\in M$ with $h\rhd m=\epsilon(h)\, m$ for all $h\in H$. The invariants  of $M$ form a linear subspace $M_{inv}\subset M$. 

\begin{lemma}\label{lem:haarproj} Let $M$ with $\rhd: H\oo M\to M$ be a left  module over a Hopf  algebra $H$.  If $H$ is finite-dimensional semisimple with Haar integral $\ell \in H$, then $\Pi_M: M\to M$, $m\mapsto \ell\rhd m$
is a projector on  
$
M_{inv}=\{m\in M: h\rhd m=\epsilon(h)\,m\;\forall h\in H\}.
$
\end{lemma}

Specific examples  required in the  following
are the invariants for the left and right coadjoint action of a Hopf algebra $H$ on its dual  $H^*$ from Definition \ref{def:regacts}. If $H$ is finite-dimensional semisimple, then the invariants of the two coincide and can be viewed as the Hopf algebra analogue of the character algebra of a finite group.

\begin{example}\label{ex:kad} Let $H$ be a finite-dimensional semisimple Hopf algebra with dual $H^*$. Then the  invariants of the left and right coadjoint action $\rhd^*_{ad}$ and  $\lhd^*_{ad}$  are given by the {\bf character algebra}
$C(H)=\{\alpha\in H^*:\Delta(\alpha)=\Delta^{op}(\alpha)\}$. The  map $\pi_{ad}: H^*\to H^*$, $\alpha\mapsto \alpha\lhd^*_{ad}\ell$ is a projector on $C(H)$.
\end{example}

Note that for {\em factorisable} Hopf algebras $H$ the Drinfeld map defines an algebra isomorphism between the character algebra $C(H)$ and centre $Z(H)$.
In particular, this applies to the case where $H$ is the  Drinfeld double of a finite-dimensional semisimple Hopf algebra.

If $M$ is not only a module  over a Hopf algebra $H$ but  also an associative algebra  such that the module structure is  compatible with the multiplication,
 then $M$ is called a {\em module algebra} over $H$. In other words,  a left (right) module algebra  over a  Hopf algebra $H$  is  an algebra object in the category $H$-Mod (Mod-$H$) of left (right) modules over $H$.

\begin{definition}\label{def:modalg} Let $H$ be a Hopf algebra over $\mathbb F$.

\begin{compactenum}
\item  An {\bf $H$-left module algebra}  is an associative, unital algebra $A$ over $\mathbb F$ with an $H$-left module structure  $\rhd: H\oo A\to A$, $h\oo a\mapsto h\rhd a$  such that  
$h\rhd (a\cdot a')=(h_{(1)}\rhd a)\cdot (h_{(2)}\rhd a')$ and $h\rhd 1_A=\epsilon(h)\, 1$ for all $h\in H$, $a,a'\in A$.\\[-2ex]
\item If $A,A'$ are $H$-left module  algebras, then  $f: A\to A'$ is called a {\bf morphism of $H$-left module algebras} if it is both, an algebra morphism  and a morphism of $H$-left modules.
\end{compactenum}
\end{definition}

An {\bf $H$-right module algebra} is  an  $H^{op}$-left module algebra, and  
an {\bf $(H,K)$-bimodule algebra} is a $(H\oo K^{op})$-left module algebra. Morphisms of $H$-right and $(H,K)$-bimodule algebras are defined correspondingly. 
Important examples arise from the left and right adjoint action of $H$ on itself and the  left and right action of $H$ on  $H^*$ in Definition \ref{def:regacts}.

\begin{example}\label{ex:regacts2}
Let $H$ be a  Hopf algebra with dual $H^*$. Then the left and right  regular action of $H$ on $H^*$ give $H^*$ the structure of an $(H,H)$-bimodule algebra. The  left and right adjoint action of $H$ on itself 
give $H$ the structure of an $H$-left and of an $H$-right module algebra.\end{example}

An important motivation for the role of module algebras in Hopf algebra gauge theory is  the fact that their invariants form  a  subalgebra. 

\begin{lemma}\label{lem:project} Let $H$ be a  Hopf algebra and  $A$ an $H$- left module algebra. Then the linear subspace $A_{inv}=\{a\in A: h\rhd a=\epsilon(h)\,\forall h\in H\}$ is a subalgebra of $A$. 
\end{lemma}

Another important feature of a module algebra $A$ over a  Hopf algebra $H$ is that it induces an algebra structure on the vector space $A\oo H$,  the so-called smash or cross product, which  can be viewed as the Hopf algebra analogue of a semidirect product of groups. 

\begin{definition} \label{def:hdouble} Let $H$ be a 
Hopf algebra, $A$ an $H$-left module algebra and $B$ an $H$-right module algebra.
The  {\bf left cross product}  or {\bf left smash product} $A\#_L H$   is the algebra   $ (A\oo H,\cdot)$ with  
\begin{align}
\label{eq:cross_lefthd}
&(a\oo h)\cdot (a'\oo h')=a(h_{(1)}\rhd a')\oo h_{(2)}h'.
\end{align}
The {\bf right cross product} or {\bf right smash product}  $H\#_R B$  
is the algebra  $(H\oo B,\cdot)$ 
with
\begin{align}\label{eq:cross_rhd}
&(h\oo b)\cdot (h'\oo b')=h\low {h'} 1\oo (b\lhd \low {h'}2)b'.
\end{align}
\end{definition}

\begin{example} The left and right {\bf Heisenberg double}  of $H$  are the cross products $\mathcal H_L(H)=H^*\#_LH$ and  $\mathcal H_R(H)=H \#_R H^*$  for  the left and right regular action of $H$ on $H^*$: 
\begin{align}
&\mathcal H_L(H): \qquad(\alpha\oo h)\cdot (\alpha'\oo h')=\langle\alpha'_{(2)}, h_{(1)}\rangle\; \alpha \alpha'_{(1)}\oo h_{(2)} h'\\
&\mathcal H_R(H): \qquad( h\oo \alpha)\cdot (h'\oo \alpha')=\langle \alpha_{(1)}, h'_{(2)}\rangle\; hh'_{(1)}\oo \alpha_{(2)}\alpha'.\label{eq:hd2}
\end{align}
\end{example}

In the following we focus on the Heisenberg doubles for the right regular action of $H$ on its dual and denote it $\mathcal H(H)=\mathcal H_R(H)$. Some structural results on  the Heisenberg double $\mathcal H(H)$ that are required throughout the article are the following.

\begin{lemma} \label{lem:antipodehd}Let $H$ be a finite-dimensional semisimple Hopf algebra with dual $H^*$ and Heisenberg double $\mathcal H(H)$.  Denote by $S_D: H\oo H^*\to H\oo H^*$ the antipode  and 
 by $\Delta_D: H\oo H^*\to H\oo H^*\oo H\oo H^*$ the comultiplication of $D(H)^*$ from \eqref{eq:ddualmult} and by $S$ the antipodes of $H$ and $H^*$. Then: 
\begin{compactenum}
\item   $S_D: \mathcal H(H)\to\mathcal H(H)$  is an algebra automorphism, and for all $y,z\in H$ and $\gamma,\delta\in H^*$ one has
\begin{align}\label{eq:hdcommutelr}
S_D(y\oo 1)\cdot (z\oo 1)=(z\oo 1)\cdot S_D(y\oo 1)\qquad S_D(1\oo\gamma)\cdot (1\oo\delta)=(1\oo\delta)\cdot S_D(1\oo\gamma).
\end{align}
\item The following linear maps  are injective algebra morphisms from $\mathcal H(H)$ to $\mathcal H(H)\oo \mathcal H(H)$
\begin{align*}
&\phi_1=(\id\oo\epsilon)\oo (\id\oo\id))\circ\Delta_D:  y\oo\gamma\mapsto (\low y 1\oo 1)\oo (\low y 2\oo\gamma)\\
&\phi_2=((\id\oo\id)\oo(\epsilon\oo\id))\circ\Delta_D: y\oo\gamma\mapsto (y\oo\low\gamma 1)\oo (1\oo\low\gamma 2)\\
&\xi_1=((\id\oo\id)\oo (S\oo\epsilon))\circ\Delta_D:   y\oo\gamma\mapsto \Sigma_{i,j} \,(\low y 1\oo\alpha^i\gamma \alpha^j)\oo (S(x_i)S(\low y 2)x_j\oo 1)\\
&\xi_2=((\epsilon\oo S)\oo(\id\oo\id))\circ\Delta_D:   y\oo\gamma\mapsto \Sigma_{i,j}\,(1\oo \alpha^iS(\low\gamma 1)S(\alpha^j))\oo (x_iy x_j\oo\low\gamma 2).
\end{align*}
They satisfy the relations
\begin{align}\label{eq:hdmapids}
&(\phi_1\oo \id)\circ \phi_1=(\id\oo\phi_1)\circ \phi_1 & &(\phi_2\oo \id)\circ \phi_2=(\id\oo\phi_2)\circ \phi_2\\
&(\id\oo\phi_2)\circ \phi_1=(\phi_1\oo\id)\circ\phi_2 & &(\id\oo \xi_1)\circ\xi_2=(\xi_2\oo \id)\circ \xi_1\nonumber.
\end{align}
\end{compactenum}
\end{lemma}

\begin{proof} 
The claims  follow by direct but lengthy computations from formula \eqref{eq:hd2} for the multiplication of the Heisenberg double, the formulas  for the antipode and comultiplication of $D(H)^*$ in \eqref{eq:ddualmult} and the identity $S^2=\id$.
Their proof also makes use of the following auxiliary identities. If
  $\{x_i\}$ is a basis of $H$ and $\{\alpha^i\}$  the dual basis of $H^*$ then
$h=\Sigma_i\,\langle \alpha^i,h\rangle x_i$ for  all $h\in H$  and $\beta=\Sigma_i \langle\beta, x_i\rangle\alpha^i$ for
all $\beta\in H^*$. This implies 
\begin{align}\label{eq:basexx}
&\Sigma_i \; \Delta(\alpha^i)\oo x_i=\Sigma_i \alpha^i_{(1)}\oo\alpha^i_{(2)}\oo x_i=\Sigma_{i,j} \;\alpha^i\oo\alpha^j\oo x_ix_j \\
&\Sigma_i\;\alpha^i\oo\Delta(x_i)=\Sigma_i \alpha^i\oo x_{i(1)}\oo x_{i(2)}=\Sigma_{ij} \;\alpha^i\alpha^j\oo x_i\oo x_j.\nonumber
\end{align}
If $H$  is semisimple, then $H^*$ is semisimple as well and the identity $S^2=\id$ for these two Hopf algebras together with  the expression for the universal $R$-matrix in \eqref{eq:dmult} imply
\begin{align}\label{eq:helpanti}
\Sigma_i \,x_i\oo S(\alpha^i)=\Sigma_i \,S(x_i)\oo\alpha^i\qquad \Sigma_{i} \,S(x_i)\oo S(\alpha^i)=\Sigma_i\,x_i\oo\alpha^i.
\end{align}
By combining these identities with \eqref{eq:basexx} one obtains 
\begin{align}\label{eq:helpanti2}
&\Sigma_{i,j}\, x_ix_j\oo S(\alpha^i)\alpha^j=\Sigma_{i,j}\, S(x_i)x_j\oo \alpha^i\alpha^j=\Sigma_{i,j}\, x_ix_j\oo \alpha^iS(\alpha^j)=\Sigma_{i,j}\, x_iS(x_j)\oo \alpha^i\alpha^j=1.
\end{align}
Inserting these identities into the formulas for the maps $S_D, \phi_1, \phi_2,\xi_1,\xi_2$ together with \eqref{eq:basexx} and \eqref{eq:helpanti},  \eqref{eq:ddualmult} and the identity $S^2=\id$ for the antipodes of $H$ and $H^*$ proves the claim.
\end{proof}

\subsection{Ribbon graphs}
\label{subsec:graph}

In the following, we consider finite directed graphs $\Gamma$. We denote by $V(\Gamma)$  and $E(\Gamma)$, respectively,  their sets of vertices and edges and omit the argument $\Gamma$ when this is unambiguous.  We denote by $\st(e)$ the {\bf starting vertex} and  by $\ta(e)$ the {\bf target vertex} of an oriented edge $e$ and call $e$ 
 a {\bf loop} if  $\st(e)=\ta(e)$.  The reversed edge  is denoted $e^\inv$, and one has $\st(e^\inv)=\ta(e)$, $\ta(e^\inv)=\st(e)$.

\begin{definition}\label{def:vertex_nb} Let $\Gamma$ be a directed graph.
\begin{compactenum}

\item The {\bf vertex neighbourhood} $\Gamma_v$ of a vertex $v$ of $\Gamma$ is the directed graph obtained by subdividing each edge of $\Gamma$ and   deleting all edges from the resulting graph  that are not incident at $v$, as shown in Figure \ref{fig:cutting}. \\[-2ex]

\item For an edge $e$ of $\Gamma$ the associated edges  $s(e)\in\Gamma_{\st(e)}$ and $t(e)\in\Gamma_{\ta(e)}$ are called the {\bf starting end} and  {\bf target end} of $e$. 
They satisfy $t(e^\inv)=s(e)^\inv$ and $s(e^\inv)=t(e)^\inv$.
 \end{compactenum}
\end{definition}

\begin{figure}
\centering
\includegraphics[scale=0.35]{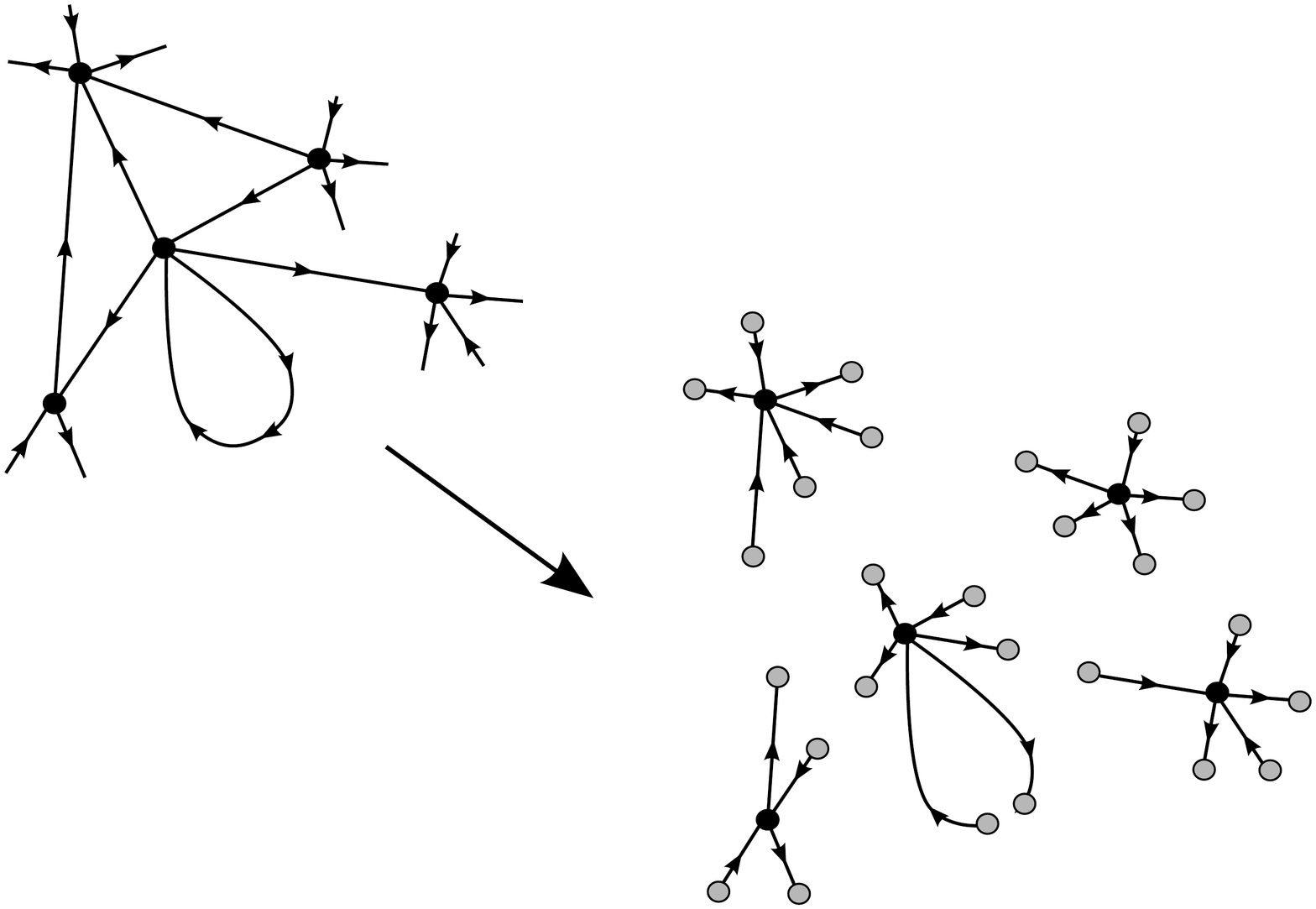}

\vspace{-.5cm}
\caption{Splitting a directed graph  $\Gamma$ into a disjoint union of vertex neighbourhoods.}
\label{fig:cutting}
\end{figure}

Paths in a directed graph $\Gamma$ are morphisms in the free groupoid generated by $\Gamma$.
They are described by {\bf words} $w=e_n^{\epsilon_n}\circ\ldots\circ e_1^{\epsilon_1}$
with $n\in\mathbb N$, $e_i\in E$,  $\epsilon_i\in\{\pm 1\}$  or empty words $\o_v$ for each vertex $v\in V$. 
A word $w$  is called {\bf composable} if  it is empty or if  $\ta(e_i^{\epsilon_i})=\st(e_{i+1}^{\epsilon_{i+1}})$ for all $i=1,...,n-1$. In this case we set $\st(w)=\st(e_1^{\epsilon_1})$ and $\ta(w)=\ta(e_n^{\epsilon_n})$ and $\st(w)=\ta(w)=v$ if $w=\o_v$.
The number $n\in\mathbb N$ is called the {\bf length} of $w$. 
A word $w$ is called {\bf reduced} if it is empty or of the form $w=e_n^{\epsilon_n}\circ\ldots\circ e_1^{\epsilon_1}$
 with $e_i^{\epsilon_i}\neq e_{i+1}^{-\epsilon_{i+1}}$ for all $i\in\{1,...,n-1\}$.

 \begin{definition} \label{def:free_cat} 
 Let $\Gamma$ be a directed graph. 
The  {\bf path groupoid} $\mathcal G(\Gamma)$   is the free groupoid generated by $\Gamma$. Its objects are the  vertices of $\Gamma$.  A morphism
 from  $u$ to $v$  is an equivalence class of  composable words $w$ with
 $\st(w)=u$ and  $\ta(w)=v$  with respect to 
  $e^{-1}\circ e\sim \o_{\st(e)}$, $e\circ e^\inv\sim\o_{\ta(e)}$ for all edges $e$ of $\Gamma$. Identity morphisms  are  equivalence classes of  trivial words $\o_v$,  and the composition of morphisms is induced by the concatenation.  
A {\bf path}  in $\Gamma$ is a morphism in $\mathcal G(\Gamma)$.
  \end{definition}

In the following, we consider  directed graphs with additional structure, called {\em ribbon graphs}, {\em fat graphs} or  {\em embedded graphs} (for background see \cite{LZ, EM}). These are directed graphs with a {\em cyclic ordering} of the incident edge ends at each vertex, i.~e.~an ordering up to cyclic permutations.    

This cyclic ordering equips the graph with the notion of a {\em face}.  A path  $p\in \mathcal G(\Gamma)$ given by a reduced word $e_n^{\epsilon_n}\circ \ldots\circ  e_1^{\epsilon_1}$ is said to  turn maximally right (left) at the vertex $v_i=\st(e_{i+1}^{\epsilon_{i+1}})=\ta(e_i^{\epsilon_i})$ if the starting end of  $e_{i+1}^{\epsilon_{i+1}}$ comes directly after (before) the target end of $e_i^{\epsilon_i}$ with respect to the cyclic ordering at $v_i$.
 If  $\st(p)=\ta(p)=v_n$,  $p$  is said to turn maximally right (left) at $v_n=\st(e_1^{\epsilon_1})=\ta(e_n^{\epsilon_n})$ if the starting end of $e_1^{\epsilon_1}$ comes directly after (before) the target end of $e_n^{\epsilon_n}$ with respect to the cyclic ordering at $v_n$.  Faces of $\Gamma$  are equivalence classes of closed paths that turn maximally right\footnote{Note that a different convention is used in \cite{MW}, where a face is required to turn maximally left at each vertex. However, the conventions in this article give a better match with Kitaev models.} 
at each vertex and pass any edge at most once in each direction, up to cyclic permutations. The set of faces of $\Gamma$ is denoted $F(\Gamma)$

In the following, we also consider ribbon graphs  in which some or all vertices  are equipped with a
 {\em linear} ordering of the incident edge ends. In this case, we always require that the  cyclic ordering induced by the linear ordering is the one from the ribbon graph structure.  Any such linear ordering of the incident edge ends at a vertex is obtained from their cyclic ordering of the ribbon graph  by selecting one of the incident edge ends  as the edge end of minimal order. 
 This is indicated in  pictures by placing a marking, called {\bf cilium}  at the vertex and ordering the incident edges at the vertex  counterclockwise from the cilium in the plane of the drawing, as shown in Figure \ref{fig:vertex_edgeends}. For a vertex  $v$ with a linear odering of the incident edge ends, we 
write  $e<f$  if $e,f$ are edge ends incident at  $v$ and $e$  is of lower order than $f$.

\begin{definition} Let $\Gamma$ be a ribbon graph.
\begin{compactenum}

\item A {\bf ciliated vertex} of  $\Gamma$ is a vertex $v$  with a linear ordering of the incident edge ends that induces their cyclic ordering from the ribbon graph structure. Two edge ends $e,f$ at a ciliated vertex $v$ are called {\bf adjacent}  if there is no edge end $g$  at $v$ with $e<g<f$ or $f<g<e$.  The {\bf valence} $|v|$ of $v$ is the number of incident edge ends.
\\[-2ex]

\item A {\bf ciliated ribbon graph} is a ribbon graph in which all vertices are ciliated. \\[-2ex]

\item  A {\bf ciliated face}  of $\Gamma$ is a closed path which turns maximally right at each vertex, including the starting vertex,  and traverses each edge at most once in each direction. The  {\bf valence} $|f|$ of $f$  is  its length  as a  reduced word in $E(\Gamma)$. \\[-2ex]

\item Two ciliated faces $f,f'$  of $\Gamma$ are  {\bf equivalent} if they induce the same face, i.~e.~their expressions as reduced words in the edges of $\Gamma$ are related  by  cyclic permutations. 
\end{compactenum}
\end{definition}

This terminology is compatible with Poincar\'e duality.  By passing from a ribbon graph $\Gamma$ to its Poincar\'e dual $\bar\Gamma$, we can  interpret a ciliated face of $\Gamma$ as a ciliated vertex of  $\bar\Gamma$ and a ciliated vertex of $\Gamma$ as a ciliated face of $\bar\Gamma$. 
Vertices and faces of $\Gamma$ with the associated {\em cyclic} orderings are given as equivalence classes of ciliated vertices and faces under cyclic permutations and hence correspond to, respectively,  faces and vertices  of $\bar\Gamma$.
In the following, we sometimes require that the ciliated ribbon graphs $\Gamma$ satisfy certain regularity conditions.

\begin{figure}
\centering
\includegraphics[scale=0.17]{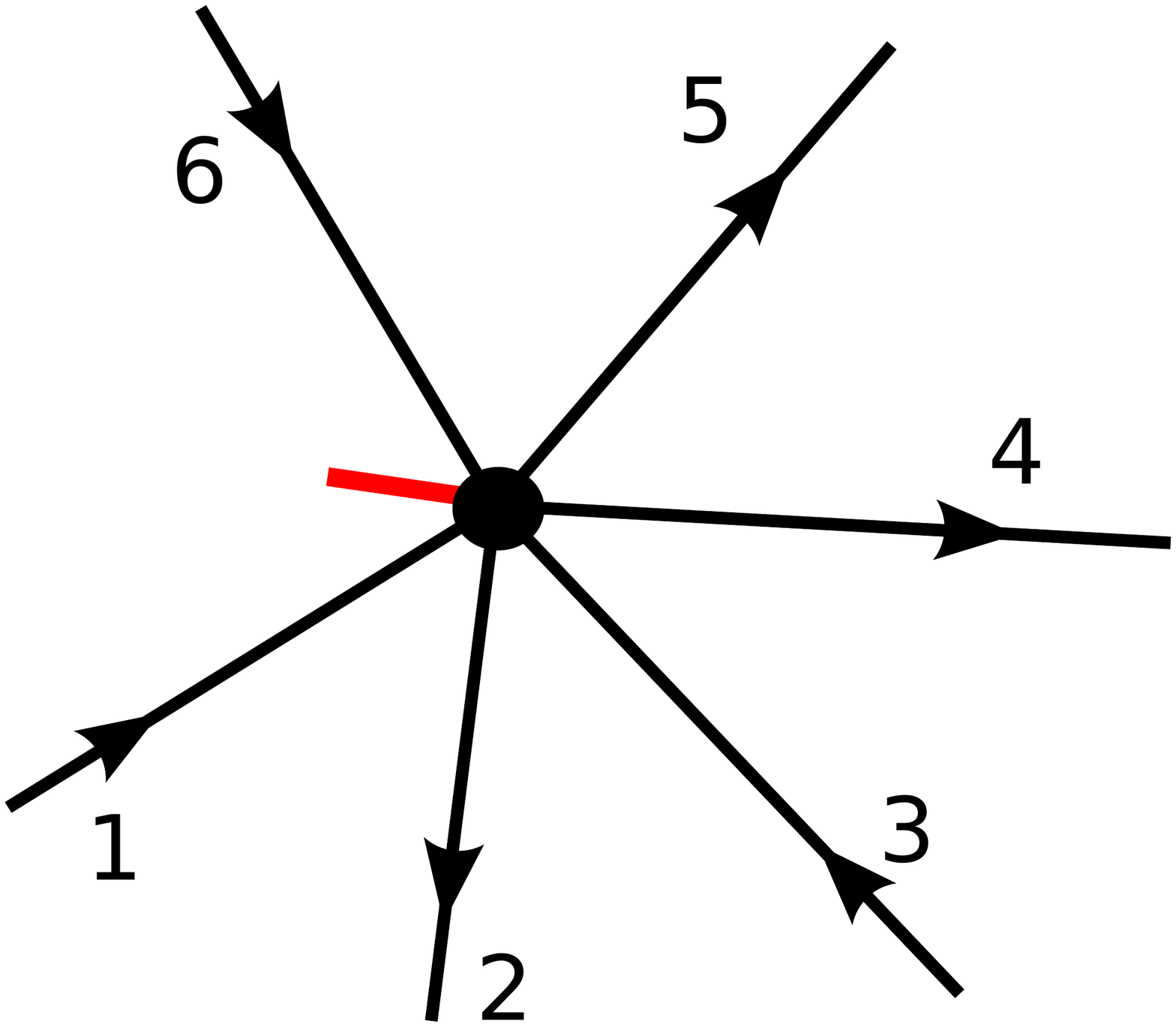}
\caption{Vertex with  incident edge ends and the  ordering induced by the  cilium. }
\label{fig:vertex_edgeends}
\end{figure}

\begin{definition}\label{def:regular} A ciliated ribbon graph $\Gamma$ is called {\bf regular} if:
\begin{compactenum}
\item $\Gamma$ has no loops or multiple edges.
\item Each  face of $\Gamma$ traverses each edge  at most once.
\item Each  face  of $\Gamma$ contains exactly one cilium.
\end{compactenum}
\end{definition}

\begin{figure}
\centering
\includegraphics[scale=0.36]{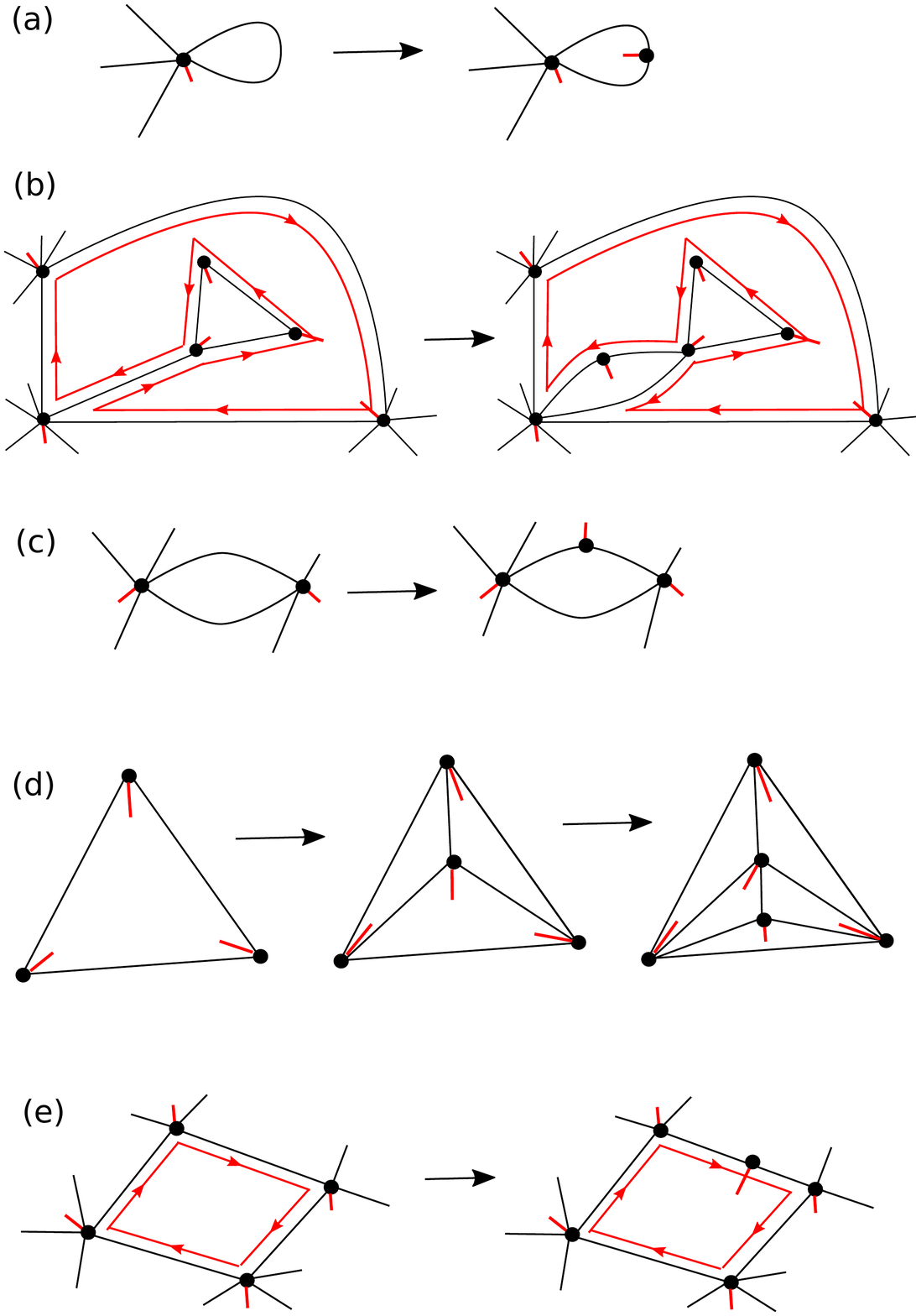}
\caption{Regularising   a ciliated ribbon graph: (a) Subdividing a loop, (b) doubling an edge, (c) subdividing  multiple edges, (d) subdividing a face, (e) adding a ciliated vertex to a face.}
\label{fig:subdiv}
\end{figure}

Note that for a regular  ciliated ribbon graph $\Gamma$, one has $|V(\Gamma)|=|F(\Gamma)|$.
Each cilium corresponds to a pair $(v,f)$ of a ciliated vertex $v\in V(\Gamma)$ and a  ciliated face $f$ based at the cilium. 
Such pairs $(v,f)$ are called {\em sites} in the context of Kitaev models. 

The regularity conditions  in Definition \ref{def:regular} are mild because any ciliated ribbon graph can be transformed into a regular ciliated ribbon graph by subdividing edges, doubling edges and subdividing faces. 
More precisely, 
for any ciliated ribbon graph $\Gamma$ one can construct a regular ciliated ribbon graph $\Gamma'$  by the following procedure,   illustrated in Figure \ref{fig:subdiv}:
\begin{compactenum}[(a)]
\item Subdivide each loop  by adding a  bivalent ciliated vertex whose cilium points inside the loop, i.~e.~such that that the path along the loop that starts and ends at this new vertex and turns maximally right at each vertex  becomes a ciliated face, as shown in Figure \ref{fig:subdiv} (a).\\[-2ex]
\item Double each edge  that is traversed twice by a  face and  add a ciliated bivalent  vertex whose cilium points into the resulting face, as shown in Figure \ref{fig:subdiv} (b).\\[-2ex]
\item For each pair of edges $e,e'$ with $\st(e)=\st(e')$ and  $\ta(e)=\ta(e')$, subdivide either $e$ or $e'$ by adding  a bivalent ciliated vertex, as shown in Figure \ref{fig:subdiv} (c).\\[-2ex]

\item Subdivide each face  that contains more than one cilium by adding a vertex  and connecting it to the vertices of the face in such a way that each of the resulting faces contains at most one cilium. Equip the new vertex with a cilium and repeat if necessary,  as shown in Figure \ref{fig:subdiv} (d).\\[-2ex]

\item  For each face  that contains no cilia, add a bivalent ciliated vertex to one of its edges such that the cilium  points into  the face, as shown in Figure \ref{fig:subdiv} (e).
\end{compactenum}

Ribbon graphs can be viewed as graphs embedded into oriented surfaces. A graph embedded into an oriented surface inherits a cyclic ordering of the incident edge ends at each vertex from the orientation of the surface. Conversely, each ribbon graph defines oriented surface which is unique up to homeomorphisms and obtained as follows. Given a graph  $\Gamma$, understood as a combinatorial graph, one obtains a graph in the topological sense,  a 1-dimensional CW-complex, by gluing intervals to the vertices as specified by the edges. 
If additionally $\Gamma$ has a ribbon graph structure, one obtains an oriented surface $\Sigma_\Gamma$
 by  gluing a disc to each face of $\Gamma$.
  If $\Gamma$ is a  graph embedded in an oriented surface $\Sigma$ and equipped with the induced ribbon graph structure, then the surface $\Sigma_\Gamma$  is homeomorphic to $\Sigma$ if and only if  each connected component of $\Sigma\setminus\Gamma$ is homeomorphic to a disc.

Ribbon graphs $\Gamma$ and $\Gamma'$ that are related by certain graph operations 
 define homeomorphic surfaces $\Sigma_\Gamma$ and $\Sigma_{\Gamma'}$. These graph operations include 
 edge contractions, edge subdivisions,  subdivisions of faces,  doubling edges and adding or removing edges from the graph if this does not change the number of connected components  \cite{LZ, EM, MW}. This implies in particular that for each ciliated ribbon graph $\Gamma$ there is a ciliated ribbon graph $\Gamma'$ that is regular in the sense  of Definition \ref{def:regular} and 
  such that $\Sigma_\Gamma$ and $\Sigma_{\Gamma'}$ are homeomorphic.

\section{Kitaev models}
\label{subsec:kit}

Kitaev models were first introduced in \cite{Ki}. They were then generalised to models based on the group algebra of a finite group and with defects and domain walls in \cite{BMD} and  to  finite-dimensional semisimple Hopf algebras in \cite{BMCA}. More recent generalisations include models based on certain certain tensor categories and with defect data from higher categories \cite{KK}. 
In this article we focus on the models from \cite{BMCA}  for a  finite-dimensional semisimple\footnote{Note that the conditions on the Hopf algebra in \cite{BMCA}  are slightly stronger, as they set $\mathbb F=\mathbb C$ and  require that $H$ is a  $*$-Hopf algebra. This is needed in their definition of the scalar product on $H^{\oo E}$ and to ensure unitarity and self-adjointness of certain operators on $H^{\oo E}$. However, as we do not investigate these structures, it is sufficient for our purposes that $H$ is finite-dimensional and semisimple.} Hopf algebra $H$.

The two ingredients of a Kitaev model are a finite-dimensional semisimple Hopf algebra $H$ and a ribbon graph $\Gamma$.
The starting point in the construction is the {\em extended 
 space} $H^{\oo E}$ obtained by associating a copy of $H$ to each edge of  $\Gamma$. One then assigns to each edge $e\in E$ four  basic triangle operators $L^h_{e\pm}: H\to H$ and $T^\alpha_{e\pm}: H\to H$, indexed by  elements  $h\in H$ and  $\alpha\in H^*$. 
With the notation and conventions from Section \ref{subsec:conventions} they are defined as follows.

\begin{definition} [\cite{Ki, BMCA}]  \label{def:kitdef} Let $H$ be a finite-dimensional semisimple Hopf algebra and $\Gamma$ a ribbon graph.
The {\bf triangle operators} for an edge $e$ of $\Gamma$, $h\in H$ and $\alpha\in H^*$ are the linear maps
\begin{align*}
L^h_{e\pm}=(L^h_\pm)_e: H^{\oo E}\to H^{\oo E} \qquad\qquad T^\alpha_{e\pm}=( T^\alpha_{\pm})_e: H^{\oo E}\to H^{\oo E},
\end{align*}
with $L^h_{\pm}, T^{\alpha}_{\pm}: H\to H$  given by
\begin{align}\label{eq:kitops}
&L_+^{h}k= h\cdot k & &L_-^{h}k=(S\circ L_+^{S(h)}\circ S)k=k \cdot h\\
&T^{\alpha}_+ k=\langle \alpha,\low k 2\rangle\, \low k 1 & &T^{ \alpha}_-k=(S\circ T^{S(\alpha)}_+\circ S)k=\langle \alpha, \low k 1\rangle\, \low k 2.\nonumber
\end{align}

\end{definition}

\medskip
By combining the triangle operators of the edges at each vertex  $v$ and  in each face  $f$ of $\Gamma$, one obtains  the {\em vertex and face operators}  $A_v^h: H^{\oo E}\to H^{\oo E}$  and $B^\alpha_f: H^{\oo E}\to H^{\oo E}$. Their definition requires a linear ordering of the incident edges at each vertex and a  in each face, i.~e.~{\em ciliated} vertices and  {\em ciliated} faces. They are defined for general ribbon graphs $\Gamma$, but in the following we restrict attention to ribbon graphs without loops or multiple edges.

\begin{definition}[\cite{Ki, BMCA} ] \label{def:vertexoperator} Let $\Gamma$ be a  ribbon graph without loops or multiple edges.
\begin{compactenum}
\item Let  $v$  be a ciliated vertex of $\Gamma$  with  incident edges $e_1,...,e_n$, numbered according to the ordering at $v$ and such that $e_1^{\epsilon_1}$,..., $e_n^{\epsilon_n}$ are incoming.  The 
  {\bf vertex operator} $A^h_v:H^{\oo E}\to H^{\oo E}$ for  $h\in H$  is the linear map 
$$A^h_v=L^{S^{\tau_1}(\low h 1)}_{e_1\epsilon_1}\circ \ldots\circ L^{S^{\tau_n}(\low h n)}_{e_n\epsilon_n}: H^{\oo E}\to H^{\oo E}\qquad\text{with}\;\;\; \tau_i=\tfrac 1 2(1-\epsilon_i).$$
\item 
Let $f=e_1^{\epsilon_1}\circ \ldots\circ e_n^{\epsilon_n}$  be a ciliated face of $\Gamma$. The {\bf face operator} $B^\alpha_f:H^{\oo E}\to H^{\oo E}$ for   $\alpha\in H^*$ is the linear map 
$$B^\alpha_f=T^{S^{\tau_1}(\low \alpha 1)}_{e_1\epsilon_1}\circ\ldots\circ T^{ S^{\tau_n}(\low\alpha n)}_{e_n\epsilon_n}: H^{\oo E}\to H^{\oo E}\qquad\text{with}\;\;\; \tau_i=\tfrac 1 2(1-\epsilon_i).$$
\end{compactenum}
\end{definition}

Choosing a cilium at a vertex $v$ does not only equip  $v$ with the structure of a ciliated vertex but at the same time selects  a ciliated face of $\Gamma$, namely the unique ciliated face that starts and ends at the cilium at $v$. It was shown   
 in \cite{BMCA}  that the associated vertex and face operators define a representation of the Drinfeld double $D(H)$. This follows by a direct computation from the definition of the vertex and face operators and equation \eqref{eq:kitops} for the  triangle operators and is proven in Section \ref{subsec:vertface} in a different formalism.

\begin{lemma} [\cite{Ki,BMCA}] \label{lem:algrelskit} Let $v$ be a ciliated vertex  of $\Gamma$ and $f(v)$ the ciliated face  of $\Gamma$ that starts and ends at the cilium at $v$. The associated vertex and face operators satisfy the commutation relations 
\begin{align*}
&A^h_v \circ A^k_v=A_v^{hk}\qquad B^\alpha_{f(v)}\circ B^\beta_{f(v)}=B^{\alpha\beta}_{f(v)}\qquad A^h_v\circ B^\alpha_{f(v)}=\langle \low\alpha 3, \low h 1\rangle\langle \low\alpha 1, S(\low h 3)\rangle\; B^{\low\alpha 2}_{f(v)}\circ A^{\low h 2}_v.
\end{align*}
The map $\tau: D(H)\to \mathrm{End}_{\mathbb F}(H)$, $\alpha\oo h\mapsto B^\alpha_{f(v)} \circ A^h_v$ is an injective algebra homomorphism.
\end{lemma}

Similarly one can show that for any choice of the cilia the vertex operators for different vertices $v,w$ and the  faces operators for different faces $f,g$ of $\Gamma$ commute and that the vertex operators commute with all face operators that satisfy a certain  condition on the cilia.

\begin{lemma} [\cite{Ki,BMCA}] \label{lem:fvcomm} Let $\Gamma$ be a ciliated ribbon graph, $h,k\in H$ and $\alpha,\beta\in H^*$.
\begin{compactenum}
\item For all choices of the ciliation, one has  $A_v^h\circ A_w^k=A^k_w\circ A^h_v$ if $v\neq w$.
\item For all choices of the ciliation, one has $B_f^\alpha\circ B_g^\beta=B_g^\beta\circ B^\alpha_f$ if $f\neq g$.
\item If $v$ is a ciliated vertex  and $f$ a ciliated face  that  is not based at $v$ and does not traverse the cilium at $v$, then  $A^h_v\circ B^\alpha_f=B^\alpha_f\circ A^h_v$.
\end{compactenum}
\end{lemma}

If one chooses for $h\in H$ the Haar integral $\ell \in H$ and for $\alpha\in H^*$ the Haar integral $\eta\in H^*$, the associated vertex and face operators 
$A^\ell_v, B^\eta_f: H^{\oo E}\to H^{\oo E}$ no longer depend on the choice of the cilia at $v$ and at $f$. This is a direct consequence of the properties of the Haar integral in Remark \ref{rem:haar}.  The properties of the Haar integral also imply that they are projectors with images
$$
A^\ell_v(H^{\oo E})=\{x\in H^{\oo E}:\, A^h_v x=\epsilon(h) x\;\forall h\in H\}\quad B^\eta_f(H^{\oo E})=\{x\in H^{\oo E}:\, B^\alpha_f x=\epsilon(\alpha) x\;\forall \alpha\in H^*\}
$$ 
and that their  commutation relations from Lemma \ref{lem:algrelskit} simplify to the following.

\begin{lemma} [\cite{Ki,BMCA}] \label{lem:fvhaar} Denote by $\ell\in H$, $\eta\in H^*$ the Haar integrals of $H$ and $H^*$. Then the vertex and face operators $A^\ell_v$, $B^\eta_f$ do not depend on the cilia and form a set of commuting projectors.\end{lemma}

\begin{definition} [\cite{Ki,BMCA}] \label{def:hamiltonian}  Let $\Gamma$ be a ribbon graph, $H$  a finite-dimensional semisimple Hopf algebra and $\ell \in H$, $\eta\in H^*$  the Haar integrals of $H$ and $H^*$.  The {\bf Hamiltonian} of the Kitaev model on $\Gamma$ is 
$$
H_{\mathrm{K}}=\Pi_{v\in V} A^\ell_v\circ \Pi_{f\in F} B^\eta_f: H^{\oo E}\to H^{\oo E}.
$$
Its image is  the {\bf protected space} 
\begin{align*}
\prot&=\ham(H^{\oo E})=\{h\in H^{\oo E}: \ham(h)=h\}.
\end{align*}
\end{definition}

It was shown in \cite{Ki,BMCA} that the protected space is a topological invariant. It depends only  on the homeomorphism class of  the  oriented surface $\Sigma_\Gamma$ obtained by gluing discs to the faces of $\Gamma$. 
Topological {\em excitations} in Kitaev models are obtained by removing the vertex and face operators $A^\ell_v$ and $B^\ell_f$ from the Hamiltonian $\ham$ for a fixed number of sites $(v,f)$ which have no vertices or faces in common.  This yields another projector $\ham'$ whose ground state $\ham'(H^{\oo E})=\{h\in H^{\oo E}: \ham'(h)=h\}$ describes a model with excitations labelled by representations of the Drinfeld double $D(H)$.

While Kitaev models are mostly formulated in terms of operators acting on Hilbert spaces, 
we  take  a more  algebraic viewpoint  and focus on algebra structures and not on specific representations. 
For this we note that the triangle operators $L^y_{e\pm}, T^\alpha_{e\pm}: H^{\oo E}\to H^{\oo E}$ for an edge $e$ of $\Gamma$ form a faithful representation of the  Heisenberg double $\mathcal H(H)$ from Definition \ref{def:hdouble}. This allows one to identify Kitaev's edge operator algebra with the $E$-fold tensor product $\mathcal H(H)^{\oo E}$.

\begin{lemma} \label{lem:kitoprel}  Let $H$  be a  finite-dimensional semisimple Hopf algebra with dual $H^*$. Denote by $S$  the antipodes of $H$ and $H^*$ and by  $S_D: H\oo H^*\to H\oo H^*$  the antipode of $D(H)^*$ from \eqref{eq:ddualmult}.
\begin{compactenum}
\item  The linear isomorphism $\phi: \mathcal H(H)\to \mathrm{End}_{\mathbb F}(H)$, $\phi(y\oo \alpha) h=(L^y_+\circ T^\alpha_+) h= \langle \alpha,\low h 2\rangle\, y\cdot \low h 1$ 
defines a cyclic $\mathcal H(H)$-left module structure on $H$.\\[-2ex]
\item The  triangle operators $L^y_\pm, T^\alpha_\pm: H\to H$ from \eqref{eq:kitops} are given by
\begin{align*}
&L^y_+=\phi(y\oo 1) & &T^\alpha_+=\phi(1\oo\alpha) & &L^y_-= \phi\circ S_D (S(y)\oo 1) & &T^\alpha_-= \phi\circ S_D(1\oo S(\alpha)).
\end{align*}
\item The linear map $\rho=\phi^{\oo E}: \mathcal H(H)^{\oo E}\to \mathrm{End}_{\mathbb F}(H^{\oo E})\cong \mathrm{End}_{\mathbb F}(H)^{\oo E}$
is an algebra isomorphism.
\end{compactenum}
\end{lemma}
\begin{proof}
That  $\phi$ defines an $\mathcal H(H)$-left module structure on $H$ follows by a direct computation from the properties of the left and right regular action of $H$ on itself and on $H^*$ in Definition \ref{def:regacts} and the multiplication of the Heisenberg double in \eqref{eq:hd2}. That the module is cyclic follows from  the identity $h=\phi(h\oo 1) 1$ for all $h\in H$. To show that $\phi$ is an isomorphism, it is sufficient to show that it is surjective.
This can be seen as follows.  Let $\{x_i\}$ be a basis of $H$ and $\{\alpha^i\}$ the dual basis of $H^*$ and $\ell \in H$, $\eta\in H^*$ the Haar integrals. The properties of the Haar integral from Remark \ref{rem:haar} imply
\begin{align*}
&\phi((x_i\oo\eta)\cdot(\ell\oo \alpha^j))x_k=\langle \low\eta 1, \low\ell 2\rangle\, \phi(x_i\low\ell 1\oo \low\eta 2 \alpha^j)x_k=\langle \low\eta 1, \low\ell 2\rangle\langle \low\eta 2\alpha^j, x_{k(2)}\rangle\, x_i\low\ell 1x_{k(1)}\\
&=\langle \eta, \low\ell 2 x_{k(2)}\rangle\langle \alpha^j, x_{k(3)} \rangle x_i\low\ell 1x_{k(1)}=\langle \eta, \ell x_{k(1)}\rangle\langle \alpha^j, x_{k(2)}\rangle\, x_i=\epsilon(x_{k(1)})\,\langle \eta,\ell\rangle\, \langle\alpha^j, x_{k(2)}\rangle x_i=\langle \eta,\ell\rangle\, \delta^j_k x_i.
\end{align*}
As the  Haar integrals  satisfy   $\langle \eta,\ell\rangle\neq 0$ and the linear maps $\phi_{ij}\in \mathrm{End}_{\mathbb F}(H)$ with $\phi_{ij}(x_k)=\delta^j_k x_i$ form a basis of  $\mathrm{End}_{\mathbb F}(H)$, this proves that $\phi$ is surjective.
The formulas in \eqref{eq:kitops} for  $L^y_+$ and $T^\alpha_+$ are  obtained directly from the definition of $\phi$. 
The  ones for $L^y_\pm$ and $T^\alpha_\pm$ follow from the definition of the antipode of $D(H)^*$ in \eqref{eq:ddualmult}, which yields
\begin{align*}
&\phi\circ S_D\circ (S\oo\id)(y\oo 1)=\Sigma_{i,j} \phi(x_iy S(x_j)\oo\alpha^j\alpha^i)=\Sigma_{i,j}\;L_+^{x_iyS(x_j)}\circ T^{ \alpha^j\alpha^i}_+ \\
&\phi\circ S_D\circ (\id\oo S)(1\oo\alpha)=\Sigma_{i,j} \phi(x_iS(x_j)\oo\alpha^j\alpha\alpha^i)=\Sigma_{i,j}\;L_+^{x_iS(x_j)}\circ T^{ \alpha^j\alpha\alpha^i}_+.
\end{align*}
This implies for all $h\in H$
\begin{align*}
&\phi\circ S_D (S(y)\oo 1) h
=\Sigma_{i,j}\,\langle\alpha^j\alpha^i, \low h 2\rangle x_i y S(x_j) \low h 1=\low h 3  y  S(\low h 2)  \low h 1=h \cdot y=L^y_-h\\
&\phi\circ S_D (1\oo S(\alpha)) h=\Sigma_{i,j}\,\langle\alpha^j\alpha\alpha^i, \low h 2\rangle x_i S(x_j) \low h 1=\langle \alpha,\low h 3\rangle\,  \low h 4   S(\low h 2)  \low h 1=\langle \alpha,\low h 1\rangle\, \low h 2=T^\alpha_- h.
\end{align*}
The last claim follows  from  the fact that  $L^y_{e\pm}$, $T^\alpha_{e\pm}$ commute with  $L^y_{f\pm}$, $T^\alpha_{f\pm}$ if $e\neq f$.
\end{proof}

\section{Hopf algebra gauge theory}
\label{sec:gtheory}

The action of vertex and face operators in the Kitaev models on the extended 
space  resemble gauge symmetries. They define representations of the Hopf algebras $H$, $H^*$ and $D(H)$ on the extended 
space $H^{\oo E}$, and the protected space $\prot$ is defined as the set of states that transform trivially 
under these representations. 
 This suggests that Kitaev models could be understood as a Hopf algebra analogue of a lattice gauge theory.

This  requires a concept of a Hopf algebra valued gauge theory on a ribbon graph. An axiomatic description of such a Hopf algebra gauge theory was derived  in \cite{MW} by generalising  lattice gauge theory for a finite group. 
 It is shown there that the resulting Hopf algebra gauge theory coincides with the algebras  
  obtained in \cite{AGSI,AGSII,AS}  and \cite{BR,BR2}  via the combinatorial quantisation of Chern-Simons gauge theory, which were analysed further  in \cite{BFK}. We summarise the description in \cite{MW} but with a  change of notation and  for ribbon graphs without loops or multiple edges.

The general definition of a Hopf algebra gauge theory is obtained by linearising the corresponding structures for a gauge theory based on a finite group.
A lattice gauge theory for a ribbon graph $\Gamma$ and a finite group $G$ consists of the following:
\begin{compactitem}
\item {\bf gauge fields and functions of gauge fields:}
A 
 {\em gauge field} is an assignment  of an element  $g_e\in G$  to each oriented edge $e\in E$ and hence can be interpreted as an element of the {\em set} $G^{\times E}$. Functions of the gauge fields with values in $\mathbb F$ form a commutative algebra $\mathrm{Fun}(G^{\times E})$ with respect to pointwise multiplication and addition.  They are related to gauge fields by an evaluation map $\mathrm{ev}: \mathrm{Fun}(G^{\times E})\times G^{\times E}\to \mathbb F$, $(f,g)\mapsto f(g)$.\\[-2ex]
 
\item  {\bf gauge transformations:}
A {\em gauge transformation} is an assignment of a group element $g_v$ to each vertex $v\in V$. 
Gauge transformations at different vertices commute and the composition of 
 gauge transformations at a given vertex is given by the group multiplication in $G$. A  gauge transformation can therefore be viewed as an element of the {\em group} $G^{\times V}$. \\[-2ex]
 
 \item {\bf action of gauge transformations:}
Gauge transformations act on gauge fields via a left action $\rhd: G^{\times V}\times G^{\times E}\to G^{\times E}$.  This action is  local: a gauge transformation at a vertex $v$ acts only on the components of  gauge fields associated with edges  at $v$. This action  is given by left multiplication and right multiplication  for outgoing and incoming edges.
The left action of gauge transformations on gauge fields induces a right action  $\lhd: \mathrm{Fun}(G^{\times E})\times G^{\times V}\to \mathrm{Fun}(G^{\times E})$ defined by $(f\lhd h)(g)=f(h\rhd g)$ for all $f\in \mathrm{Fun}(G^{\times E})$, $g\in G^{\times E}$ and $h\in G^{\times V}$.\\[-2ex]

\item  {\bf observables:}
The physical observables are the {\em gauge invariant} functions $f\in \mathrm{Fun}(G^{\times E})$ with $f\rhd h=f$ for all $h\in G^{\times V}$.  They form a subalgebra of the commutative algebra $\mathrm{Fun}(G^{\times E})$. 
 \end{compactitem}

The corresponding structures for lattice gauge theories with values in a finite-dimensional Hopf algebra $K$ were obtained in \cite{MW} by {\em linearising} the structures for a finite group $G$. 
\begin{compactitem}
\item {\bf gauge fields and functions of gauge fields:}
The set $G^{\times E}$ of gauge fields  is replaced by the {\em vector space} $K^{\oo E}$. Functions of gauge fields are identified with elements of the dual vector space $K^{*\oo E}$, and the evaluation map is replaced by the pairing $\langle\,,\,\rangle: K^{*\oo E}\oo K^{\oo E}\to \mathbb F$.
One also requires an algebra structure on  $K^{*\oo E}$, although not necessarily the canonical one.\\[-2ex]

\item {\bf gauge transformations:}
The group $G^{\times V}$ of gauge transformations is replaced by the {\em Hopf algebra} $K^{\oo V}$ of gauge transformations. \\[-2ex]

\item {\bf action of gauge transformations:}
The action of gauge transformations on gauge fields takes the form of a $K^{\oo V}$-left module structure $\rhd: K^{\oo V}\oo K^{\oo E}\to K^{\oo E}$. This action must be local in the sense that the component of a gauge transformation associated with a vertex $v$ acts only on the components of gauge fields associated with edges incident at $v$. Instead of the action of $G$ on itself by left and right multiplication, it is given by the left and right regular action of $K$ on itself for incoming and outgoing edges at $v$. 
 Via the pairing, this $K^{\oo V}$-left module structure induces a $K^{\oo V}$-right module structure on the vector space $K^{*\oo E}$, which is given by the left and right regular action of $K$ on $K^*$ for  incoming and outgoing  edges at  $v$.  \\[-2ex]
 
\item {\bf observables:}
The {\em gauge invariant functions} or {\em observables} of the Hopf algebra gauge theory are defined as the invariants of the $K^{\oo V}$-right module structure on $K^{*\oo E}$, the elements $\alpha\in K^{*\oo E}$ with $\alpha\lhd h=\epsilon(h)\, \alpha$ for all $h\in K^{\oo V}$.
They  must form a subalgebra of $K^{*\oo E}$. 
\end{compactitem}

\medskip
The fundamental difference between a Hopf algebra gauge theory and a group gauge theory  is that generally one cannot  equip the vector space $K^{*\oo E}$  with the canonical algebra structure induced by the tensor product. To ensure that the invariants of the $K^{\oo V}$-right module $K^{*\oo E}$ form not only a linear subspace but a  {\em subalgebra} of $K^{*\oo E}$, the algebra structure on $K^{*\oo E}$ must chosen in such a way that $K^{*\oo E}$ is not only a $K^{\oo V}$-right module but a  $K^{\oo V}$-right  module {\em algebra} over $K^{\oo E}$. This is not the case for the canonical algebra structure on $K^{*\oo E}$ unless $K$ is cocommutative. 

This shows that the essential mathematical structure in a Hopf algebra gauge theory is a $K^{\oo V}$-right module algebra structure on the vector space $K^{*\oo E}$. It was shown in \cite{MW} that such a $K^{\oo V}$-right module algebra structure, subject to certain additional locality conditions,  can  be built up  from local Hopf algebra gauge theories on the vertex neighbourhoods of $\Gamma$.

By definition, a  Hopf algebra gauge theory on the vertex neighbourhood $\Gamma_v$ of an $n$-valent vertex $v$  is  a $K$-right module algebra structure on $K^{*\oo n}$. 
 It is shown in \cite{MW} that the natural algebraic ingredient  for a Hopf algebra gauge theory
on $\Gamma_v$ is a {\em ribbon Hopf algebra} $K$. A quasitriangular structure on $K$ is required for the  $K$-module algebra structure on $K^{*\oo n}$ is all edges are incoming. 

The
 condition that $K$ is ribbon is needed for the reversal of edge orientation. 
As we consider only the case of a  finite-dimensional semisimple Hopf algebra  $K$ over a field of characteristic zero,  quasi-triangularity of $K$ implies that $K$  is ribbon \cite{sgel} and edge orientation is reversed with the antipode $S: K^*\to K^*$. 
The $K$-right module algebra structure on $K^{*\oo n}$ is then defined by the following theorem, which makes use of  the  `braided tensor product'  introduced in \cite{majidbraid1,majidbraid}.

\pagebreak

\begin{theorem} [\cite{MW}]\ \label{rem:flipalg} Let $K$ be a finite-dimensional semisimple quasitriangular Hopf algebra  with  $R$-matrix $R\in K\oo K$ and   $\sigma_i\in\{0,1\}$ for $i\in\{1,...,n\}$. Then 
the multiplication law
\begin{align}\label{eq:flipalg}
&(\alpha)_i\cdot (\beta)_i=\begin{cases} 
\langle  \beta_{(1)}\oo\alpha_{(1)}, R\rangle \,(\low\beta 2\low\alpha 2)_i  &\sigma_i=0\\
 (\alpha\beta)_i & \sigma_i=1
\end{cases}\\
&(\alpha)_i\cdot (\beta)_j=\begin{cases}
\langle \beta_{(1)}\oo \alpha_{(1)}, R\rangle \, (\alpha_{(2)}\oo \beta_{(2)})_{ij}  &i>j\\
(\alpha\oo\beta)_{ij} & i<j,
\end{cases}
\nonumber
\end{align}
and the linear map $\lhd:  K^{*\oo n}\oo K\to K^{*\oo n}$ 
\begin{align}\label{eq:act_v}
&(\alpha^1\oo...\oo\alpha^n)\lhd h=\langle \alpha^1_{(1)}\cdots \alpha^n_{(1)}, h\rangle \; \alpha^1_{(2)}\oo\ldots \oo \alpha^n_{(2)}
\end{align}
 define a $K$-right module algebra structure on $K^{*\oo n}$. 
\end{theorem}

A Hopf algebra gauge theory for a vertex neighbourhood of a {\em ciliated} vertex $v$ with $n$ incident edge ends is then obtained as follows.  The edge ends at $v$ are numbered according to the ordering at $v$
as in Figure \ref{fig:vertex_edgeends},
 and the $i$th copy of $K^*$ in $K^{*\oo n}$ is associated with the $i$th edge end  at $v$.  
One chooses arbitrary parameters $\sigma_i\in\{0,1\}$ and sets 
 $\tau_i=0$ if the $i$th edge end is incoming at $v$ and $\tau_i=1$ if it is outgoing for $i\in\{1,...,n\}$.  The $K$-right module algebra 
  structure from Theorem \ref{rem:flipalg} is then  modified with the involution $S^{\tau_1}\oo...\oo S^{\tau_n}: K^{*\oo n}\to K^{*\oo n}$ to take into account edge orientation.

\begin{definition}[\cite{MW}] \label{th:vertex_gt}  
Let  $v$ be a ciliated vertex with $n$ incident edge ends. 
The  $K$-right module algebra structure $\mathcal A^*_v$  of a Hopf algebra gauge theory on $\Gamma_v$ is defined by the condition  
that the involution $S^{\tau_1}\oo\ldots\oo S^{\tau_n}: K^{*\oo n}\to \mathcal A^*_v$ is a morphism of $K$-right module algebras  when $K^{*\oo n}$ is equipped with  the $K$-right module algebra structure in Theorem \ref{rem:flipalg}. 
\end{definition}

The parameters $\sigma_i\in\{0,1\}$ can be chosen arbitrarily at this stage but play an essential role in gluing together the Hopf algebra gauge theories on different vertex neighbourhoods to a Hopf algebra gauge theory on $\Gamma$. This is achieved by embedding the vector space  $K^{*\oo E}$ into the tensor product $\oo_{v\in V} A^*_v$ via the injective linear map
\begin{align}\label{eq:dualemb}
&G^*=\oo_{e\in E}G^*_e: \,
K^{*\oo E}\to\oo_{v\in V}K^{*\oo |v|}, \quad (\alpha)_e\mapsto  (\low\alpha 2\oo \low\alpha 1)_{s(e)t(e)}.
\end{align}
Note that this map is dual to the map $G: K^{\oo 2E}\to K^{\oo E}$, $(\alpha\oo\beta)_{t(e)s(e)}\to (\alpha\cdot \beta)_e$ that assigns to an edge $e\in E(\Gamma)$  the product of the components of a gauge field of the edge ends $s(e)\in E(\Gamma_{\st(e)})$ and $t(e)\in E(\Gamma_{\ta(e)})$.
As the vertex neighbourhoods are obtained by splitting the edge $e\in E(\Gamma)$ into the edge ends $s(e)$ and $t(e)$, this is 
natural and intuitive from the perspective of gauge theory.

To define a local  Hopf algebra gauge theory on  $\Gamma$,  one  then assigns a universal $R$-matrix $R_v$ to each vertex $v$ and selects for each edge $e\in E(\Gamma)$ one of the associated edge ends $t(e), s(e)\in \cup_{v\in V} E(\Gamma_v)$. This  determines  
 a map $\rhop: E(\Gamma)\to \cup_{v\in V}E(\Gamma_v)$, and for each edge end  $f\in \cup_{v\in V}E(\Gamma_v)$ one sets $\sigma_{f}=0$ if $f$ is in the image of $\rhop$ and $\sigma_f=1$ else.   This defines 
 a Hopf algebra gauge theory  on each vertex neighbourhood $\Gamma_v$ and   a  $K^{\oo V}$-right module algebra  $\oo_{v\in V} \mathcal A^*_v$.
 The $K^{\oo V}$-module algebra structure of the vector space $K^{*\oo E}$ is then obtained from the following theorem.

\begin{theorem} [\cite{MW}]\label{lem:edge_algebra} Let $K$ be a finite-dimensional   
semisimple quasitriangular Hopf algebra and $\Gamma$ a ciliated ribbon graph without loops or multiple edges and equipped with the data above. Then:\\[-3ex]
\begin{compactenum}
\item  $G^*(K^{*\oo E})\subset \otimes_{v\in V} \mathcal A^*_v$  is a subalgebra  and a $K^{\oo V}$-submodule of   $\otimes_{v\in V} \mathcal A^*_v$.\\[-2ex]
\item The induced  $K^{\oo V}$-right module algebra structure  on $K^{*\oo E}$  is a Hopf algebra gauge theory on $\Gamma$.\\[-2ex]
\item If $R_v=R$ for all $v\in V$, the $K^{\oo V}$-right module algebra structure on $K^{*\oo E}$  does not depend on the choice of  $\rhop$.\\[-2ex]
  \end{compactenum}
  The vector space $K^{*\oo E}$ with this $K^{\oo V}$-module algebra structure is denoted $\mathcal A^*_\Gamma$, and its elements are called {\bf functions}. Elements of the vector space $K^{\oo E}$  are called  {\bf gauge fields}, and elements  of the Hopf algebra  $K^{\oo V}$   {\bf gauge transformations}. The invariants of $\mathcal A^*_\Gamma$ are   
 called {\bf gauge invariant functions} or {\bf observables}. 
 \end{theorem}

In the following, we restrict attention to local Hopf algebra gauge theories in which the same $R$-matrix  is assigned to each vertex $v$ and to loop-free ribbon graphs $\Gamma$ without multiple edges.  An explicit description  of the  algebra $\mathcal A^*_\Gamma$  in terms of generators and multiplication relations is given in \cite[Lemma 3.20]{MW}.  For the case at hand, we can summarise 
 the 
relations from \cite[Lemma 3.20]{MW} and the $K^{\oo V}$-module structure on $\mathcal A^*_\Gamma$  in the notation from Section \ref{subsec:conventions} as follows.

\begin{proposition}[{\cite{MW}}]\label{prop:multrel} Let $K$ be a finite-dimensional semisimple quasitriangular Hopf algebra and  $\Gamma$  a ciliated ribbon graph without loops or multiple edges. Assign the same $R$-matrix $R\in K\oo K$ to each vertex $v\in V(\Gamma)$.  Then the algebra $\mathcal A^*_\Gamma$ is generated by the elements $(\alpha)_e$ for  the edges $e$ of $\Gamma$ and $\alpha\in K^*$, subject to the following  relations:
\begin{compactitem}
\item $(\beta)_e\cdot (\alpha)_e=\langle \low\alpha 1\oo\low\beta 1, R\rangle\, (\low \alpha 2\low\beta 2)_e$ for all edges $e$ of $\Gamma$,\\[-2ex]
\item $(\beta)_f\cdot (\alpha)_e=(\alpha)_e\cdot (\beta)_f$ if the edges $e,f$ have no vertex in common,\\[-2ex]
\item $(\beta)_f\cdot (\alpha)_e=\langle S^{\tau_e}(\low\alpha {1+\tau_e})\oo S^{\tau_f}(\low\beta {1+\tau_f}, R\rangle\, (\low\beta {2-\tau_f})_f\cdot (\low \alpha {2-\tau_e})_e$ for edges $e,f$ with a common vertex $v$, where $\tau_h=0$ if $h$ is incoming at $v$ and $\tau_h=1$ if $h$ is outgoing.
\end{compactitem}
Its $K^{\oo V}$-module structure is given by
$(\alpha)_e\lhd(h)_v=\epsilon(h)(\alpha)_e$ for $v\notin\{\st(e), \ta(e)\}$,  $(\alpha)_e\lhd(h)_{\ta(e)}=\langle \low\alpha 1, h\rangle(\alpha)_e$ and $(\alpha)_e\lhd (h)_{\st(e)}=\langle S(\low\alpha 2), h\rangle (\alpha)_e$ for all 
 $\alpha\in K^{*\oo E}$ and $h\in K^{\oo V}$.
\end{proposition}

By Lemma  \ref{lem:project} the gauge invariant functions, which are the invariants of the $K^{\oo V}$-right module algebra $K^{*\oo E}$, form a subalgebra
 $\mathcal A^*_{\Gamma\,inv}\subset \mathcal A^*_{\Gamma}$.  It is shown in \cite{MW} that this subalgebra is independent of the choice of the cilia and depends only on the homeomorphism  class of the surface $\dot\Sigma_\Gamma$ obtained by gluing annuli to the faces of $\Gamma$.

 \begin{theorem}[\cite{MW}]\label{th:ginvsub} Let 
 $K$ be a finite-dimensional semisimple quasitriangular Hopf algebra with Haar integral $\ell$ and $\Gamma$ a ciliated ribbon graph. 
  Then $\mathcal A^*_{\Gamma\, inv}$ 
  is a subalgebra of $\mathcal A^*_\Gamma$ and depends only on the homeomorphism class of  $\dot\Sigma_\Gamma$. 
The map
$P_{inv}: \mathcal A^*_\Gamma\to\mathcal A^*_\Gamma$, $\alpha\mapsto \alpha\lhd \ell^{\oo V}$  
is a projector on $\mathcal A^*_{\Gamma\,inv}$. 
 \end{theorem}

Besides gauge invariance, there is another essential  concept in a lattice gauge theory, namely curvature. 
In a classical gauge theory on an oriented surface curvature is locally a 2-form and hence can be integrated over discs on the surface. In a description based on a ribbon graph $\Gamma$, these discs correspond to the faces of $\Gamma$. Hence, curvature  in a Hopf algebra gauge theory is described by Hopf algebra valued holonomies of faces of $\Gamma$.
  
On the level of gauge fields,  holonomy is an assignment of a  linear map  $\mathrm{Hol}_p: K^{\oo E} \to K$
to each path $p\in \mathcal G(\Gamma)$. 
In the dual picture for functions this corresponds to an assignment of a linear map  $\mathrm{Hol}_p: K^* \to K^{*\oo E}$ to each path  $p\in \mathcal G(\Gamma)$. This assignment must be compatible with trivial paths, composition of paths and inverses and hence define a functor $\mathrm{Hol}:\mathcal G(\Gamma)\to \mathrm{Hom}_{\mathbb F}(K^*,K^{*\oo E})$, where $\mathrm{Hom}_{\mathbb F}(K^*,K^{*\oo E})$ is equipped with the structure of an $\mathbb F$-linear category with a single object, i.~e.~with an associative algebra structure over $\mathbb F$. This  
was defined in \cite{MW} as  a convolution product.
For any algebra $(A,m,1)$ and coalgebra $(C,\Delta, \epsilon)$ over $\mathbb F$ the convolution product
\begin{align}\label{eq:assalg}\bullet: \mathrm{Hom}_{\mathbb F}(C,A)\oo \mathrm{Hom}_{\mathbb F}(C,A)\to \mathrm{Hom}_{\mathbb F}(C,A), \quad \phi\oo\psi\mapsto \phi\bullet \psi=m \circ (\phi\oo\psi)\circ \Delta, \end{align}
defines an  associative algebra structure on  $\mathrm{Hom}_{\mathbb F}(C,A)$
 with unit  $\epsilon\,1$. It is argued in \cite{MW} that the natural choice of $A$ and $C$ for a Hopf algebra gauge theory are  $A=K^{*\oo E}$ and  $C=K^*$, with the canonical algebra and coalgebra structure.
As the path groupoid $\mathcal G(\Gamma)$ 
is generated by the edges of $\Gamma$,  a functor $\mathrm{Hol}: \mathcal G(\Gamma)\to \mathrm{Hom}_{\mathbb F}(K^*, K^{*\oo E})$ is defined uniquely by its values on the paths $e^{\pm 1}\in \mathcal G(\Gamma)$ for $e\in E$. For these paths it is natural to choose $\mathrm{Hol}_e=\iota_e$ and $\mathrm{Hol}_{e^\inv}=\iota_e\circ S$, where $\iota_e: K^*\to K^{*\oo E}$, $\alpha \mapsto (\alpha)_e$ is the map  that sends an element $\alpha\in K^*$
to the pure tensor $(\alpha)_e\in K^{*\oo E}$ that has the entry $\alpha$ in the component of $K^*$ associated with $e\in E$ and $1$ in all other components, as defined at the beginning of  Section \ref{subsec:conventions}.

\begin{definition} [\cite{MW}]\label{def:hol} Let $\Gamma$ be a ciliated ribbon graph and $K$ a  finite-dimensional semisimple quasitriangular Hopf algebra. 
Equip $\mathrm{Hom}_{\mathbb F}(K^*, K^{*\oo E})$ with the algebra structure \eqref{eq:assalg} for $A=K^{*\oo E}$ and $C=K^*$. 
The {\bf holonomy} for a $K$-valued Hopf algebra gauge theory on $\Gamma$   is the functor $\mathrm{Hol}:\mathcal G(\Gamma)\to \mathrm{Hom}_{\mathbb F}(K^*,K^{*\oo E })$ defined by 
$\mathrm{Hol}_{e}=\iota_e$ and $\mathrm{Hol}_{e^\inv}=\iota_e\circ S$ for all $e\in E$. If $f$ is a ciliated face of $\Gamma$, then the holonomy $\mathrm{Hol}_f: K^*\to K^{*\oo E}$ is called a {\bf curvature}.
\end{definition}

It is shown in \cite{MW} that under certain assumptions on $\Gamma$, the curvatures of  ciliated faces $f$ 
take values in the centre of $\mathcal A^*_{\Gamma\,inv}$, give rise to representations of the character algebra $C(K)$ and define projectors on subalgebras of  $\mathcal A^*_{\Gamma\, inv}$. The latter can be viewed as the algebras of gauge invariant  functions on the linear subspace of gauge fields that are flat at $f$. If $\Gamma$ is a regular ciliated ribbon graph, these assumptions on $\Gamma$ are  satisfied, and the results from \cite{MW} can be summarised as follows.

\begin{lemma}[\cite{MW}]\label{lem:facecentral} 
Let  $K$ be a finite-dimensional semisimple quasitriangular Hopf algebra, $\Gamma$ a regular ciliated ribbon graph and $f$  a ciliated face of $\Gamma$ that starts and ends at a cilium.
Then:
\begin{compactenum}
\item  The linear map $P_{inv}\circ \mathrm{Hol}_f: K^*\to \mathcal A^*_{\Gamma\,inv}$ takes values in the centre  $Z(\mathcal A^*_{\Gamma\,inv})$ and depends only on the associated face $f$. \\[-2ex]

\item It satisfies $P_{inv}\circ \mathrm{Hol}_f=\mathrm{Hol}_f\circ \pi_{ad}$, where $P_{inv}$ is the projector from Theorem \ref{th:ginvsub} and $\pi_{ad}: K^*\to K^*$, $\alpha\mapsto \alpha\lhd^*_{ad}\ell$  the projector on the character algebra $C(K)$  from Example \ref{ex:kad}.\\[-2ex]
\item  The map
$\mathrm{Hol}_f\vert_{C(K)}: C(K)\to \mathcal A^*_{\Gamma\,inv}$
 is an algebra morphism with values in $Z(\mathcal A^*_{\Gamma\,inv})$. \\[-2ex]
 \item The map $\lhd: \mathcal A^*_{\Gamma\,inv}\oo C(K)^{\oo F}\to \mathcal A^*_{\Gamma\,inv}$, $\alpha\oo (\beta)_{f(v)} \mapsto \mathrm{Hol}_{f}(\beta)\cdot \alpha$  defines a $C(K)^{\oo F}$-right module structure on $\mathcal A^*_{\Gamma\,inv}$.
 \end{compactenum}
\end{lemma}
As the Haar integral $\eta\in K^*$ is contained in the character algebra $C(K)$, it follows from Lemma \ref{lem:facecentral} that the element  $\mathrm{Hol}_f(\eta)$ is central in $\mathcal A^*_{\Gamma\,inv}$ for each  face $f$ that is  based at a cilium of $\Gamma$.
Moreover, as $\mathrm{Hol}_f\vert_{C(K)}: C(K)\to \mathcal A^*_{\Gamma\,inv}$ is an algebra morphism, the element $\mathrm{Hol}_f(\eta)$ is an idempotent.  This associates to each face $f$
an algebra morphism that projects on a subalgebra of $\mathcal A^*_{\Gamma\,inv}$.

\begin{lemma}[\cite{MW}] \label{lem:faceprpr} Let $\Gamma$ be a regular ciliated ribbon graph,  $K$ a finite-dimensional semisimple quasitriangular Hopf algebra  and $\eta\in K^*$ the Haar integral of its dual. Then for each  face
$f$ of $\Gamma$  based at a cilium
the map  $P_{f}:\mathcal A^*_{\Gamma}\to\mathcal A^*_{\Gamma}$, $\alpha\mapsto \mathrm{Hol}_f(\eta)\cdot \alpha$
 is a projector, and its  restriction to $\mathcal A^*_{\Gamma\,inv}$ is an algebra endomorphism. \end{lemma}

For a regular ciliated ribbon graph $\Gamma$ each face $f$ of $\Gamma$ is represented by a unique ciliated face that starts and ends at a cilium of $\Gamma$.
Hence
Lemma \ref{lem:faceprpr} associates a projector
$P_{f}:\mathcal A^*_{\Gamma\,inv}\to\mathcal A^*_{\Gamma\,inv}$ to each face $f$ of $\Gamma$. 
 Lemma \ref{lem:facecentral} then implies that the elements $\mathrm{Hol}_f(\eta)$ and $\mathrm{Hol}_{f'}(\eta)$ commute for all  faces $f,f'$. By composing the projectors $P_f$ for all faces $f\in F$, one then obtains a projector  on a subalgebra of $\mathcal A^*_{\Gamma\,inv}$. As explained in \cite{MW},  its image can be interpreted as the algebra of functions on the linear subspace of flat gauge fields on $\Gamma$.

\begin{theorem}[\cite{MW}]\label{th:flatinv}  Let  $\Gamma$ be a regular ciliated ribbon graph and $K$ a finite-dimensional semisimple quasitriangular Hopf algebra.  Then  the linear map
$$P_{flat}=\Pi_{f\in F} P_f:\mathcal A^*_{\Gamma\,inv}\to\mathcal A^*_{\Gamma\,inv}, \quad \alpha\mapsto \Pi_{f\in F}\mathrm{Hol}_f(\eta)\cdot \alpha$$  is an algebra morphism and a projector. Its image  $\mathcal M_{\Gamma}=P_{flat}(\mathcal A^*_{\Gamma\,inv})$ is a subalgebra of $\mathcal A^*_{\Gamma\,inv}$, the {\bf quantum moduli algebra}.
\end{theorem}

The quantum moduli algebra was first constructed in \cite{AGSII,BR,BR2}. It  can be defined  more generally for ciliated ribbon graphs with loops or multiple edges and with milder assumptions on the faces \cite{AGSII,BR,BR2,MW}. It was shown in \cite{AGSII,BR2}  and then with different methods in \cite{MW}  that the quantum moduli algebra is a topological invariant. It depends only on the oriented surface $\Sigma_\Gamma$ obtained by gluing discs to the faces of  $\Gamma$ and not on  $\Gamma$  itself or the choice of the cilia. This result holds for more general ciliated ribbon graphs, but  we require only the regular case.

\begin{theorem} [\cite{AGSI,AGSII,BR,MW}]\label{th:flat}Let $\Gamma,\Gamma'$ be regular ciliated ribbon graphs and  $\Sigma_\Gamma$ and $\Sigma_{\Gamma'}$  the  oriented surfaces obtained by gluing discs to the faces of $\Gamma$ and $\Gamma'$.  If  $\Sigma_\Gamma$ and $\Sigma_{\Gamma'}$ are homeomorphic, then the quantum moduli algebras  $\mathcal M_\Gamma$ and $\mathcal M_{\Gamma'}$ are isomorphic.
\end{theorem}

\section{Holonomies and gauge symmetries in the Kitaev model}
\label{sec:holkitaev}

We are now ready to relate  Kitaev's lattice model on a ribbon graph $\Gamma$ and  for a finite-dimensional semisimple Hopf algebra $H$  to a Hopf algebra gauge theory  for the Drinfeld double $D(H)$.  
The concept that is essential in relating the two models is {\em holonomy}. 
It will become apparent that Kitaev's {\em ribbon operators} \cite{Ki,BMD} are examples of  holonomies. However, ribbon operators are defined only for paths with certain regularity properties, the {\em ribbon paths}, while we require a more  general notion of holonomy that is not  restricted to ribbon paths.

For this reason, we define holonomies for a Kitaev model  on $\Gamma$   along the same lines as the holonomy for a Hopf algebra gauge theory, but with respect to a different ribbon graph $\gammad$ which was introduced in the context of Kitaev models \cite{Ki,BMD} and is obtained by thickening the ribbon graph $\Gamma$. 
We start by introducing this thickened ribbon graph $\gammad$ in Section \ref{subsec:thickening}. In Section \ref{subsec:kithols}  we then   define holonomies for the Kitaev models and investigate their basic properties, in particular their relation to ribbon operators.   In Section \ref{subsec:vertface}  we show how  vertex and face operators of the Kitaev model arise as examples of holonomies.

\subsection{Thickenings of ribbon graphs}
\label{subsec:thickening}

The thickening of a  ribbon graph  $\Gamma$  is obtained by replacing each edge $e$ by a rectangle $R_e$, each vertex  $v$ by a $|v|$-gon $P_v$ and by gluing two opposite sides of the rectangle $R_e$ to the sides of the polygons $P_{\st(e)}$ and $P_{\ta(e)}$ according to the cyclic ordering  at $\st(e)$ and $\ta(e)$.
This can be viewed as a generalisation of the gluing procedure 
from Section \ref{subsec:graph}. For each ribbon graph $\Gamma$ there is a  punctured surface $\dot\Sigma_\Gamma$ obtained by gluing annuli to the faces of $\Gamma$, and this surface $\dot\Sigma_\Gamma$ is homeomorphic to the  thickening of $\Gamma$. 

This thickening procedure assigns to each ribbon graph $\Gamma$ a 4-valent ribbon graph $\gammad$. 
The oriented edges of $\gammad$ are the  edges of the rectangles $R_e$ for each edge $e$ of $\Gamma$, where the two edges of $R_e$ that are not glued to polygons $P_{\st(e)}$ and $P_{\ta(e)}$ are oriented parallel to $e$. The  remaining two edges are oriented by duality, i.~e.~such that they cross $e$ from the right to the left when viewed in the direction of $e$,  as shown in Figure \ref{fig:foursites}. Vertices of  $\gammad$ are in bijection with  vertices of the  polygons $P_v$. Each  face of $\gammad$  corresponds either to an edge, to a face  or to a vertex of $\Gamma$, as shown in Figure \ref{fig:graphfusion}.

\begin{figure}
\centering
\includegraphics[scale=0.32]{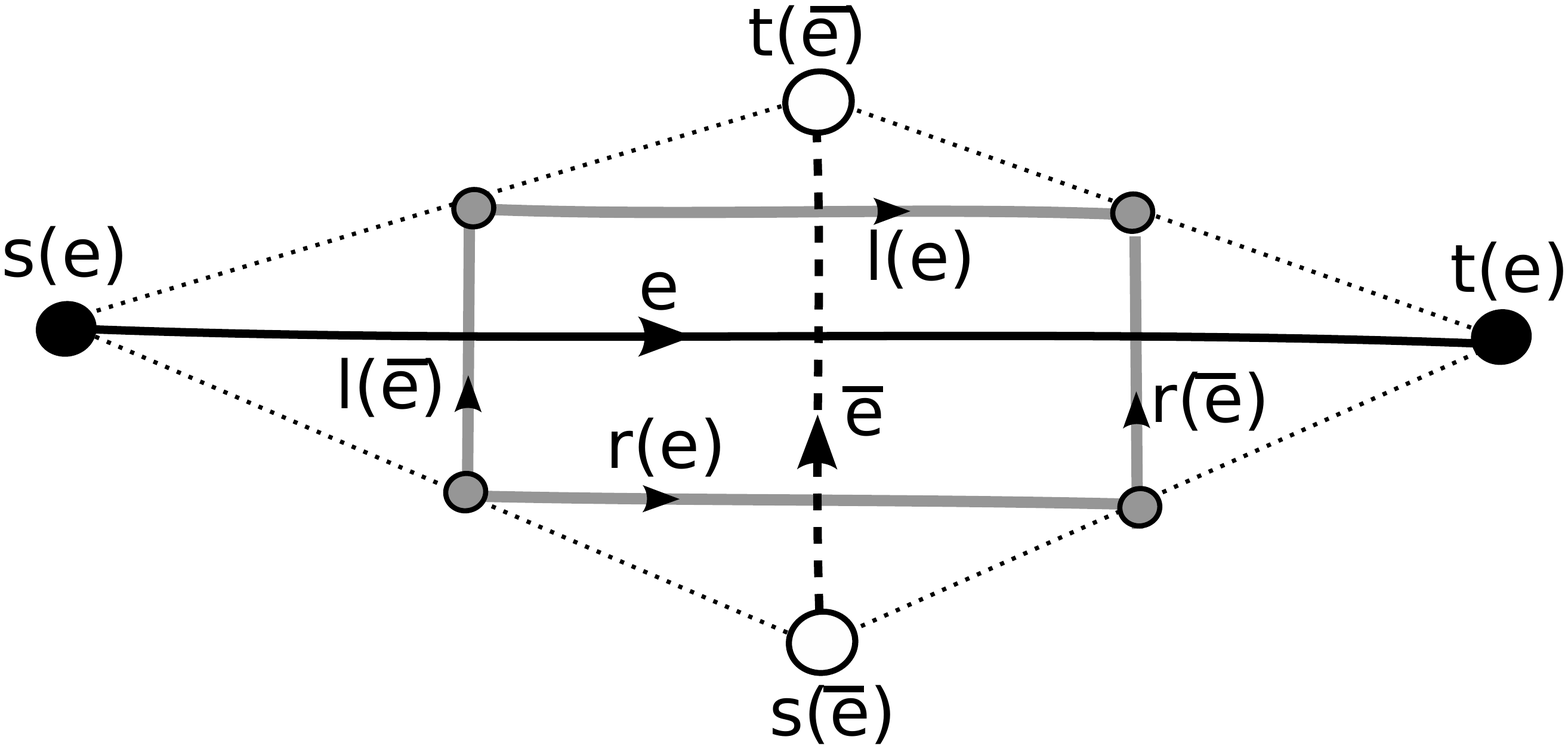}

\vspace{-1.8cm}
\caption{The four edges $r(e)$, $l(e)$, $r(\bar e)$, $l(\bar e)$ of $\gammad$ for an edge $e$ of $\Gamma$.}
\label{fig:foursites}
\end{figure}

Alternatively, the ribbon graph $\gammad$ can also be obtained from the Poincar\'e dual. For this, one embeds $\Gamma$ into the  oriented surface $\Sigma_\Gamma$ and considers its Poincar\'e dual  $\bar\Gamma$. 
 To construct $\gammad$ one connects each vertex $v\in V(\Gamma)$  to those vertices $\bar f\in V(\bar\Gamma)$
 that are dual to faces containing $v$. By selecting a  point in the interior of each of the resulting edges $e_{vf}$, one obtains the vertices of $\gammad$. The edges of $\gammad$ are  obtained by connecting those vertices of $\gammad$ for which the edges $e_{vf}$ and $e_{v'f'}$ are adjacent  at a vertex of $\Gamma$ or at a dual vertex of $\bar\Gamma$. This construction  is shown in Figure \ref{fig:graphfusion}. 

\begin{figure}
\centering
\includegraphics[scale=0.43]{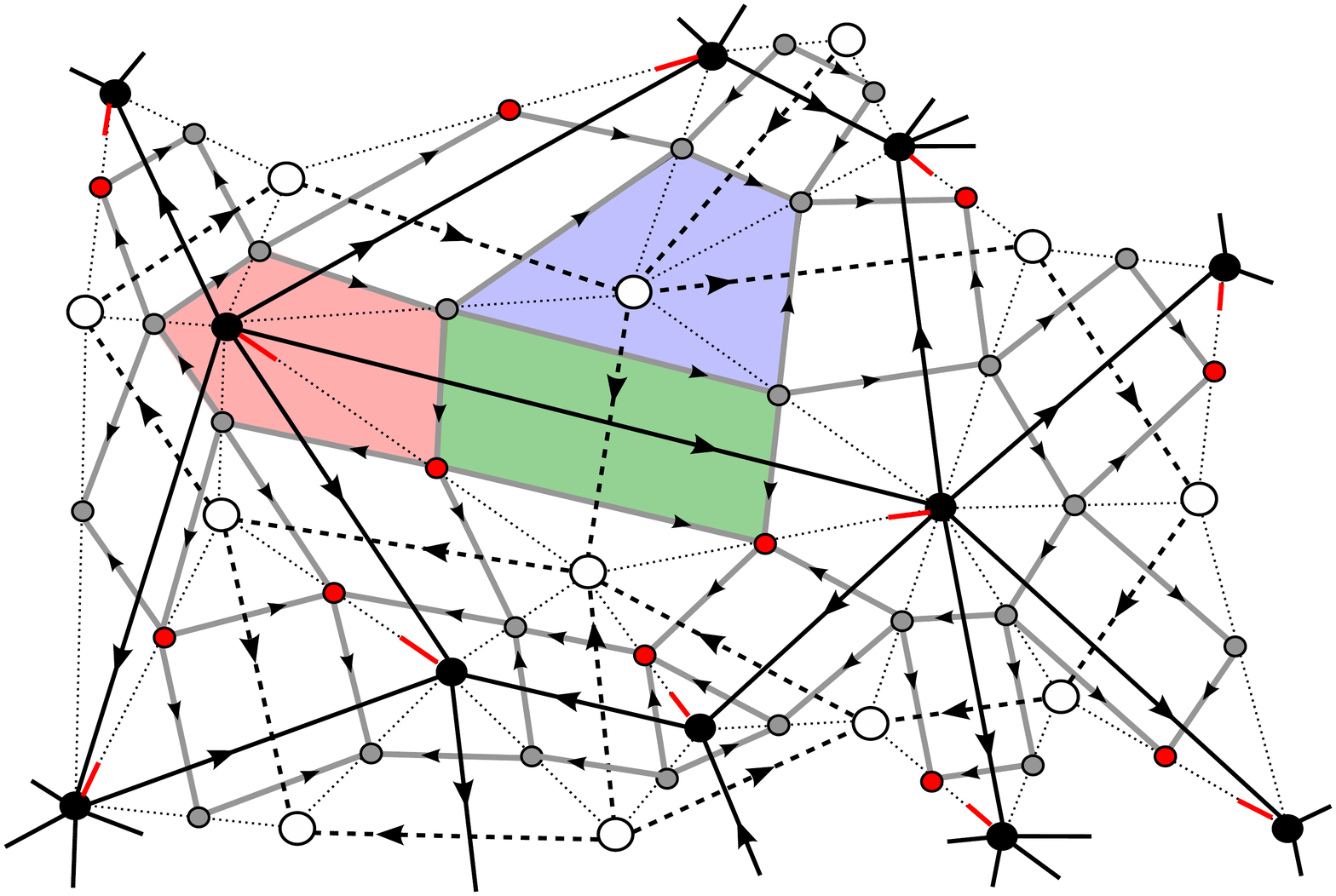}
\caption{Ribbon graph $\Gamma$ (black vertices and black edges), its dual $\bar\Gamma$ (white vertices and dashed edges), the graph $\gammad$ (grey and red vertices, grey edges). The cilia at the vertices of $\Gamma$ and the corresponding vertices of $\gammad$ are indicated in red. The rectangle $R_e$ for an edge $e$ in $\Gamma$ is highlighted in green, the polygon $P_v$ for a vertex $v$ of $\Gamma$ in red and the polygon $P_f$ for a face  $f$ of $\Gamma$ in blue. }
\label{fig:graphfusion}
\end{figure}

The four edges of $\gammad$ that correspond to a given edge $e$ of $\Gamma$ are denoted $r(e)$, $l(e)$, $r(\bar e)$ and $l(\bar e)$,  where  $r(e)$ and  $l(e)$ stand for the edges of $\gammad$ to the right and left of $e$, viewed in the direction of $e$.  Similarly,  $r(\bar e)$, $l(\bar e)$  are the edges of $\gammad$
transversal to $e$ at the target and starting end of $e$, as shown in Figure \ref{fig:foursites}.  
If $e\in E(\Gamma)$ and $\bar e\in E(\bar\Gamma)$ are  oriented edges  and $e^\inv$, $\bar e^{\,\inv}$ the corresponding edges with the reversed orientation, then one has
 \begin{align}\label{eq:orrev}
 r(e^\inv)=l(e)^\inv\qquad 
r(\bar e^\inv)=l(\bar e)^\inv.
 \end{align}
The four vertices  of $\gammad$ associated with a generic edge $e\in E(\Gamma)$ are given by the pairs 
$(\st(e), \st(\bar e))$, $(\st(e), \ta(\bar e))$, $(\ta(e), \st(\bar e))$, $(\ta(e),\ta(\bar e))$
where $\st(e)$ and $\ta(e)$ are the starting and target vertex of $e$ and $\st(\bar e)$ and $\ta(\bar e)$ the starting and target vertex of $\bar e$. 
More specifically, we have in $\gammad$
\begin{align}\label{eq:statarel}
&\st( r(\bar e))= \ta(r(e))=(\ta(e), \st(\bar e)) & &\ta(r(\bar e))=\ta(l(e))=(\ta(e), \ta(\bar e))\\
 &\st(l(\bar e))=\st(r(e))=(\st(e), \st(\bar e)) & &\ta(l(\bar e))=\st(l(e))=(\st(e), \ta(\bar e)),\nonumber
\end{align}
as shown in Figure \ref{fig:foursites}.
Note that for edges of $\Gamma$  that are loops or dual to loops some of these four vertices may coincide, but this does not happen if $\Gamma$ is a regular ciliated ribbon graph. In this case the vertices of $\gammad$ are in bijection with the sites of $\Gamma$ and each edge of $\Gamma$ corresponds to exactly four sites. The four vertices $s(e)$, $t(e)$, $s(\bar e)$, $t(\bar e)$ for an edge $e\in E(\Gamma)$ form four triangles,  $s(e)s(\bar e)t(\bar e)$,  $t(e)t(\bar e)s(\bar e)$,  $t(e)t(\bar e)s(e)$  and $t(e)s(e)s(\bar e)$, which are called {\bf direct triangles} and {\bf dual triangles} in the context of Kitaev models, depending on their number of vertices and dual vertices.

\begin{definition} Let $\Gamma$ be a ribbon graph. The thickening  of $\Gamma$ is the 4-valent ribbon graph $\gammad$ with
\begin{align*}
&E(\gammad)=\cup_{e\in E(\Gamma)}\{r(e), l(e), r(\bar e), l(\bar e)\}\\
&V(\gammad)=\cup_{e\in E(\Gamma)} \{(\st(e), \st(\bar e)), (\st(e), \ta(\bar e)), (\ta(e), \st(\bar e)), (\ta(e),\ta(\bar e))\}\subset V(\Gamma)\times V(\bar \Gamma),
 \end{align*}
subject to the relations \eqref{eq:statarel}.
\end{definition}

Note that a cilia of a ribbon graph $\Gamma$ can also be incorporated into the thickening $\gammad$.
Choosing a cilium at a vertex $v$ amounts to selecting a vertex of the polygon $P_v$ in the thickening $\gammad$,
namely the vertex of $P_v$ that is between the edge ends of highest and lowest order at $v$, as indicated in Figure \ref{fig:graphfusion}.
We will also call this  vertex of $\gammad$  the cilium at $v$ in the following.

\subsection{Holonomies in the Kitaev model}
\label{subsec:kithols}

Holonomies for  a Kitaev model  on $\Gamma$  are defined  analogously to holonomies for a Hopf algebra gauge theory on $\Gamma$, but with respect to the thickened graph $\gammad$  instead of  $\Gamma$. 
As the algebra of triangle operators  is isomorphic to the $E$-fold tensor product $\mathcal H(H)^{\oo E}$ by Lemma \ref{lem:kitoprel} and $\mathcal H(H)\cong H\oo H^*$ as a vector space,  we define holonomy as a functor $\mathrm{Hol}:\mathcal G(\gammad)\to \mathrm{Hom}_{\mathbb F}(H\oo H^*, (H\oo H^*)^{\oo E})$, where $\mathcal G(\gammad)$ is the path groupoid of  $\gammad$ and 
$\mathrm{Hom}_{\mathbb F}(H\oo H^*, (H\oo H^*)^{\oo E})$ is equipped with the structure of an $\mathbb F$-linear category with a single object, i.~e.~an associative algebra structure over $\mathbb F$. 

As in the case of a Hopf algebra gauge theory, this algebra structure is obtained from a coalgebra structure on $H\oo H^*$ and an algebra structure on $(H\oo H^*)^{\oo E}$ via  \eqref{eq:assalg}. As   $H\oo H^*=D(H)^*$ as a vector space and to make contact with the holonomies in a Hopf algebra gauge theory for $D(H)$, it is natural to choose the coalgebra structure on $D(H)^*$ for the former and the canonical algebra structure on $D(H)^{*\oo E }$ for the latter.   
As in the case of a Hopf algebra gauge theory, a holonomy functor is then determined uniquely by its values on the edges of $\gammad$, i.~e.~by its values on the paths $r(e)^{\pm 1}$, $r(\bar e)^{\pm 1}$, $l(e)^{\pm 1}$ and $l(\bar e)^{\pm 1}$ for edges $e\in E(\Gamma)$. With  the notation from Section \ref{subsec:conventions} we obtain

\begin{lemma} \label{lem:functor} Let $\Gamma$ be a ribbon graph with thickening $\gammad$.
Equip $\mathrm{Hom}_{\mathbb F}(H\oo H^*, (H\oo H^*)^{\oo E})$ with the associative algebra structure $\bullet$ from \eqref{eq:assalg} for the coalgebra $D(H)^*$ and the algebra  $D(H)^{*\oo E}$ and denote by $S_D: H\oo H^*\to H\oo H^*$  the antipode of $D(H)^*$.
Then the following defines a functor $\mathrm{Hol}: \mathcal G(\gammad)\to \mathrm{Hom}_{\mathbb F}(H\oo H^*, (H\oo H^*)^{\oo E})$
\begin{align}\label{eq:phirl}
&\mathrm{Hol}_{r(e)}(y\oo \gamma)=\epsilon(y)\, (1\oo \gamma)_e   \qquad \qquad \qquad \mathrm{Hol}_{r(\bar e)}(y\oo\gamma)=\epsilon(\gamma)
  (y\oo 1)_e,  \\
&\mathrm{Hol}_{l(e)}(y\oo\gamma)=(S_D\circ \mathrm{Hol}_{r(e)}\circ S_D)(y\oo\gamma)=\epsilon(y)\Sigma_{i,j} (x_iS(x_j)\oo \alpha^j\gamma\alpha^i)_e\nonumber\\
&\mathrm{Hol}_{l(\bar e)}(y\oo\gamma)=(S_D\circ \mathrm{Hol}_{r(\bar e)}\circ S_D)(y\oo\gamma)=\epsilon(\gamma)\, \Sigma_{i,j}(x_i y S(x_j)\oo\alpha^j\alpha^i)_e\nonumber \\
&\mathrm{Hol}_{x^\inv}=\mathrm{Hol}_{x}\circ S_D \quad \forall x\in E(\gammad), e\in E(\Gamma).\nonumber
\end{align}
\end{lemma}
\begin{proof}  As $\mathcal G(\gammad)$ is the free groupoid generated by $E(\gammad)$, it is sufficient to show that  for all $x\in E(\gammad)$ one has
$\mathrm{Hol}_x\bullet \mathrm{Hol}_{x^\inv}= \mathrm{Hol}_{x^\inv}\bullet \mathrm{Hol}_x=\epsilon\, (1\oo 1)^{\oo E}$ .
The definition of $\bullet$ in \eqref{eq:assalg} implies
\begin{align}\label{eq:holhelp1}
\mathrm{Hol}_{x^{\pm 1}\circ x^{\mp 1}}(y\oo\gamma)&= \mathrm{Hol}_{x^{\pm 1}} ((y\oo\gamma)_{(1)})\cdot' \mathrm{Hol}_{x^{\mp 1}}((y\oo\gamma)_{(2)})\\
&=\Sigma_{i,j}\, \mathrm{Hol}_{x^{\pm 1}} (\low y 1\oo \alpha^i\low\gamma 2\alpha^j)\cdot' \mathrm{Hol}_{x^{\mp 1}}(S(x_j)\low y 2 x_i\oo\low\gamma 2),\nonumber
\end{align}
where we use Sweedler notation and $\cdot'$ denotes the multiplication of  $D(H)^{*\oo E}$.  With the expression for $S_D$ in \eqref{eq:ddualmult}, one obtains from 
\eqref{eq:phirl}
\begin{align}\label{eq:helpinv}
&\mathrm{Hol}_{x^\inv}(y\oo\gamma)
=\epsilon(y)\;\mathrm{Hol}_{x}(1\oo S(\gamma))\;\;\forall x\in \{r(e),l(e)\}\\
 &\mathrm{Hol}_{x^\inv}(y\oo\gamma)
=\epsilon(\gamma)\;\mathrm{Hol}_{x}(S(y)\oo  1)\;\;\forall x\in \{\bar r(e),\bar l(e)\}.\qquad\qquad\qquad\nonumber
\end{align}
Inserting this  into \eqref{eq:holhelp1}  with the formulas in  \eqref{eq:ddualmult} for the multiplication and comultiplication of $D(H)^*$ and 
 the identity $m\circ (S_D\oo\id)\circ\Delta=m\circ (\id\oo S_D)\circ\Delta=\epsilon\, 1$ and $S^2_D=\id$  then proves the claim.
\end{proof}

Note that this is not the only possible holonomy functor since there is at least one other candidate for the multiplication on $(H\oo H^*)^{\oo E }$ that enters the definition of $\bullet$ in \eqref{eq:assalg}, namely the multiplication map of  $\mathcal H(H)^{\oo E}$. 
However, the multiplication  of $D(H)^{*\oo E}$ is  more natural from the viewpoint of Hopf algebra gauge theory. Defining the multiplication $\bullet$ in \eqref{eq:assalg} with the comultiplication of $K^*$ and the multiplication of $K^{*\oo E}$ is possible for any Hopf algebra $K$, while the multiplication of  $\mathcal H(H)^{\oo E}$ is only available for  a Drinfeld double $K=D(H)$. Another strong motivation to choose the multiplication of $D(H)^{*\oo E}$ are the properties of the resulting holonomy functor.

\begin{lemma}\label{lem:holprops} The holonomy functor  from Lemma \ref{lem:functor} satisfies:
\begin{compactenum}

\item $\mathrm{Hol}_p(1\oo 1)=(1\oo 1)^{\oo E}$ for all paths $p$ in $\gammad$.\\[-2ex]
\item $\mathrm{Hol}_{p^\inv}=\mathrm{Hol}_p\circ S_D$ if $p$ is a path that traverses each edge of $\gammad$  at most once and  at most one of the edges $r(e)$, $l(e)$ and of the edges $r(\bar e)$, $l(\bar e)$ for each edge $e$ of $\Gamma$.  \\[-2ex]
\item $\mathrm{Hol}_p(y\oo\gamma)=\epsilon(y)\mathrm{Hol}_p(1\oo \gamma)$ if $p$ is a path composed of edges $r(e), l(e)$  for edges  $e$ of $\Gamma$.\\[-2ex] 
\item $\mathrm{Hol}_p(y\oo\gamma)=\epsilon(\gamma)\mathrm{Hol}_p(y\oo 1)$ if $p$ is a path composed of the edges $r(\bar e), l(\bar e)$ 
 for edges  $e$ of $\Gamma$.\\[-2ex] 
\item $\mathrm{Hol}_{r(\bar e)\circ r(e)}=\mathrm{Hol}_{l(e)\circ l(\bar e)}$ and $\mathrm{Hol}_{r(\bar e)^\inv\circ l(e)}=\mathrm{Hol}_{ r(e)\circ l(\bar e)^\inv}.$
\end{compactenum}
\end{lemma}
\begin{proof}
Claim 1.~follows directly from the identities $\Delta(1\oo 1)=(1\oo 1)\oo (1\oo 1)$, $\epsilon(1)=1$ together with equation  \eqref{eq:phirl}. Claims 2.-4. follow by induction over the length of $p$. If $p$ is a path of length 1, then  they hold by definition. For a composite  path $p\circ q$ one has
\begin{align}\label{eq:helphelp}
&\mathrm{Hol}_{p\circ q}(y\oo\gamma)=\mathrm{Hol}_{p}( (y\oo\gamma)_{(1)})\cdot' \mathrm{Hol}_{q}( (y\oo\gamma)_{(2)})\\
&\mathrm{Hol}_{(p\circ q)^\inv}(y\oo\gamma)=\mathrm{Hol}_{q^\inv\circ p^\inv}(y\oo\gamma)=\mathrm{Hol}_{q^\inv}( (y\oo\gamma)_{(1)})\cdot' \mathrm{Hol}_{p^\inv}( (y\oo\gamma)_{(2)}).\nonumber
\end{align}
Suppose that 2.~is shown for all paths of length $\leq n$. 
If $p\circ q$ is a path of length $n+1$ that satisfies the assumptions of 2.,  then  so do $p$ and $q$. As they are of length $\leq n$, by induction hypothesis   the last expression in \eqref{eq:helphelp} can be rewritten  as
\begin{align*}
&\mathrm{Hol}_{(p\circ q)^\inv}(y\oo\gamma)
=\mathrm{Hol}_{q}( S_D(y\oo\gamma)_{(2)})\cdot' \mathrm{Hol}_{p}( S_D(y\oo\gamma)_{(1)}),
\end{align*}
where we used the identity $(S_D\oo S_D)\circ \Delta=\Delta^{op}\circ S_D$.
If $p\circ q$ traverses each edge of $\gammad$ at most once and at most one of the edges $r(e)$, $l(e)$ and of $r(\bar e)$, $l(\bar e)$ for each edge $e\in E(\Gamma)$ then 
the holonomies $\mathrm{Hol}_{q}( S_D(y\oo\gamma)_{(2)})$ and $\mathrm{Hol}_{p}( S_D(y\oo\gamma)_{(1)})$ commute in 
$D(H)^{*\oo E}\cong (H^{op}\oo H^*)^{\oo E}$, and one obtains
$$
(\mathrm{Hol}_{p\circ q}\circ S_D)(y\oo\gamma)=\mathrm{Hol}_{p}( S_D(y\oo\gamma)_{(1)})\cdot' \mathrm{Hol}_{q}( S_D(y\oo\gamma)_{(2)})=\mathrm{Hol}_{(p\circ q)^\inv}(y\oo\gamma).
$$
Suppose that 3. is shown for all paths of length $\leq n$. If  $p\circ q$ is a path of length $n+1$ that satisfies the assumptions in 3., then so do 
$p$ and $q$. By induction hypothesis this implies
\begin{align*}
&\mathrm{Hol}_{p\circ q}(y\oo\gamma)=\Sigma_{i,j}\,\mathrm{Hol}_{p}( \low y 1\oo \alpha^i\low \gamma 1 \alpha^j)\cdot'\mathrm{Hol}_{q}( S(x_j)\low y 2 x_i\oo \low\gamma 2)\\
&=\epsilon(\low y 1)\Sigma_{i,j}\,\epsilon(S(x_j)\low y 2 x_i)\; \mathrm{Hol}_{p}(1\oo \alpha^i\low \gamma 1 \alpha^j)\cdot' \mathrm{Hol}_{q}( 1\oo \low\gamma 2)\\
&=\epsilon(y)\;  \mathrm{Hol}_{p}(1\oo \low\gamma 1)\cdot' \mathrm{Hol}_{q}( 1\oo \low\gamma 2)
=\epsilon(y) \Sigma_{i,j}\, \epsilon(x_iS(x_j)) \mathrm{Hol}_{p}(1\oo \alpha^i\low\gamma 1\alpha^j)\cdot' \mathrm{Hol}_{q}( 1\oo \low\gamma 2)\\
&=\epsilon(y)  \mathrm{Hol}_{p}((1\oo\gamma)_{(1)})\cdot'\mathrm{Hol}_{q}( (1\oo \low\gamma 2)_{(2)})=\epsilon(y)  \mathrm{Hol}_{p\circ q}(1\oo \gamma )
\end{align*}
Suppose 4. holds for all paths of length $\leq n$. If $p\circ q$ is a path of length $n+1$ that satisfies the assumptions of 4., then so do 
$p$ and $q$. With the induction hypothesis one obtains
\begin{align*}
&\mathrm{Hol}_{p\circ q}(y\oo\gamma)= \Sigma_{i,j}\,\mathrm{Hol}_{p}( \low y 1\oo \alpha^i\low \gamma 1 \alpha^j)\cdot' \mathrm{Hol}_{q}( S(x_j)\low y 2 x_i\oo \low\gamma 2)\\
&= \Sigma_{i,j}\,\epsilon(\alpha^i\low \gamma 1\alpha^j)\epsilon(\low\gamma 2)\; \mathrm{Hol}_{p}(\low y 1\oo 1)\cdot' \mathrm{Hol}_{q}( S(x_j)\low y 2 x_i\oo 1)\\
&=\epsilon(\gamma)\; \mathrm{Hol}_{p}(\low y 1\oo 1)\cdot' \mathrm{Hol}_{q}( \low y 2\oo 1)=
\epsilon(\gamma) \Sigma_{i,j}\,\epsilon(\alpha^i\alpha^j) \mathrm{Hol}_{p}(\low y 1\oo 1)\cdot' \mathrm{Hol}_{q}( S(x_j)\low y 2x_i\oo 1)\\
&=\epsilon(\gamma) \; \mathrm{Hol}_{p}(  (y\oo 1)_{(1)})\cdot' \mathrm{Hol}_{q}( (y\oo 1)_{(2)})=\epsilon(\gamma) \mathrm{Hol}_{p\circ q}(y\oo 1).
\end{align*}
To prove 5.  we compute the holonomies using the expressions in Lemma \ref{lem:functor}, the expression for the comultiplication and antipode of $D(H)^*$ in \eqref{eq:ddualmult} and the identities \eqref{eq:basexx}, \eqref{eq:helpanti} and \eqref{eq:helpanti2}. This yields
\begin{align*}
&\mathrm{Hol}_{r(\bar e)\circ r(e)}(y\oo\gamma)= \Sigma_{i,j}\,\epsilon(\alpha^i\low\gamma 1\alpha ^j)\epsilon(S(x_j)\low y 2x_i) (\low y 1\oo 1)_e\cdot' (1\oo\low\gamma 2)_e=(y\oo\gamma)_e,
\intertext{}\\[-10ex]
&\mathrm{Hol}_{l(e)\circ l(\bar e)}(y\oo\gamma)\\
&=\epsilon(\low y 1)\epsilon(\low\gamma 2)  \Sigma_{i,j,r,s,u,v}\,(x_rx_s\oo S(\alpha^s)\alpha^i\low\gamma 2\alpha^j\alpha^r)_e\cdot' (x_uS(x_j)\low y 2x_ix_v\oo S(\alpha^v)\alpha^u)_e\\
&=\Sigma_{i,j,r,s,u,v}\,(x_uS(x_j) y x_ix_vx_rx_s\oo S(\alpha^s)\alpha^i\gamma \alpha^j\alpha^rS(\alpha^v)\alpha^u)_e=(y\oo\gamma)_e,\intertext{}
 \\[-10ex]
&\mathrm{Hol}_{ r(\bar e)^\inv\circ l(e)}(y\oo\gamma)\\
&=\Sigma_{i,j,r,s,u,v}\,\epsilon(S(\alpha^s) S(\alpha^i \low\gamma 1\alpha^j)\alpha^r)\epsilon(S(x_j)\low y 2x_i) (x_rS(\low y 1)x_s\oo 1)_e\cdot' (x_ux_v\oo S(\alpha^v)\low\gamma 2\alpha^u)_e\\
&= \Sigma_{u,v}\,(S(y )\oo 1)_e\cdot' (x_ux_v\oo S(\alpha^v)\gamma \alpha^u)_e=\Sigma_{u,v}\,(x_ux_vS(y)\oo S(\alpha^v)\gamma \alpha^u)_e,\intertext{}
 \\[-10ex]
&\mathrm{Hol}_{ r(e)\circ l(\bar e)^\inv}(y\oo\gamma)\\
&= \epsilon(\low y 1)\Sigma_{i,j,r,s,u,v}\, \epsilon(S(\alpha^s)S(\low\gamma 2)\alpha^r)  (1\oo \alpha^i\low\gamma 1\alpha^j)_e\cdot' ( x_u x_s S(S(x_j)\low y 2 x_i)x_r x_v\oo S(\alpha^v)\alpha^u )_e\\
&= \Sigma_{i,j,u,v}\, (1\oo \alpha^i\gamma \alpha^j)_e\cdot' ( x_u  S(x_i) S(y)x_j  x_v\oo S(\alpha^v)\alpha^u )_e\\
&= \Sigma_{i,j,u,v}\,  ( x_u  S(x_i) S(y)x_j  x_v\oo \alpha^i\gamma \alpha^jS(\alpha^v)\alpha^u )_e
= \Sigma_{u,v}\, ( x_u  x_v S(y)\oo S(\alpha^v)\gamma \alpha^u )_e.\\[-7ex]
\end{align*}
\end{proof}

The identities in Lemma \ref{lem:holprops}, 5. have a direct geometrical interpretation,  shown  in
Figure \ref{fig:edge_comp1}. They state that for the rectangle $R_e$ associated to an edge $e\in E(\Gamma)$ the two paths  from a vertex of $R_e$  to the diagonally opposite vertex have the same holonomy.  In particular, this implies that the holonomies of   two paths in $\gammad$ that involve only the edges $r(e), l(e), r(\bar e), l(\bar e)$ and their inverses agree whenever the paths have the same starting and target vertex. 
Had one defined the holonomy functor   by  expressions  \eqref{eq:phirl}   but  with the multiplication of $\mathcal H(H)^{\oo E}$ instead of $D(H)^{*\oo E}$ in \eqref{eq:assalg},  this result would not hold. 
 
Although the choice of the multiplication  of $D(H)^{*\oo E}$ and the multiplication  $\mathcal H(H)^{\oo E}$
for the multiplication in \eqref{eq:assalg} generally lead to different notions of holonomies, 
there is a class of paths for which the resulting holonomies agree. These are precisely the {\em ribbon paths} introduced in 
 \cite{Ki,BMD}, and their holonomies are the {\em ribbon operators} from  \cite{Ki,BMD}.  
 In a formulation adapted to our notation and conventions,  ribbon paths are defined as follows.

\begin{figure}
\centering
\includegraphics[scale=0.42]{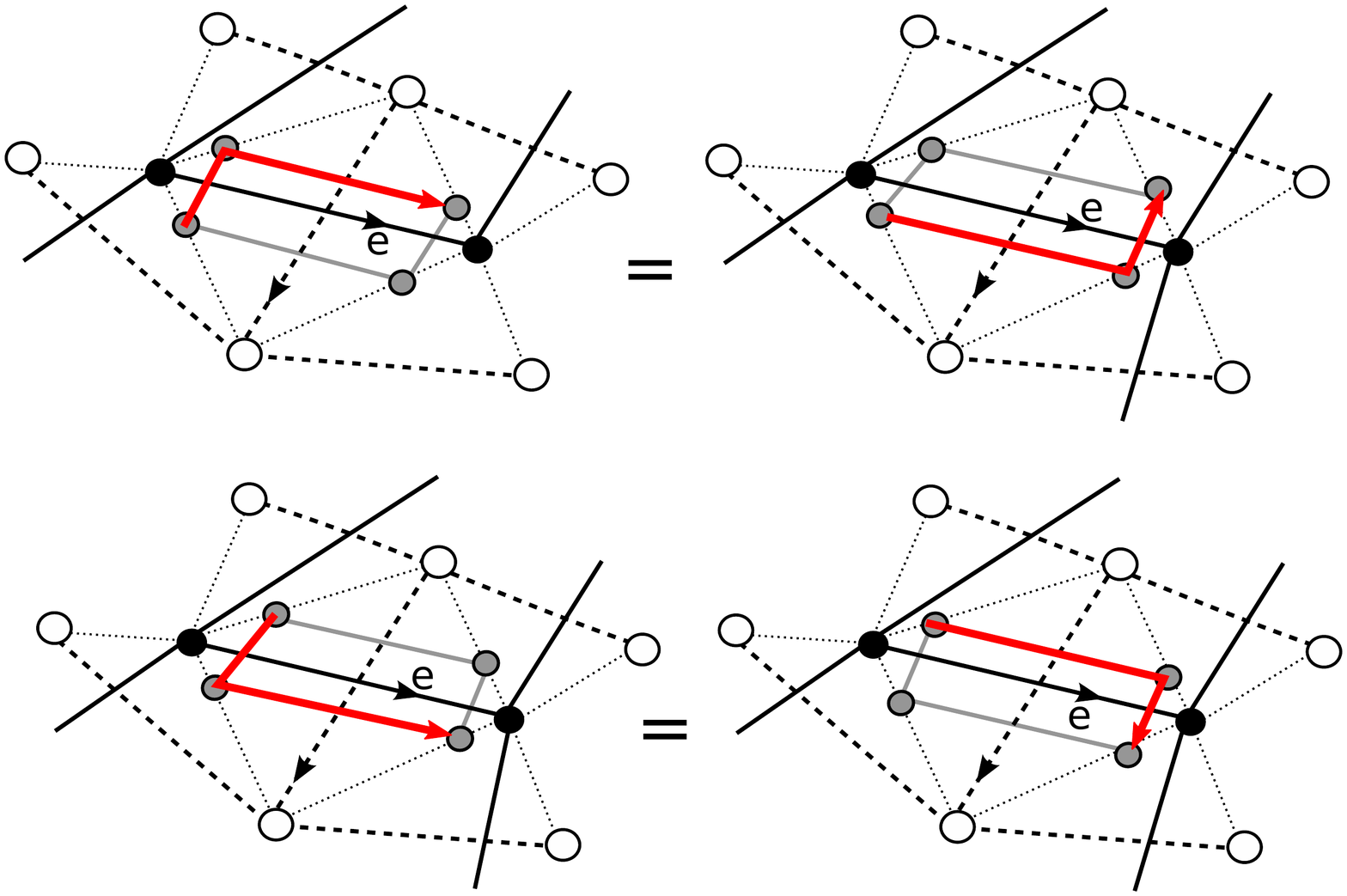}

\vspace{-1cm}
\caption{The relations from  Lemma \ref{lem:holprops}, 5.}
\label{fig:edge_comp1}
\end{figure}

\begin{definition} \label{def:rpath}A path $p\in \mathcal G(\gammad)$  is called a  {\bf ribbon path} if it traverses each edge of $\gammad$ at most once and  for each edge $e\in E(\Gamma)$ traverses either  edges in $\{r(e), l(e)\}$ or edges in $\{r(\bar e), l(\bar e)\}$.
\end{definition}

The name {\em ribbon path} is motivated by the fact that a ribbon  path can be thickened to a ribbon on a surface by associating to each edge $x\in  E(\gammad)$ one of the four triangles in Figure \ref{fig:foursites} \cite{Ki, BMD}, namely the triangle $s(e)s(\bar e)t(\bar e)$  to $l(\bar e)$, the triangle $t(e)t(\bar e)s(\bar e)$   to $r(\bar e)$, the  triangle $t(e)t(\bar e)s(e)$ to $l(e)$ and the triangle $t(e)s(e)s(\bar e)$ to $r(\bar e)$. 
If $p\in \mathcal G(\gammad)$ is a ribbon path, then the triangles for edges in $p$  overlap only on their boundaries and  thicken $p$ to a ribbon, as shown in Figure \ref{fig:graphfusion2}.

\begin{figure}
\centering
\includegraphics[scale=0.45]{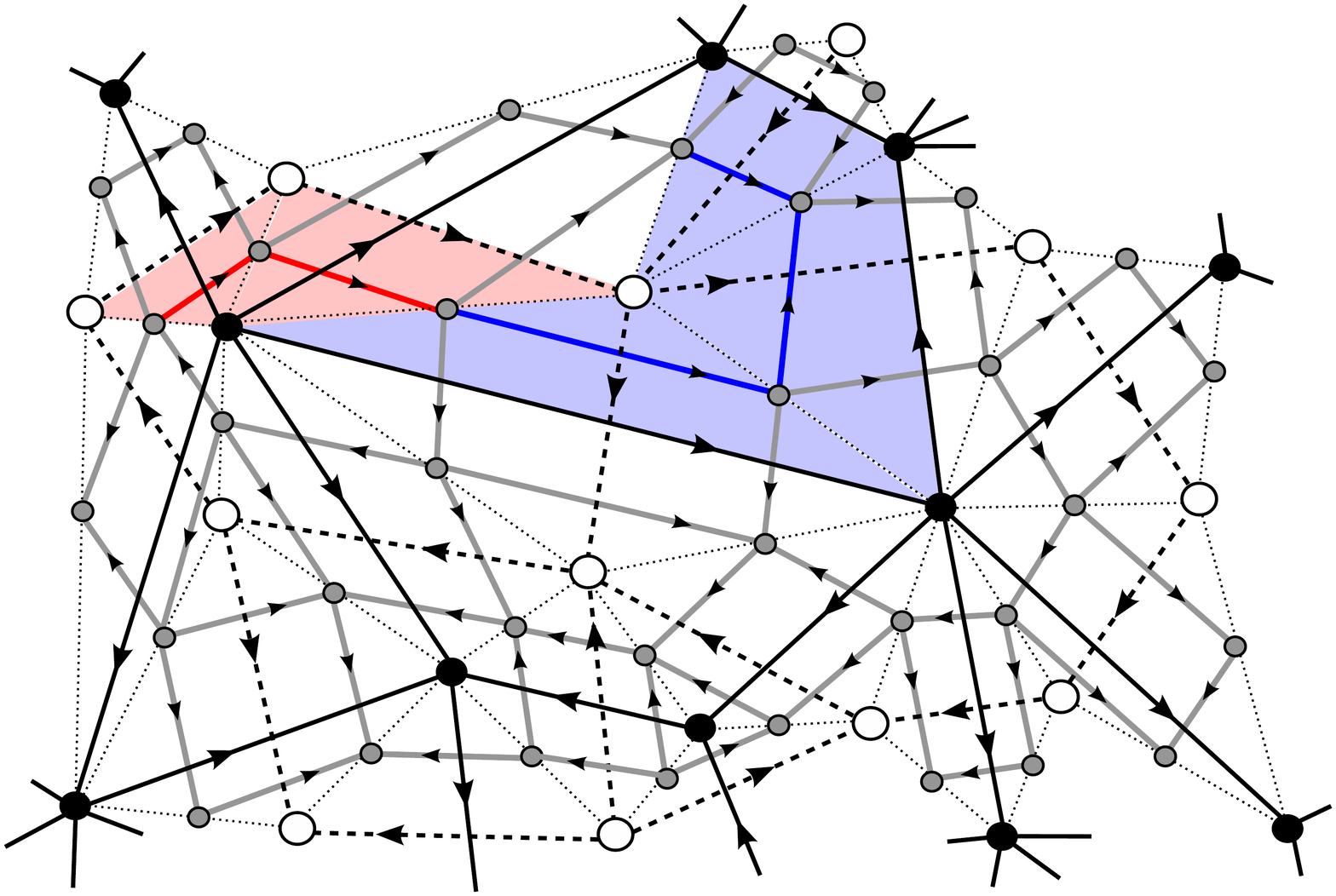}
\caption{A ribbon graph in the thickened graph $\gammad$ and its thickening by triangles. Edges of the form $r(\bar e)^{\pm1}$, $l(\bar e)^{\pm 1}$ in the path and the associated triangles are indicated in red and edges of the form $r(e)^{\pm 1}$, $l(e)^{\pm 1}$ and the associated triangles in blue.}
\label{fig:graphfusion2}
\end{figure}

The {\em ribbon operators} from  \cite{Ki,BMD,BMCA}  associate to each element of $H\oo H^*$ and each ribbon path $p\in \mathcal G( \gammad)$ a linear map $H^{\oo E}\to H^{\oo E}$, which corresponds to a unique element of the algebra $\mathcal H(H)^{\oo E}$ by Lemma \ref{lem:kitoprel}. Hence, ribbon operators can be viewed as a linear map  $H\oo H^* \to (H\oo H^*)^{\oo E}$. A  close inspection of the formulas in \cite{Ki,BMD} show that this map is obtained by   choosing the comultiplication of $D(H)^*$ and the multiplication of $\mathcal H(H)^{\oo E}$  in \eqref{eq:assalg}.  This is made explicit in \cite{BMD} for  the group algebra of a finite group and  generalised implicitly in \cite{BMCA} to finite-dimensional semisimple Hopf algebras.
It turns out that for ribbon paths that satisfy a mild additional assumption  the two notions of holonomies agree and our notion of  holonomy  yields precisely the {\em ribbon operators}.

\begin{lemma} \label{lem:rpath}Let  $p\in \mathcal G(\gammad)$ be  a ribbon path such  that for every edge $e\in E(\Gamma)$ for which $p$ traverses both $r(e),l(e)\in E(\gammad)$ the edge  $r(e)$ is traversed first and for each edge $e\in E(\Gamma)$ for which $p$ traverses both $r(\bar e), l(\bar e)\in E(\gammad)$ the edge $r(\bar e)$ is traversed first.  Then the holonomies of $p$ with respect to the multiplication  of $\mathcal H(H)^{\oo E}$ and  $D(H)^{*\oo E}$ agree.
\end{lemma}
\begin{proof} 
We denote by $\cdot$ the multiplication of $\mathcal H(H)^{\oo E}$ and by $\cdot'$ the multiplication of $D(H)^{*\oo E}$.
Suppose that $p$ is given by a reduced word $p=x_1^{\epsilon_1}\circ ...\circ x_n^{\epsilon_n}$ with $x_i\in E(\gammad)$ and $\epsilon_i\in\{\pm 1\}$. Then the holonomy of $p$ with respect to the multiplication of $D(H)^{*\oo E}$  takes the form
\begin{align*}
\mathrm{Hol}_p(y\oo\gamma)=\mathrm{Hol}_{x_1^{\epsilon_1}}((y\oo\gamma)_{(1)})\cdot' ....\cdot' \mathrm{Hol}_{x_n^{\epsilon_n}}((y\oo\gamma)_{(n)}),
\end{align*}
and the  expression for the holonomy  of $p$ with  respect to the multiplication of $\mathcal H(H)^{\oo E}$ is obtained by replacing $\cdot'$ with $\cdot$ in this expression.
If $p$ is a ribbon path, then for each edge $e\in E(\gammad)$ there are at most two distinct $i,j\in\{1,...,n\}$ with  $x_i,x_j\in\{r(e),l(e), r(\bar e), l(\bar e)\}$, and if there are two of them,  one has
either $\{x_i,x_j\}=\{r(e),l(e)\}$ or $\{x_i,x_j\}=\{r(\bar e), l(\bar e)\}$.  Hence, the contribution of $\mathrm{Hol}_p(y\oo\gamma)$ to the copy of $H\oo H^*$
associated with an edge $e\in E(\Gamma)$ is one of the following
\begin{compactenum}[(i)]
\item  $\mathrm{Hol}_{x}(z\oo\delta)$ with $z\in H$, $\delta\in H^*$ and  $x\in \{r(e), l(e), r(\bar e), l(\bar e)\}$ if  $p$ traverses at most one of the edges $r(e), l(e), r(\bar e), l(\bar e)$, \\[-2ex]
\item $\mathrm{Hol}_{l(e)}(x\oo\beta)\cdot'\mathrm{Hol}_{r(e)}(z\oo\delta)$ with $x,z\in H$, $\beta,\delta\in H^*$ if $p$ traverses both $l(e)$ and $r(e)$, \\[-2ex]
\item $\mathrm{Hol}_{l(\bar e)}(x\oo\beta)\cdot'\mathrm{Hol}_{r(\bar e)}(z\oo\delta)$ with $x,z\in H$, $\beta,\delta\in H^*$ if $p$ traverses both $l(\bar e)$ and $r(\bar e)$.
\end{compactenum}
As the different copies of $H\oo H^*$ on $(H\oo H^*)^{\oo E}$ commute with respect to both $\cdot$ and $\cdot'$, it is sufficient to consider the last two cases.  For (ii) and (iii) one computes with \eqref{eq:phirl},  the expressions for the multiplication, comultiplication and antipode of $D(H)^*$ in \eqref{eq:ddualmult}, the multiplication of $\mathcal H(H)$ in \eqref{eq:hd2} and with the identities \eqref{eq:basexx} and \eqref{eq:helpanti}
\begin{align*}
&\mathrm{Hol}_{l(e)}(x\oo\beta)\cdot'\mathrm{Hol}_{r(e)}(z\oo\delta)=\epsilon(x)\epsilon(z) \Sigma_{i,j}\,(x_i S(x_j)\oo \alpha^j\beta\alpha^i )_e\cdot ' (1\oo\delta)_e\\
&=\epsilon(x)\epsilon(z)\Sigma_{i,j}\, (x_i S(x_j)\oo \alpha^j\beta\alpha^i \delta)_e= \epsilon(x)\epsilon(z)\Sigma_{i,j}\, (x_i S(x_j)\oo \alpha^j\beta\alpha^i )_e\cdot (1\oo\delta)_e\\
&=\mathrm{Hol}_{l(e)}(x\oo\beta)\cdot \mathrm{Hol}_{r(e)}(z\oo\delta)\\[+1ex]
&\mathrm{Hol}_{l(\bar e)}(x\oo\beta)\cdot'\mathrm{Hol}_{r(\bar e)}(z\oo\delta)=\epsilon(\beta)\epsilon(\gamma)\Sigma_{i,j}\,(x_i xS(x_j)\oo\alpha^j\alpha^i)\cdot' (z\oo 1)\\
&=\epsilon(\beta)\epsilon(\delta)\Sigma_{i,j}\,(zx_i xS(x_j)\oo\alpha^j\alpha^i)=\epsilon(\beta)\epsilon(\delta) \Sigma_{i,j}\,  (z_{(3)}x_ixS(x_j)S(z_{(2)})\low z 1\oo \alpha^j\alpha^i)\\
&=\epsilon(\beta)\epsilon(\delta) \Sigma_{i,j}\,\langle \alpha^j_{(1)}\alpha^i_{(1)},\low z 2\rangle  (x_ixS(x_j)\low z 1\oo \alpha^j_{(2)}\alpha^i_{(2)})\\
&=\epsilon(\beta)\epsilon(\delta)\Sigma_{i,j}\,(x_i xS(x_j)\oo\alpha^j\alpha^i)\cdot (z\oo 1)
=\mathrm{Hol}_{l(\bar e)}(x\oo\beta)\cdot\mathrm{Hol}_{r(\bar e)}(z\oo\delta).
\end{align*}
This shows that the holonomies of $p$ with respect to  the multiplications $\cdot$ and $\cdot'$ coincide.
\end{proof}

Although for ribbon paths the holonomies  from  Lemma \ref{lem:functor} yield the ribbon operators from \cite{Ki,BMD, BMCA} and hence coincide with the notion of holonomy in the Kitaev models, the notion of holonomy in  Lemma \ref{lem:functor} is more conceptual.  
 More importantly, it not restricted to ribbon paths but defined 
 for any path $p$ in $\gammad$ and hence more general than  ribbon operators in Kitaev models.
 This will be essential when we relate
the  Kitaev models to a Hopf algebra gauge theory in Section \ref{sec:kithopf}. We will show that the relation between the two models is given by the holonomies  of certain paths  in $\gammad$ that are {\em not} ribbon paths.  The identity in Lemma \ref{lem:holprops}, 5, which holds only if one defines holonomy with the multiplication of $D(H)^{*\oo E}$ in \eqref{eq:assalg}, will be essential in establishing this relation.

\subsection{Vertex and face operators}
    \label{subsec:vertface}

 In this subsection we consider  the holonomies of loops in the thickened ribbon graph $\gammad$ that go clockwise around the vertices and faces of $\Gamma$ and relate them to the vertex and face operators in Kitaev models.
 We then determine the commutation relations of these holonomies and prove  analogues of Lemma \ref{lem:algrelskit} and Lemma \ref{lem:fvcomm}. 
 As these loops are ribbon paths, this is essentially a rederivation of the results on vertex and face operators in \cite{Ki,BMD,BMCA}, and readers familiar with them may skip this subsection.
However, as we use a different notion of holonomy, with different conventions and build on these results  in the following, it is necessary to derive them rigorously.   Another reason to so is to  make the paper self-contained and accessible to other communities.

\begin{definition} \label{def:vertloop} Let $\Gamma$ be  ribbon graph without loops  or multiple edges.
\begin{compactenum}
\item If   $v$ is a ciliated vertex of $\Gamma$ with $n$ incident edges  $e_1,...,e_n$, numbered according to the ordering at $v$  and such that $e_1^{\epsilon_1}$,...,$e_n^{\epsilon_n}$ are incoming, then 
 the  {\bf vertex loop} for $v$  is the path 
$$p_v=r(\bar e_1^{\epsilon_1})\circ\ldots\circ r(\bar e_n^{\epsilon_n})\in \mathcal G(\gammad),
$$

\item If $f=e_1^{\epsilon_1}\circ \ldots\circ e_n^{\epsilon_n}$ is a ciliated face of $\Gamma$, then the associated {\bf face loop} is the path  
$$p_f=r(e_1^{\epsilon_1})\circ \ldots\circ r(e_n^{\epsilon_n})\in \mathcal G(\gammad).$$
\end{compactenum}
\end{definition}

An example of a vertex loop  is given in Figure \ref{fig:vertexpath} and an example of a face loop in Figure \ref{fig:facepath}.
Note  that a vertex  and face loops are in duality. A vertex loop  $p_v$ for a ciliated vertex $v$ of $\Gamma$ can be viewed as a face loop for the associated ciliated  face of  the Poincar\'e dual  $\bar\Gamma$. Similarly,   a face loop $p_f$ for a  ciliated face $f$ of $\Gamma$ corresponds to a vertex loop for the associated ciliated  vertex of $\bar\Gamma$.

\begin{figure}
\centering
\includegraphics[scale=0.45]{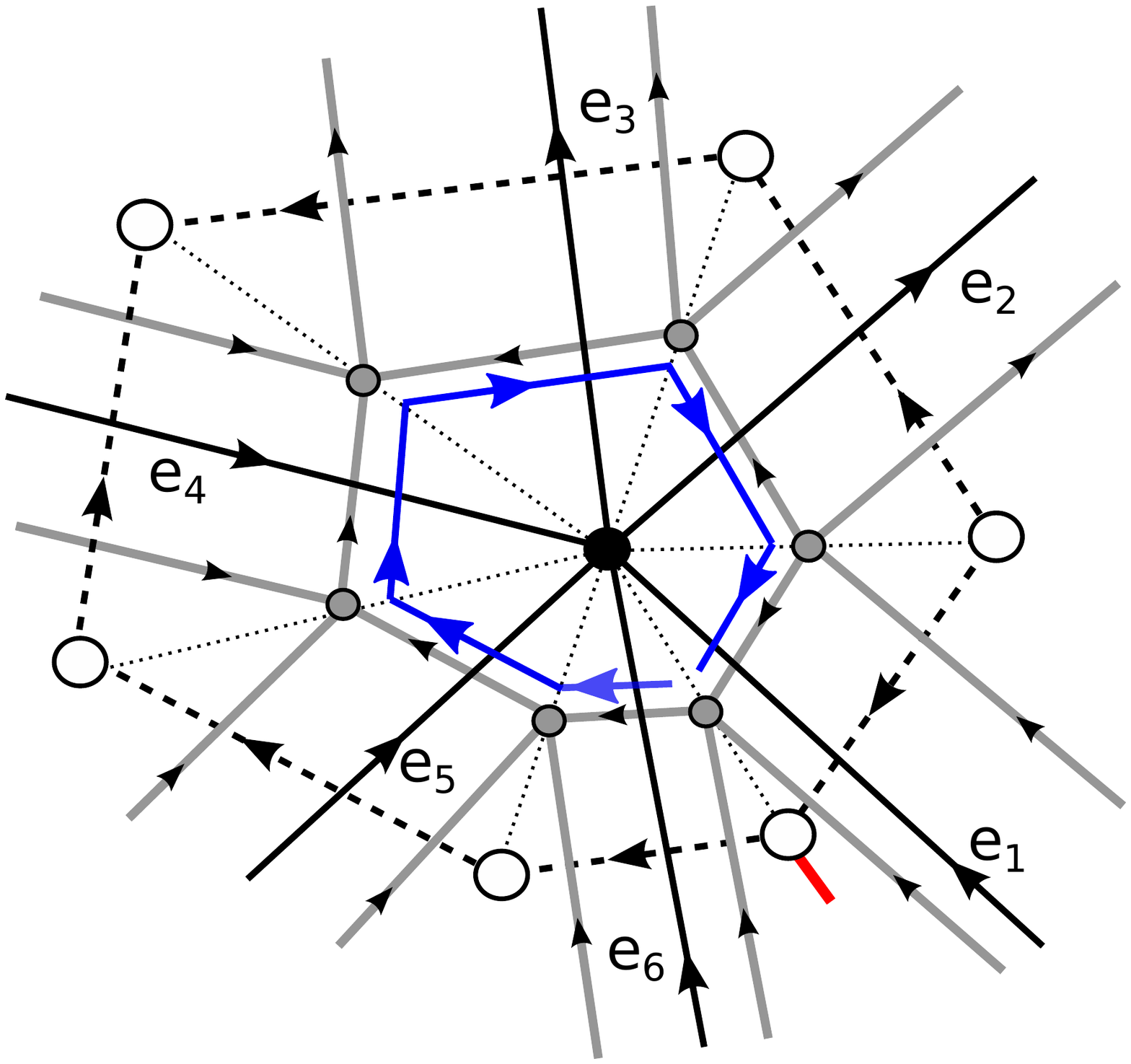}

\vspace{-.5cm}
\caption{The vertex loop
$p_v=r(\bar e_1)\circ l(\bar e_2)^\inv\circ l(\bar e_3)^\inv\circ r(\bar e_4) \circ r(\bar e_5)\circ r(\bar e_6).$}
\label{fig:vertexpath}
\end{figure}

\begin{figure}
\centering
\includegraphics[scale=0.4]{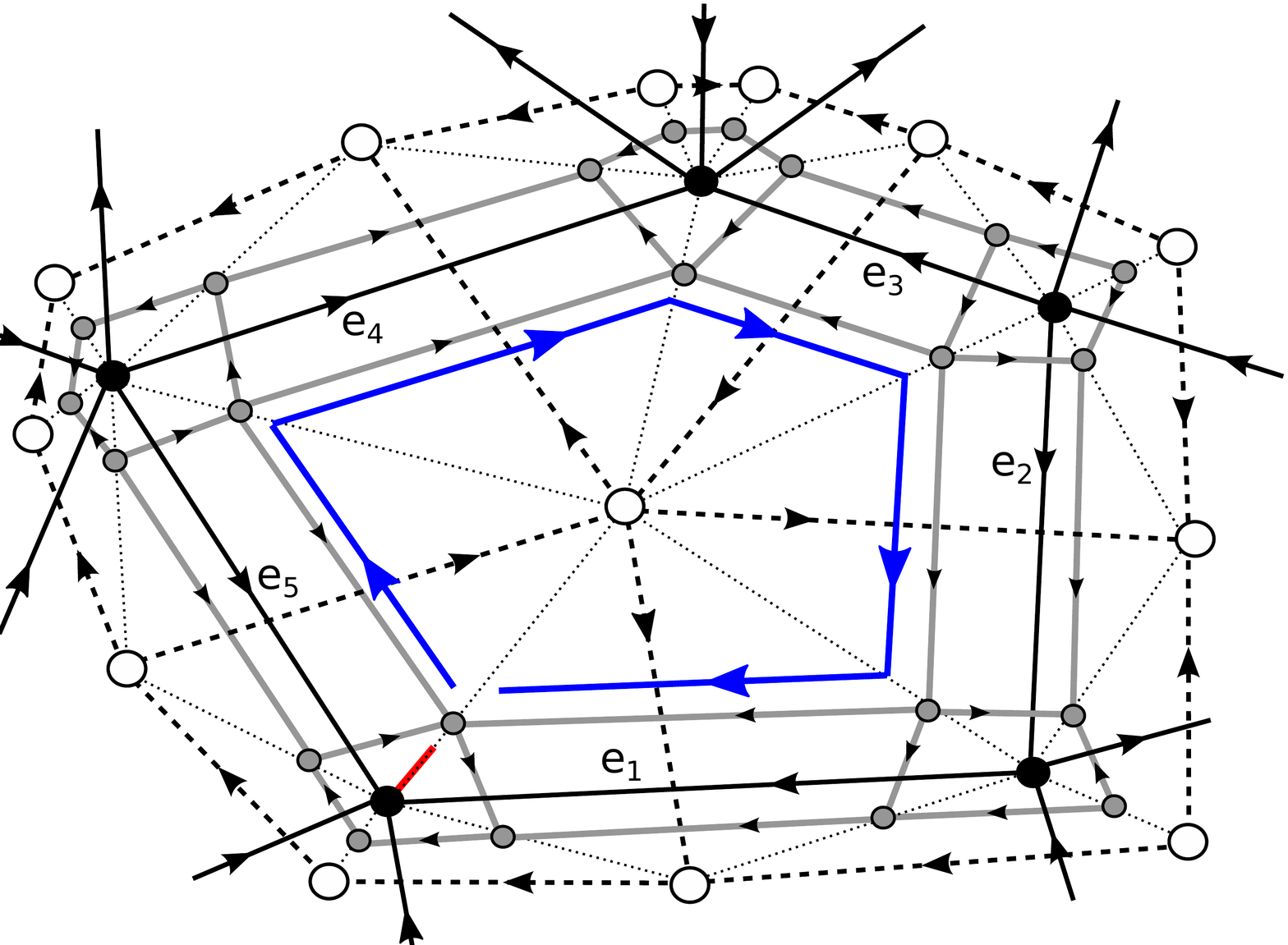}
\caption{The face loop 
$p_f=r(e_1)\circ r(e_2)\circ l(e_3)^\inv\circ r(e_4) \circ l(e_5)^\inv$ for  $f=e_1\circ e_2\circ e_3^\inv\circ e_4\circ e_5^\inv$. }
\label{fig:facepath}
\end{figure}

With the definition of the holonomies from Lemma \ref{lem:functor}, the correspondence between the holonomies of vertex and face loops and vertex and face operators requires mild additional assumptions 
that  can be satisfied for any ciliated ribbon graph $\Gamma$ by reversing the orientations of certain edges.   However, to keep the notation simple and because this is the only case required in the following, we restrict attention to ciliated ribbon graphs $\Gamma$  that satisfy the stronger regularity conditions in Definition \ref{def:regular}. The holonomies of vertex and face loops  are then mapped to the vertex and face operators from Definition \ref{def:vertexoperator} by the
algebra isomorphism $\rho: \mathcal H(H)\to\mathrm{End}_{\mathbb F}(H^{\oo E})$ from Lemma \ref{lem:kitoprel}. 

\begin{lemma}\label{lem:holab}Let $\Gamma$ be a regular ciliated ribbon graph with thickening  $\gammad$.  
For every ciliated vertex $v$  and every ciliated face $f$ of $\Gamma$, the maps
  $\mathrm{Hol}_{p_v}\vert_H: H\to\mathcal H(H)^{\oo E}$  and $\mathrm{Hol}_{p_f}\vert_{H^*}: H^*\to \mathcal H(H)^{\oo E}$ are algebra morphisms and for all $y\in H$ and $\alpha\in H^*$ one has
\begin{align*}
\rho \circ \mathrm{Hol}_{p_v}(y\oo\alpha)=\epsilon(\alpha)\,A_v^y\qquad  \rho \circ \mathrm{Hol}_{p_f}(y\oo\alpha)=\epsilon(y)\, B^\alpha_f.  
\end{align*}    
\end{lemma}    
\begin{proof}  Throughout the proof let  
 $\cdot$ be the multiplication of $\mathcal H(H)^{\oo E}$. If  $v$  is a ciliated vertex with incident 
edges $e_1,...,e_n$, numbered according to the ordering at $v$ and  such that   $e_1^{\epsilon_1}$, ..., $e_n^{\epsilon_n}$ are incoming at $v$, then the
 associated vertex loop $p_v$ from Definition \ref{def:vertloop} is  given by
$
p_v=r(\bar e_1^{\epsilon_1}) \circ \ldots\circ r(\bar e_n^{\epsilon_n}).
$
Similarly, for a ciliated face $f=e_1^{\epsilon_1}\circ \ldots\circ e_n^{\epsilon_n}$  the associated face loop from Definition \ref{def:vertloop} is given by 
$p_f=r(e_1^{\epsilon_1})\circ \ldots\circ r(e_n^{\epsilon_n})$, both subject to  convention  \eqref{eq:orrev}.  
By Lemma \ref{lem:rpath} their holonomies are 
\begin{align}\label{eq:holformula2}
&\mathrm{Hol}_{p_v}(y\oo\alpha)=\epsilon(\alpha)\, \mathrm{Hol}_{r(\bar e_1^{\epsilon_1})}(\low y 1\oo 1)\cdot ... \cdot \mathrm{Hol}_{r(\bar e_n^{\epsilon_n})}(\low y n \oo 1)\\
&\mathrm{Hol}_{p_f}(y\oo\alpha)=\epsilon(y)\, \mathrm{Hol}_{r(e_1^{\epsilon_1})}(1\oo \low \alpha 1)\cdot ...\cdot \mathrm{Hol}_{r( e_n^{\epsilon_n})}(1\oo \low \alpha n ),\nonumber
\end{align}
and with Definition \ref{def:kitdef} and  Lemma \ref{lem:kitoprel}   this yields
\begin{align}\label{eq:phiex2}
\rho\circ \mathrm{Hol}_{r(\bar e_i^{\epsilon_i})}(y\oo 1)
=\begin{cases}
L^{y}_{e_i,+} & \epsilon_i=1\\
L^{S(y)}_{e_i,-} & \epsilon_i=-1,\end{cases}\qquad 
\rho\circ \mathrm{Hol}_{r( e_i^{\epsilon_i})}(1\oo\alpha)
=\begin{cases}
T^{\alpha}_{e_i,+} & \epsilon_i=1\\
T^{S(\alpha)}_{e_i,-} & \epsilon_i=-1.
\end{cases}
\end{align}
By applying the algebra isomorphism $\rho: \mathcal H(H)^{\oo E}\to\mathrm{End}_{\mathbb F}(H^{\oo E})$ from Lemma \ref{lem:kitoprel}  to \eqref{eq:holformula2} and 
inserting \eqref{eq:phiex2} one obtains the operators  $A^y_v$, $B^\alpha_f$ from Definition  \ref{def:vertexoperator}.
 That  $\mathrm{Hol}_{p_v}\vert_H:H\to\mathcal H(H)^{\oo E}$  is an algebra morphism  follows  from \eqref{eq:holformula2}, because one has $\bar e_i\neq \bar e_j$ for $i\neq j$ by the assumptions on $\Gamma$ and hence
the holonomy of $r(\bar e_i^{\epsilon_i})$ commutes with the holonomy of $r(\bar e_j^{\epsilon_j})$ for $i\neq j$.
From  \eqref{eq:phirl} one has  $\mathrm{Hol}_{r(\bar e_i^{\epsilon_i})}(y\oo 1)\cdot \mathrm{Hol}_{r(\bar e_i^{\epsilon_i})}(z\oo 1)=\mathrm{Hol}_{r(\bar e_i^{\epsilon_i})}(yz\oo 1)$ for all $y,z\in H$.
Similarly, the assumptions on $\Gamma$ imply that $e_i\neq e_j$ for all $i\neq j$ and hence
the holonomy of $r({e_i}^{\epsilon_i})$ commutes with the holonomy of $r({e_j}^{\epsilon_j})$ for $i\neq j$. Formula \eqref{eq:phirl}  implies  $\mathrm{Hol}_{r({e_i}^{\epsilon_i})}(1\oo\alpha)\cdot \mathrm{Hol}_{r({e_i}^{\epsilon_i})}(1\oo \beta)=\mathrm{Hol}_{r({e_i}^{\epsilon_i})}(1\oo \alpha\beta)$ for all $\alpha,\beta\in H^*$, and it follows that  $\mathrm{Hol}_{p_f}\vert_{H^*}:H^*\to\mathcal H(H)^{\oo  E}$ is an algebra morphism.
    \end{proof}

We now consider the algebraic properties of the holonomies of vertex and face loops with respect to the algebra structure  of $\mathcal H(H)^{\oo E}$ and prove analogues of Lemma \ref{lem:algrelskit} and Lemma \ref{lem:fvcomm}.
This requires  the following technical lemma.

\begin{lemma}\label{lem:commhelp} Let $\Gamma$ be a regular ciliated ribbon graph and $v$  a vertex with $n$ incoming edges $e_1,...,e_n$, numbered according to the ordering at $v$. Then the  holonomies of  the path $l(e_{i+1})^\inv\circ r(e_i)$ commute with the holonomies of the paths  $r(e_{i+1})$, $l(e_i)$ and  $r(\bar e_i)\circ r(\bar e_{i+1})$.
\end{lemma}
\begin{proof}  We denote by $\cdot$ the multiplication of $\mathcal H(H)^{\oo E}$ and by $\cdot'$ the multiplication of $D(H)^{*\oo E}$. 
That $\mathrm{Hol}_{l(e_{i+1}^\inv)\circ r(e_i)}(z\oo\beta)$ commutes with $\mathrm{Hol}_{r(e_{i+1})}(y\oo\alpha)$ and with $\mathrm{Hol}_{l(e_{i})}(y\oo\alpha)$ for all $y,z\in H$ and $\alpha,\beta\in H^*$ follows from the definition of the holonomy functor and the fact that for all $e\in E(\Gamma)$ the holonomies of $r(e)$ and $l(e)$ commute, which  is a consequence of their definition  in \eqref{eq:phirl}   and the identities in \eqref{eq:hdcommutelr}.
To prove that the holonomy of  $l(e_{i+1}^\inv)\circ r(e_i)$ commutes  with the holonomy of $r(\bar e_i)\circ r(\bar e_{i+1})$, we compute with the multiplication law \eqref{eq:hd2} of the Heisenberg double
\begin{align*}
&\mathrm{Hol}_{l(e_{i+1})^\inv\circ r(e_{i})}(z\oo\beta)\cdot \mathrm{Hol}_{r(\bar e_i)\circ  r(\bar e_{i+1})}(y\oo\alpha)\\
&=\epsilon(z)\epsilon(\alpha)\; (x_kx_l\oo S(\alpha^l)S(\low\beta 1)\alpha^k)_{e_{i+1}}\cdot (1\oo\low\beta 2)_{e_{i}}\cdot (\low y 1\oo 1)_{e_i}\cdot (\low y 2\oo 1)_{e_{i+1}}\\
&=\epsilon(z)\epsilon(\alpha)\; \langle S(\alpha^l_{(2)})S(\low \beta 2)\alpha^k_{(1)},\low y 4\rangle\langle\low\beta 3,\low y 2\rangle\;
(x_kx_l\low y 3\oo S(\alpha^l_{(1)})S(\low\beta 1)\alpha^k_{(2)})_{e_{i+1}}\cdot(\low y 2\oo\low\beta 4)_{e_i}
\\
&=\epsilon(z)\epsilon(\alpha)\; \langle S(\low \beta 2),\low y 5\rangle\langle\low\beta 3,\low y 2\rangle\;
(\low y 6x_kx_lS(\low y 4)\low y 3\oo S(\alpha^l)S(\low\beta 1)\alpha^k)_{e_{i+1}}\cdot(\low y 2\oo\low\beta 4)_{e_i}
\\
&=\epsilon(z)\epsilon(\alpha)\; \langle S(\low \beta 2),\low y 3\rangle\langle\low\beta 3,\low y 2\rangle\;
(\low y 4x_kx_l\oo S(\alpha^l)S(\low\beta 1)\alpha^k)_{e_{i+1}}\cdot(\low y 2\oo\low\beta 4)_{e_i}
\\
&=\epsilon(z)\epsilon(\alpha)\; 
(\low y 2x_kx_l\oo S(\alpha^l)S(\low\beta 1)\alpha^k)_{e_{i+1}}\cdot(\low y 1\oo\low\beta 2)_{e_i}
\\
&=\epsilon(z)\epsilon(\alpha)\;  (\low y 1\oo 1)_{e_i}\cdot (\low y 2\oo 1)_{e_{i+1}}\cdot (x_kx_l\oo S(\alpha^l)S(\low\beta 1)\alpha^k)_{e_{i+1}}\cdot (1\oo\low\beta 2)_{e_{i+1}}\\
&=\mathrm{Hol}_{r(\bar e_i)\circ  r(\bar e_{i+1})}(y\oo\alpha)\cdot \mathrm{Hol}_{l(e_{i+1})^\inv\circ r(e_{i})}(z\oo\beta).
\end{align*}
Note that we replaced the multiplication of $D(H)^{*\oo E}$  by the multiplication of $\mathcal H(H)^{\oo E}$ in the expressions for the holonomies of $l(e_{i+1}^\inv)\circ r(e_i)$ and $r(\bar e_i)\circ r(\bar e_{i+1})$ because the paths  $l(e_{i+1}^\inv)\circ r(e_i)$ and $r(\bar e_i)\circ r(\bar e_{i+1})$ satisfy the assumptions of Lemma \ref{lem:rpath}.
\end{proof}

\begin{figure}
\centering
\includegraphics[scale=0.33]{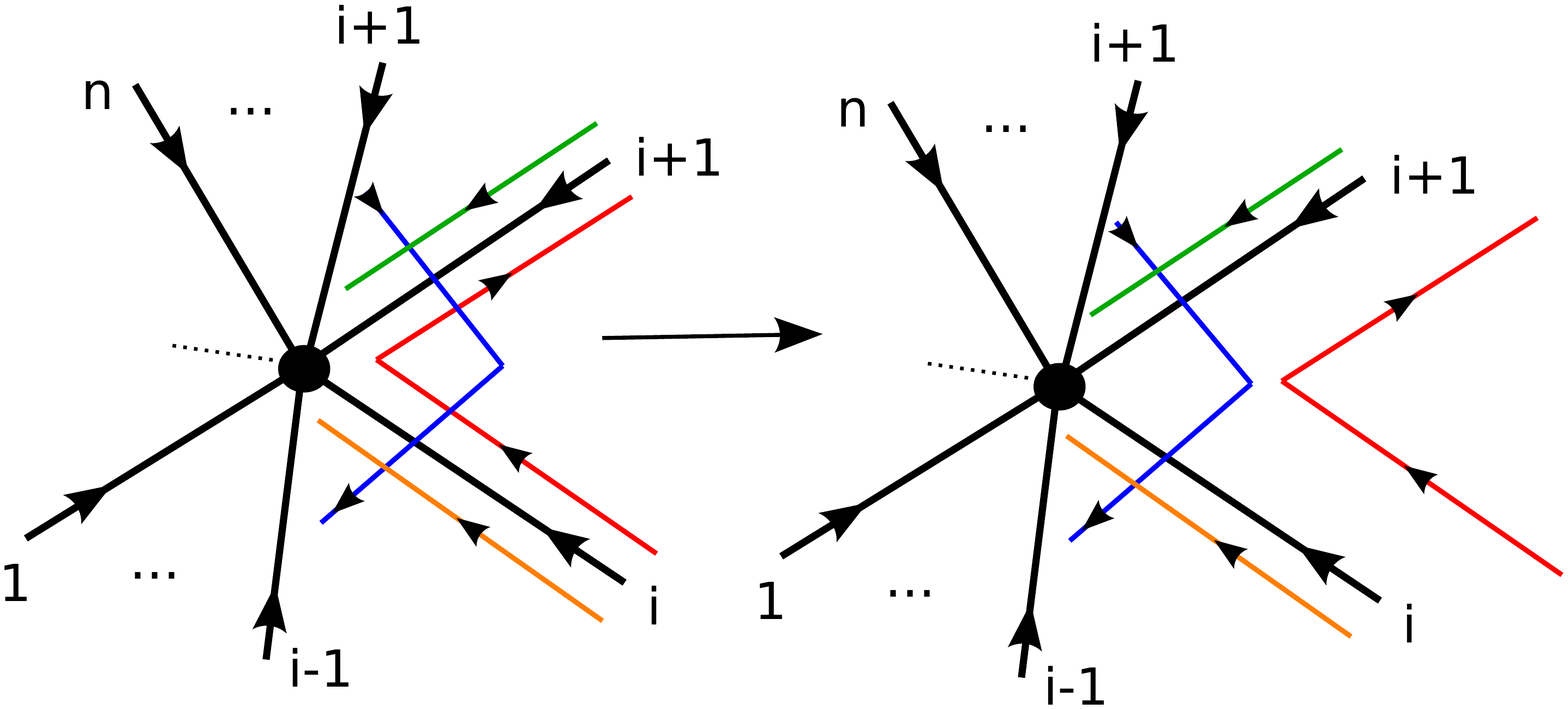}

\vspace{-1.7cm}
\caption{The paths from Lemma \ref{lem:commhelp} and the Reidemeister II move at the vertex:\newline $l(e_{i+1})^\inv\circ r(e_i)$ (red), $r(\bar e_i)\circ r(\bar e_{i+1})$ (blue), $l(e_i)$ (orange) and $r(e_{i+1})$ (green).}
\label{fig:reidemeister}
\end{figure}

The identities in Lemma \ref{lem:commhelp} have a natural geometrical interpretation. They state that the holonomies
of non-intersecting paths in $\gammad$ commute and that an analogue of the Reidemeister II  move can be used to remove intersection points of the paths $l(e_{i+1}^\inv)\circ r(e_i)$ and $r(\bar e_i)\circ r(\bar e_{i+1})$. 

\begin{lemma}[see Lemma \ref{lem:fvcomm}]\label{lem:vertfacecomm}  Let $\Gamma$ be a regular ciliated ribbon graph,   $v,v'$  be two distinct vertices  of $\Gamma$ and $f, f'$ two ciliated faces  that start and end at different cilia of $\Gamma$.  Then:
\begin{compactenum}
\item The holonomies of the vertex loops $p_v$ and $p_{v'}$ commute,
\item The holonomies of the face loops $p_{f}$ and $p_{f'}$ commute,
\item If the cilium of $f$ does not coincide with the cilium at $v$,  the holonomies of $p_v$ and $p_f$ commute. 
\end{compactenum}        
\end{lemma}
\begin{proof} $\quad$\newline
1. The vertex loop $p_v$ is composed of the paths $r(\bar g)$ for edges $g\in E(\Gamma)$ with $\ta(g)=v$ and of paths $l(\bar g)^\inv$ for edges $g\in E(\Gamma)$ with $\st(g)=v$.  If $g\in E(\Gamma)$ is incident at $v$ but not at $v'$  then the holonomies of  $r(\bar g)$ and $l(\bar g)^\inv$ commute with the holonomy of $p_{v'}$.  
If $g\in E(\Gamma)$ is incident at both $v$ and $v'$, we can suppose without restriction of generality that $\ta(g)=v$ and $\st(g)=v'$. Then $p_v$ contains only  factors  of the form $r(\bar g)$ and $r(\bar h)$, $l(\bar h)^\inv$ for edges $h\neq g$, and $p_{v'}$ contains only   factors of the form $l(\bar g)^\inv$ and  $r(\bar h)$, $l(\bar h)^\inv$ for edges $h\neq g$. The holonomies of $r(\bar g)$ and $l(\bar g)^\inv$ commute with the holonomies of $r(\bar h)$, $l(\bar h)^\inv$ for  $h\neq g$, and by  \eqref{eq:hdcommutelr}  they commute with each other.  Hence, the holonomies of  $p_v$ and $p_{v'}$ commute. 

2. If $f,f'$ are ciliated faces of $\Gamma$ that start and end at different cilia of $\Gamma$, they are non-equivalent. For any edge $g\in E(\Gamma)$ 
that is traversed by both $f$ and $f'$ either  $p_f$ contains  a factor $r(g)$ and 
 $p_{f'}$ contains a factor  $l(g)$ or vice versa.   The holonomies of $r(g)$ and $l(g)$ commute by \eqref{eq:phirl} and \eqref{eq:hdcommutelr}. As the holonomies of $r(g)$ and $l(g)$ also commute with the holonomies of 
$r(h)$ and $l(h)$ for all  edges $h\neq g$, the holonomies of $p_f$ and $p'_f$ commute. 

3. As $f$ starts and ends at a cilium and $\Gamma$ is regular, the  face loop $p_f$ can be decomposed into paths $q$ that do not traverse any edges incident at $v$  and into paths of the form $r(\bar f^{\pm 1})\circ r(\bar g^{\pm 1})$, where $f$, $g$ are adjacent edges at $v$ with $g<f$. 
The holonomies of the former commute with  the holonomy of $p_v$ by definition and the holonomies of the latter  by Lemma 
 \ref{lem:commhelp} because $\mathrm{Hol}_{r(\bar f^\inv)}=\mathrm{Hol}_{l(\bar f)^\inv}=S_D\circ \mathrm{Hol}_{r(\bar f)}$, $\mathrm{Hol}_{r( f^\inv)}=\mathrm{Hol}_{l(f)^\inv}=S_D\circ \mathrm{Hol}_{r(f)}$ and  $S_D: \mathcal H(H)\to \mathcal H(H)$ is an algebra homomorphism by Lemma \ref{lem:antipodehd}. Hence, the holonomies of $p_v$ and $p_f$ commute. 
 \end{proof}

\begin{lemma} [see Lemma \ref{lem:algrelskit}]\label{lem:helpvertface} Let $\Gamma$ be a regular ciliated ribbon graph.  Denote for each vertex $v$ by   $f(v)$ the ciliated  face that starts and ends at the cilium at $v$.  
Then one obtains an algebra homomorphism
\begin{align}\label{eq:tau}
\tau: D(H)^{\oo V}\to \mathcal H(H)^{\oo E},\quad (\delta\oo z)_v\mapsto \mathrm{Hol}_{p_{f(v)}}(1\oo \delta)\cdot\mathrm{Hol}_{p_v}(z\oo 1).
\end{align}
\end{lemma}
\begin{proof} 
By Lemma \ref{lem:helpvertface} the holonomies of $p_v$  and $p_{f(v)}$ commute with the holonomies of $p_w$ and $p_{f(w)}$ for all vertices $w\neq v$. Hence it is sufficient to show that  for each vertex $v$ the linear map
\begin{align}\label{eq:tauv}
\tau_v=\tau\circ \iota_v: D(H)\to \mathcal H(H)^{\oo E},\quad \delta\oo z\mapsto \mathrm{Hol}_{p_{f(v)}}(1\oo \delta)\cdot\mathrm{Hol}_{p_v}(z\oo 1)
\end{align}
is an algebra morphism. For this, let 
$v$ be a vertex with $n$ incident edges $e_1,...,e_n$, numbered according to the ordering at $v$. As $\Gamma$ is loop-free and  the map $S_D:\mathcal H(H)\to\mathcal H(H)$ is an algebra isomorphism by Lemma \ref{lem:antipodehd}, we can suppose without loss of generality that all edges are incoming. Then the vertex loop $p_v$ is given by $p_v=r(\bar e_1)\circ\ldots\circ r(\bar e_n)$ and
 the associated  face loop is of the form 
$p_{f(v)}=r(e_n)\circ q\circ l(e_1)^\inv$, where $q$ is a path that turns maximally right at each vertex, traverses each edge at most once  and does not traverse any cilia. 
 This implies that $q$ can be decomposed into paths that do not traverse any edges incident at $v$ and paths of the form $r(g^{\pm 1})\circ r(h^{\pm 1})$ where $g,h$ are adjacent edges at $v$ with $h<g$. The holonomies of the former commute with the holonomy of $p_v$ by \eqref{eq:phirl} and the holonomies of the  latter  by Lemma \ref{lem:commhelp}. This implies that the holonomy of $q$ commutes with the holonomy of $p_v$. As  the maps $\mathrm{Hol}_{p_v}\vert_H: H\to \mathcal H(H)^{\oo E}$ and $\mathrm{Hol}_{p_{f(v)}}\vert_{H^*}: H^*\to\mathcal H(H)^{\oo E}$ are algebra homomorphisms by Lemma \ref{lem:holab},  we  obtain
\begin{align*}
&\mathrm{Hol}_{p_{f(v)}}(1\oo\delta)\cdot \mathrm{Hol}_{p_v}(1\oo z)=(1\oo\low\delta 1)_{e_n}\cdot \mathrm{Hol}_{q}(1\oo\low \delta 2)\cdot (x_kx_l\oo S(\alpha^l) S(\low\delta 3)\alpha^k)_{e_1} \cdot \mathrm{Hol}_{p_v}( z\oo 1)\\
&=\langle S(\low \delta 4), \low z 1\rangle \; (1\oo\low\delta 1)_{e_n}\cdot \mathrm{Hol}_{q}(1\oo\low \delta 2)\cdot \mathrm{Hol}_{p_v}(1\oo \low z 2) \cdot (x_kx_l\oo S(\alpha^l) S(\low\delta 3)\alpha^k)_{e_1} \\
&=\langle S(\low \delta 5), \low z 1\rangle \langle \low\delta 1, \low z 3\rangle \; \mathrm{Hol}_{p_v}(1\oo \low z 2) \cdot (1\oo\low\delta 2)_{e_n}\cdot \mathrm{Hol}_{q}(1\oo\low \delta 3) \cdot (x_kx_l\oo S(\alpha^l) S(\low\delta 4)\alpha^k)_{e_1} \\
&=\langle S(\low \delta 3), \low z 1\rangle \langle \low\delta 1, \low z 3\rangle \; \mathrm{Hol}_{p_v}(1\oo \low z 2) \cdot \mathrm{Hol}_{p_{f(v)}}(1\oo\low\delta 2).\end{align*}
This implies for all $y,z\in H$ and $\gamma,\delta\in H^*$
\begin{align*}
\mathrm{Hol}_{p_v}(1\oo  z ) \cdot \mathrm{Hol}_{p_{f(v)}}(1\oo\delta)=\langle \low \delta 3, \low z 1\rangle \langle \low\delta 1, S(\low z 3)\rangle \; \mathrm{Hol}_{p_{f(v)}}(1\oo\low \delta 2)\cdot \mathrm{Hol}_{p_v}(1\oo \low z 2).
\end{align*}
A comparison with the multiplication of $D(H)$ in \eqref{eq:dmult} then proves the claim.
\end{proof}

\section{Gauge symmetries and flatness in  Kitaev models}
\label{subsec:gsymmflat}

\subsection{Gauge symmetries and gauge invariance}
\label{subsec:gaugesym}

As explained in  Section \ref{sec:gtheory}, the essential algebraic structure  in a Hopf algebra gauge theory is the {\em module algebra} of functions over the {\em Hopf algebra} of gauge transformations. 
In this section we show that the Hopf algebra of gauge transformations in a Kitaev model  is the $V$-fold tensor product $D(H)^{\oo V}$ and the algebra $\mathcal H(H)^{op\oo E}$ of triangle operators can be given the structure of a module algebra over this Hopf algebra. We can thus interpret  $D(H)^{\oo V}$ as a {\em gauge symmetry} of the Kitaev model  
similar to the gauge symmetries in a Hopf algebra gauge theory and  obtain an associated subalgebra $\mathcal H(H)^{op\oo E}_{inv}$ of invariants or {\em gauge invariant observables}. 

In Section \ref{subsec:flatkit} we investigate the notion of curvature in  Kitaev models and show that only those faces of $\gammad$ that correspond to vertices and faces of $\Gamma$ give rise to curvatures. We then construct a subalgebra  $\mathcal H(H)^{op\oo E}_{flat}$ that can be viewed as the algebra of gauge invariant functions of  flat gauge fields and  acts on the protected space.

The $D(H)^{\oo V}$- right module structure on Kitaev's triangle operator algebra  is induced by the algebra homomorphism \eqref{eq:tau}, i.~e.~by the holonomies of the vertex and face loop based at a cilium of $\Gamma$. 

\begin{theorem} \label{lem:hdmodule} Let $\Gamma$ be a regular ciliated ribbon graph. Denote for each  vertex $v$ of $\Gamma$ by $f(v)$ the ciliated face that starts and ends at the cilium at $v$. 
Then  $\lhd:  \mathcal H(H)^{\oo E}\oo D(H)^{\oo V}\to  \mathcal H(H)^{\oo E }$  
\begin{align}
X \lhd (\delta\oo z)_{v}&= \mathrm{Hol}_{p_v}(S(\low z 2)\oo 1)\cdot \mathrm{Hol}_{p_{f(v)}}(1\oo S(\low\delta 1))\cdot 
 X
 \cdot \mathrm{Hol}_{p_{f(v)}}(1\oo \low\delta 2)\cdot \mathrm{Hol}_{p_v}(\low z 1\oo 1)\nonumber\\
 &=(\tau_v\circ S)  ((\delta\oo z)_{(2)})\cdot X\cdot \tau_v((\delta\oo z)_{(1)})
\end{align}
defines a $D(H)^{\oo V}$-right module algebra structure on $\mathcal H(H)^{op\oo E}$.
\end{theorem}
\begin{proof}  It follows by a direct computation that for any Hopf algebra $K$, any algebra $A$  and any  algebra homomorphism $\tau: K\to A$ the linear map $\lhd: A\oo K\to A$, $a\lhd k=\tau(S(\low k 2))\cdot a\cdot \tau(\low k 1)$  
equips $A^{op}$ with the structure of a $K$-right module algebra. 
As the linear map $\tau: D(H)^{\oo V}\to \mathcal H(H)^{\oo E}$ is an algebra homomorphism by Lemma \ref{lem:helpvertface}, the claim follows.
\end{proof}

As the $D(H)^{\oo V}$-right module structure  from
Theorem \ref{lem:hdmodule} gives $\mathcal H(H)^{op\oo E}$ the structure of a $D(H)^{\oo V}$-module algebra, Lemma \ref{lem:haarproj} defines a projector on its subalgebra  of invariants. As $H$, $H^*$ and $D(H)$ are semisimple and $\text{char}(\mathbb F)=0$,  all  tensor products of these Hopf  algebras over $\mathbb F$ are semisimple as well and hence equipped with Haar integrals. 
If  $\ell \in H$ and $\eta\in H^*$ denote the  Haar integrals of $H$ and $H^*$, then  $\eta\oo\ell\in H^*\oo H$ is the  Haar integral for $D(H)$ and $(\eta\oo \ell)^{\oo V }$ the Haar integral for $D(H)^{\oo V}$.  By inserting the latter into the $D(H)^{\oo V}$-module structure from Theorem \ref{lem:hdmodule} one obtains a projector on the gauge invariant subalgebra.

\begin{lemma}\label{lem:hdproj} Let $\Gamma$ be a  regular ciliated ribbon graph and $H$ a finite-dimensional semisimple Hopf algebra. Consider the  $D(H)^{\oo V}$-right module algebra structure on $\mathcal H(H)^{op\oo E}$ from Theorem \ref{lem:hdmodule}. 
\begin{compactenum}
\item The linear map 
$Q_{inv}: \mathcal H(H)^{\oo E}\to\mathcal H(H)^{\oo E}$,  $X\mapsto X\lhd (\eta\oo\ell)^{\oo V}$
is a projector.\\[-2ex]
\item Its image $\mathcal H(H)^{\oo E}_{inv}=Q_{inv}(\mathcal H(H)^{\oo E})$ is a subalgebra of $\mathcal H(H)^{\oo E}$. \\[-2ex]

\item For all vertices $v$ the maps   $Q_{inv}\circ \mathrm{Hol}_{p_v}$ and $Q_{inv}\circ \mathrm{Hol}_{p_{f(v)}}$ take values in the centre of $\mathcal H(H)^{\oo E}_{inv}$.
\end{compactenum}
\end{lemma}
\begin{proof} That $Q_{inv}$ is a projector and $Q_{inv}(\mathcal H(H)^{\oo E})$  a subalgebra  of $\mathcal H(H)^{\oo E}$  follows with Lemma \ref{lem:haarproj} and Lemma \ref{lem:project}, because $(\eta\oo \ell)^{\oo V}$ is a Haar integral for $D(H)^{\oo V}$ and $\mathcal H(H)^{op\oo E}$ a $D(H)^{\oo V}$-right module algebra by Theorem \ref{lem:hdmodule}.
To prove the third claim, note first that $(\eta\oo 1)\cdot (1\oo\ell)=(1\oo \ell)\cdot (\eta\oo 1)$ in $D(H)$, which 
follows directly from formula \eqref{eq:dmult} for the multiplication of $D(H)$ and the cyclic invariance of 
$\Delta^{(3)}(\ell)$ and $\Delta^{(3)}(\eta)$.
This implies for all vertices $v$ of $\Gamma$
\begin{align}\label{eq:holgauge}
\mathrm{Hol}_{p_v}(y\oo\gamma)\lhd_v(\eta\oo\ell)=&\epsilon(\gamma)\,  \mathrm{Hol}_{p_v}(S(\low\ell 2)y \low\ell 1\oo 1) \lhd (\eta\oo 1)_v\\
\mathrm{Hol}_{p_{f(v)}}(y\oo\gamma)\lhd_v(\eta\oo\ell)=&\epsilon(y)\,  \mathrm{Hol}_{p_{f(v)}}(1\oo S(\low\eta 1)\gamma\low\eta 2) \lhd (1\oo \ell)_v. \nonumber
\end{align}
As $S(\low\ell 1)\oo\low\ell 2$ and $S(\low\eta 1)\oo\low\eta 2$ are separability idempotents for $H$ and $H^*$ and  $(1\oo\ell)$ and $(\eta\oo 1)$ commute in $D(H)$  we obtain from  the $D(H)^{\oo V }$-right module structure of $\mathcal H(H)^{\oo E}$ 
\begin{align*}
&(\mathrm{Hol}_{p_v}(y\oo\gamma)\lhd(\eta\oo\ell)_v)\cdot (X\lhd (\eta'\oo\ell')_v)\\
&=\epsilon(\gamma)\,  \tau_v( (S(\low\eta 1)\oo 1) \cdot (1\oo S(\low\ell 2)y \low\ell 1) \cdot (\low\eta 2 S(\low {\eta'} 1)\oo 1))\cdot  (X\lhd (1\oo \ell')_v) \cdot \tau_v( \low {\eta'} 2\oo 1)\\
&=\epsilon(\gamma)\,  \tau_v( (S(\low\eta 1)\oo 1) \cdot (1\oo S(\low\ell 2)y \low\ell 1) \cdot ( S(\low {\eta'} 1)\oo 1))\cdot  (X\lhd (\ell'\oo 1)_v) \cdot \tau_v( \low {\eta'} 2\low\eta 2\oo 1)\\
 &=\epsilon(\gamma)\,  \tau_v( (S(\low\eta 1)\oo 1) \cdot (1\oo S(\low\ell 2)y \low\ell 1 S(\low {\ell'} 2))) \cdot (X\lhd (\eta'\oo 1)_v) \cdot \tau_v((1\oo \low {\ell'} 1 ) \cdot ( \low\eta 2\oo 1))\\
 &=\epsilon(\gamma)\,  \tau_v( (S(\low\eta 1)\oo 1)
\cdot (1\oo S(\low {\ell'} 2))) \cdot (X\lhd (\eta'\oo 1)_v) \cdot \tau_v(1\oo \low {\ell'} 1  S(\low\ell 2)y \low\ell 1) \cdot  \tau_v( \low\eta 2\oo 1)\\
 &=\epsilon(\gamma)\, \tau_v( S(\low\eta 1S(\low {\eta'} 1)\oo 1)  \cdot (X\lhd (1\oo \ell')_v)  \cdot \tau_v( (\low {\eta'} 2\oo 1) \cdot (1\oo S(\low\ell 2)y \low\ell 1)
\cdot ( \low\eta 2\oo 1))\\
 &=\epsilon(\gamma)\, \tau_v(S(\low {\eta'} 1)\oo 1)  \cdot (X\lhd (1\oo \ell')_v)  \cdot \tau_v(( \low {\eta'} 2S(\low\eta 1)\oo 1) \cdot   (1\oo S(\low\ell 2)y \low\ell 1)
\cdot ( \low\eta 2\oo 1))\\
&=(X\lhd(\eta'\oo\ell')_v)\cdot (\mathrm{Hol}_{p_v}(y\oo\gamma)\lhd(\eta\oo\ell)_v)
\end{align*}
for all $y\in H$, $\gamma\in H^*$ and $X\in \mathcal H(H)^{\oo E}$.
As the different copies of $D(H)$ in $D(H)^{\oo V }$ commute, this proves the third claim for the paths $p_v$.
The proof for the paths $p_{f(v)}$  is analogous.
\end{proof}

Theorem \ref{lem:hdmodule} and Lemma \ref{lem:hdproj} will  allow us  later to interpret the representations of the Drinfeld double $D(H)$ associated with each vertex of $\Gamma$ as gauge transformations and to view the subalgebra $\mathcal H(H)^{\oo E}_{inv}\subset \mathcal H(H)^{\oo E}$ as the algebra of gauge invariant functions of a Hopf algebra gauge theory on $\Gamma$ with values in $D(H)$. 

\subsection{Flatness and the protected space}
\label{subsec:flatkit}

It remains to introduce a notion of curvature into the Kitaev models that is an analogue of curvature in a Hopf algebra gauge theory.
By Definition \ref{def:hol},  curvatures in a Hopf algebra gauge theory on  $\Gamma$ are given by the holonomies of ciliated faces of $\Gamma$. 
This suggests that curvatures in the Kitaev models should be given by the holonomies of faces of the  thickened ribbon graph $\gammad$.

As explained in Section \ref{subsec:thickening}, each face of $\gammad$  corresponds either to a unique edge, to a unique face or to a unique vertex  of $\Gamma$. The face of $\gammad$ corresponding to an edge $e$ of $\Gamma$ is the rectangle $R_e$ in Figure \ref{fig:foursites}, and by Lemma \ref{lem:holprops}, 5.~the holonomy of any ciliated face around the boundary of $R_e$  is trivial. Hence, the faces of $\gammad$
associated with edges of $\Gamma$ do not give rise to curvatures.

The faces of $\gammad$ that correspond to a vertex or a face of $\Gamma$ are represented by the vertex  or face loop from Definition \ref{def:vertloop}. If $\Gamma$ is a regular ciliated ribbon graph, each face $f$  of $\Gamma$ is represented by  a unique ciliated face $f(v)$  that starts and ends at the  cilium at $v$, and the pair $(v,f(v))$ is a site.  
This makes it  natural to combine the  associated vertex loop $p_v$  and face loop $p_{f(v)}$ and to consider the product of their holonomies.

By inserting the Haar integrals  of $H$ and $H^*$ into these holonomies, we then obtain a collection of idempotents in the centre of the gauge invariant subalgebra. Right-multiplication with these idempotents defines  a projector, whose image can be viewed as the analogue of the subalgebra of gauge invariant functions on the set of flat gauge fields  in a Hopf algebra gauge theory.

\begin{lemma} \label{lem:hdproj2}Let   $\Gamma$ be a regular ciliated ribbon graph, and denote for each vertex $v$ by $f(v)$ the ciliated face that starts and ends at the cilium at $v$. Then:
\begin{compactenum}
\item  The elements $G_v=\mathrm{Hol}_{p_v}(\ell \oo 1)\cdot \mathrm{Hol}_{p_{f(v)}}(1\oo\eta)$ are idempotents in the centre of $\mathcal H(H)^{\oo E}_{inv}$.\\[-2ex]
\item The linear maps
$ Q_v: \mathcal H(H)^{\oo E}\to\mathcal H(H)^{\oo E}$, $X\mapsto  X\cdot G_v$  form a set of commuting  projectors
with
$Q_v(X\lhd(\eta\oo \ell)_v)=G_v\cdot X\cdot G_v$ 
 for all $X\in \mathcal H(H)^{\oo E}$. \\[-2ex]

\item  $Q_{flat}=\Pi_{v\in V} Q_v:  \mathcal H(H)^{\oo E}\to\mathcal H(H)^{\oo E}$, $X\mapsto X\cdot \Pi_{v\in V}G_v$ is a projector.  Its restriction
to $\mathcal H(H)^{\oo E}_{inv}$  is an algebra morphism
and projects on a subalgebra   $\mathcal H(H)^{\oo E}_{flat}=Q_{flat}(\mathcal H(H)^{\oo E}_{inv})$.
\end{compactenum}
\end{lemma}
\begin{proof} The elements $G_v$ are given by $G_v=\tau_v(\eta\oo\ell)=\tau((\eta\oo \ell)_v)$, where  $\tau: D(H)^{\oo V}\to \mathcal H(H)^{\oo E}$   and $\tau_v=\tau\circ\iota_v: D(H)\to \mathcal H(H)^{\oo E}$  are the algebra morphisms from \eqref{eq:tau} and \eqref{eq:tauv}. 
As $\eta\oo \ell$ is a Haar integral for $D(H)$, it  follows  that the elements $G_v$ are idempotents and and $G_v$ and $G_w$ commute for all vertices $v,w\in V$.  
From the  Theorem \ref{lem:hdmodule} together with Lemma \ref{lem:helpvertface} and the properties of the Haar integral $\eta\oo \ell \in D(H)$  one obtains for $v\neq w$  
\begin{align*}
G_v\lhd (\eta\oo\ell)_v&=\tau_v( S((\eta\oo \ell)_{(2)})) \cdot \tau_v (\eta'\oo \ell')\cdot \tau_v ((\eta\oo\ell)_{(1)})=\tau_v( S((\eta\oo \ell)_{(2)})\cdot (\eta'\oo \ell')\cdot (\eta\oo\ell)_{(1)})\\
&=\epsilon((\eta\oo \ell)_{(2)})\epsilon((\eta\oo\ell)_{(1)})\, \tau_v(\eta'\oo\ell')=\tau_v(\eta'\oo\ell')=G_v\\
G_v\lhd (\eta\oo\ell)_w&=\tau_w( S((\eta\oo \ell)_{(2)})) \cdot \tau_v (\eta'\oo \ell')\cdot \tau_w ((\eta\oo\ell)_{(1)})=\tau_w( S((\eta\oo \ell)_{(2)})\cdot  (\eta\oo\ell)_{(1)})\cdot \tau_v(\eta'\oo\ell')\\
&=\epsilon(\eta\oo \ell)\, \tau_v(\eta'\oo\ell')=G_v.
\end{align*}
This  implies $Q_{inv}(G_v)=G_v$, and by
 Lemma \ref{lem:hdproj}, 3.  the elements $G_v$ are in the centre of $\mathcal H(H)^{\oo E}_{inv}$. 
That the maps $Q_v$  form a set of commuting projectors follows because the elements $G_v$ are commuting idempotents. 
A direct computation using again  Theorem \ref{lem:hdmodule} and Lemma \ref{lem:helpvertface} yields
 \begin{align*}
&Q_v(X\lhd(\eta\oo \ell)_v)= \tau_v(S( (\eta\oo \ell)_{(2)}))
\cdot X\cdot \tau_v((\eta\oo\ell)_{(1)})\cdot G_v \\
&=  \tau_v(S( (\eta\oo \ell)_{(2)}))
\cdot X\cdot \tau_v((\eta\oo\ell)_{(1)})\cdot \tau_v(\eta'\oo\ell') = \tau_v( S( (\eta\oo \ell)_{(2)}))
\cdot X\cdot \tau_v((\eta\oo\ell)_{(1)}\cdot (\eta'\oo\ell'))\\
&=\epsilon((\eta\oo\ell)_{(1)})\,  \tau_v(S( (\eta\oo \ell)_{(2)}))
\cdot X\cdot \tau_v(\eta'\oo\ell')=\tau_v(\eta\oo\ell)\cdot X\cdot \tau_v(\eta'\oo\ell')=G_v\cdot X\cdot G_v.
 \end{align*}
The product $\Pi_{v\in V} G_v$   of the commuting idempotents  $G_v$  is an idempotent, and hence $Q_{flat}$ is a projector. As the elements $G_v$ and  $\Pi_{v\in V} G_v$ are in the centre of $\mathcal H(H)^{\oo E}_{inv}$, its restriction to $\mathcal H(H)^{\oo E}_{inv}$ is an algebra morphism and its image a subalgebra of $\mathcal H(H)^{\oo E}_{inv}$. 
\end{proof}

Note that the projectors $Q_v$ act  by right multiplication with the idempotents $G_v$, while the corresponding projectors for a Hopf algebra gauge theory in Lemma \ref{lem:faceprpr} act by  left multiplication. This is because it is the opposite  $\mathcal H(H)^{op\oo E}$ and not  $\mathcal H(H)^{\oo E}$ that is a $D(H)^{\oo V}$-right module algebra. This convention for the projectors will simplify  the correspondence between Kitaev lattice models and  Hopf algebra gauge theories in the following sections.
We  now relate the gauge invariant subalgebra $\mathcal H(H)^{\oo E}_{inv}$ and its subalgebra $\mathcal H(H)^{\oo E}_{flat}$ to the Hamiltonian of Kitaev's lattice model and to the protected space from Definition \ref{def:hamiltonian}. 

\begin{proposition} \label{lem:protrep} Let $\Gamma$ be a regular ciliated ribbon graph and
 $\rho:\mathcal H(H)^{\oo E}\to\mathrm{End}_{\mathbb F}(H^{\oo E})$  the algebra isomorphism from Lemma \ref{lem:kitoprel}. Then  Kitaev's Hamiltonian  is given by
$H_{K}=\rho\circ Q_{\mathrm{flat}}( 1)$, and  $\rho$  induces an algebra homomorphism   $\rho_{\mathrm{pr}}: \mathcal H(H)^{\oo E}_{inv}\to\mathrm{End}_{\mathbb F}(\prot)$ that satisfies for all $X\in \mathcal H(H)^{\oo E}_{inv}$
$$
\rho_{\mathrm{pr}}(X)=\rho_{\mathrm{pr}}(Q_{flat}(X))=\ham\circ \rho(X)\vert_{\prot}.
$$
\end{proposition}
\begin{proof}The expression for the Hamiltonian follows directly from  Definition \ref{def:hamiltonian}, from the expressions for the vertex and face operators in Lemma \ref{lem:holab} and from the definition of the elements $G_v$ in Lemma \ref{lem:hdproj2}. That $\rho$ induces a representation of $\mathcal H(H)^{\oo E}_{inv}$ on the protected space $\prot=\ham(H^{\oo E})$ follows directly from the expression of the Hamiltonian and  Lemma \ref{lem:hdproj2}, which yield for all $X\in \mathcal H(H)^{\oo E}$
\begin{align*}
&\rho(Q_{inv}(X)) \circ \ham=
\rho(Q_{inv}(X))\circ \rho(\Pi_{v\in V} G_v)=\rho(Q_{inv}(X)\cdot \Pi_{v\in V}G_v)\\
&=
\rho(\Pi_{v\in V} Q_v(X\lhd(\eta\oo\ell)_v) ) 
=\rho(\Pi_{v\in V} G_v\cdot X\cdot \Pi_{w\in V}G_w)\\
&=\rho( \Pi_{v\in V}G_v\cdot Q_{flat}(X) )=
\ham\circ \rho(Q_{flat}(X))\\
&=\rho(\Pi_{v\in V} G_v)\circ \rho(X)\circ \rho(\Pi_{w\in V} G_w)=\ham\circ \rho(X)\circ \ham
\end{align*}
This shows that $\prot=\ham(H^{\oo E})$ is invariant under $\mathcal H(H)^{\oo E}_{inv}=Q_{inv}(\mathcal H(H)^{\oo E})$ and that $\rho$ induces an algebra homomorphism
$\rho_{\mathrm{pr}}: \mathcal H(H)^{\oo E}_{inv}\to\mathrm{End}_{\mathbb F}(\prot)$. As $\mathcal \prot=\ham(H^{\oo E})$,  
this representation satisfies the equations in the Lemma. 
\end{proof}

Although by Proposition \ref{lem:protrep} the protected space carries  representations of both, the gauge invariant subalgebra $\mathcal H(H)^{\oo E}_{inv}$ and the flat subalgebra $\mathcal H(H)^{\oo E}_{flat}=Q_{flat}(\mathcal H(H)^{\oo E}_{inv})$, 
 the former is a trivial extension of the latter, since $\rho_{\mathrm{pr}}=\rho_{\mathrm{pr}}\circ Q_{flat}$. Moreover, by Proposition \ref{lem:protrep}  the algebra isomorphism $\rho:\mathcal H(H)^{\oo E}\to\mathrm{End}_{\mathbb F}(H^{\oo E})$ induces an algebra isomorphism
$$
\mathcal H(H)^{\oo E}_{flat}\cong\{Y\in \mathrm{End}_{\mathbb F}(H): \, \ham\circ Y\circ \ham=Y\}.
$$ 
If one interprets the representations of $D(H)$ at the ciliated vertices of $\Gamma$ as the gauge symmetries of the model, it is then natural 
to view the protected space as the gauge invariant state space of the theory and
the  subalgebra $\mathcal H(H)^{\oo E}_{flat}$  as the algebra of gauge invariant observables. 
In fact, we will show in Section \ref{sec:qmod} that this subalgebra is isomorphic to the quantum moduli algebra of a Hopf algebra gauge theory from Theorem \ref{th:flat}.

\section{Kitaev models as a Hopf algebra gauge theory}
\label{sec:kithopf}

In this section we  relate algebra  of triangle operators for a $H$-valued Kitaev model on a ciliated ribbon graph $\Gamma$  to the algebra of functions  of a $D(H)$-valued Hopf algebra gauge theory on $\Gamma$. 
This is achieved by assigning to each edge $e$ of  $\Gamma$  two paths $p_{e,\pm}$ in the thickened graph $\gammad$. We then show  that the holonomies of these paths define an  algebra isomorphism  from the algebra $\mathcal A^*_\Gamma$ of functions in the Hopf algebra gauge theory to Kitaev's triangle operator algebra  $\mathcal H(H)^{op\oo E}$. 

The two  paths $p_{e,\pm}$ in $\gammad$ for an edge $e$ of $\Gamma$ are shown in Figure \ref{fig:edge_variable2}. 
The  path $p_{e,+}$ starts at the cilium at the starting vertex $\st(e)$ of $e$. It goes counterclockwise around the vertex $\st(e)$ until it reaches the edge $e$, then  follows  $e$ to the {\em right} of $e$ until it reaches its target vertex $\ta(e)$ and then goes clockwise around $\ta(e)$ until it reaches the cilium at $\ta(e)$. The edge path $p_{e,-}$ starts at the cilium at $\st(e)$, goes counterclockwise around  $\st(e)$ until it reaches the edge $e$, then  follows  $e$ to the {\em left} of $e$ until it reaches $\ta(e)$ and then goes clockwise around $\ta(e)$ until it reaches the cilium at $\ta(e)$. 
With the conventions from Definition \ref{def:vertex_nb} and \eqref{eq:orrev}, we have the following definition.

\begin{definition}  \label{def:edgepath}  Let  $\Gamma$ be a regular ciliated ribbon graph and $e$  an edge of $\Gamma$ that  is the $i$th edge  at its target vertex $\ta(e)$ and  the $j$th edge  at its starting vertex $\st(e)$.  Suppose the incident edges at $\ta(e)$  and $\st(e)$
of lower order than $t(e)$ and $s(e)$  are given by  $e_1,...,e_{i-1}$ and $f_1,...,f_{j-1}$, numbered according to the ordering at $\ta(e)$ and $\st(e)$  as  in Figure \ref{fig:edge_variable2}. Set $\epsilon_k=1\,(\epsilon_k=-1)$ and $\phi_l=1\,(\phi_l=-1)$ if $e_k$ and $f_l$ are  incoming   (outgoing) at $\ta(e)$ and 
$\st(e)$. Then the  {\bf edge paths} $p_{e,\pm}\in \mathcal G(\gammad)$ for  $e$ are 
\begin{align}\label{eq:edgepath}
&p_{e,+}=p_{t(e)<}\circ r(\bar e)\circ r(e)\circ p_{s(e)<}^\inv & &p_{e,-}=p_{t(e)<}\circ l( e)\circ l(\bar e)\circ p_{s(e)<}^\inv\\
\text{with}\quad &p_{t(e)<}=r(\bar e_1^{\epsilon_1})\circ\ldots\circ r(\bar e_{i-1}^{\epsilon_{i-1}}) & &p_{s(e)<}=r(\bar f_1^{\phi_1})\circ \ldots\circ r(\bar f_{j-1}^{\phi_{j-1}})\nonumber.
\end{align}
\end{definition}

Note that the  paths $p_{e,\pm}$  are {\em not} ribbon paths in the sense of Definition \ref{def:rpath} and \cite{Ki,BMD, BMCA}.
paths $r(\bar e)\circ r(e)$ and $l(e)\circ l(\bar e)$ each involve two overlapping triangles in Figure \ref{fig:foursites}, namely
the triangles $t(e)s(\bar e)s(e)$, $t(\bar e)t(e)s(\bar e)$ and $s(e)t(\bar e)t(e)$, $s(\bar e)s(e)t(\bar e)$, respectively. Nevertheless, they are natural from the geometric perspective. Similar paths were first investigated in in \cite{AM} in the context of moduli space of flat connections and then in \cite{BNR,MN}  where it was shown that they have a direct geometrical interpretation in 3d gravity.
Note also that the  path $p_{e^\inv, \pm }$ 
for the edge $e^\inv$ with the reversed orientation 
does not coincide with 
the reversed path  
$p_{e,\pm }^\inv$. Instead,  Definition \ref{def:edgepath}  and 
\eqref{eq:orrev}    imply
$p_{e^\inv,\pm }=p_{e,\mp }^\inv$. 
Nevertheless,  Lemma \ref{lem:holprops},  5. ensures that the holonomies of these paths agree. 

\begin{figure}
\centering
\includegraphics[scale=0.33]{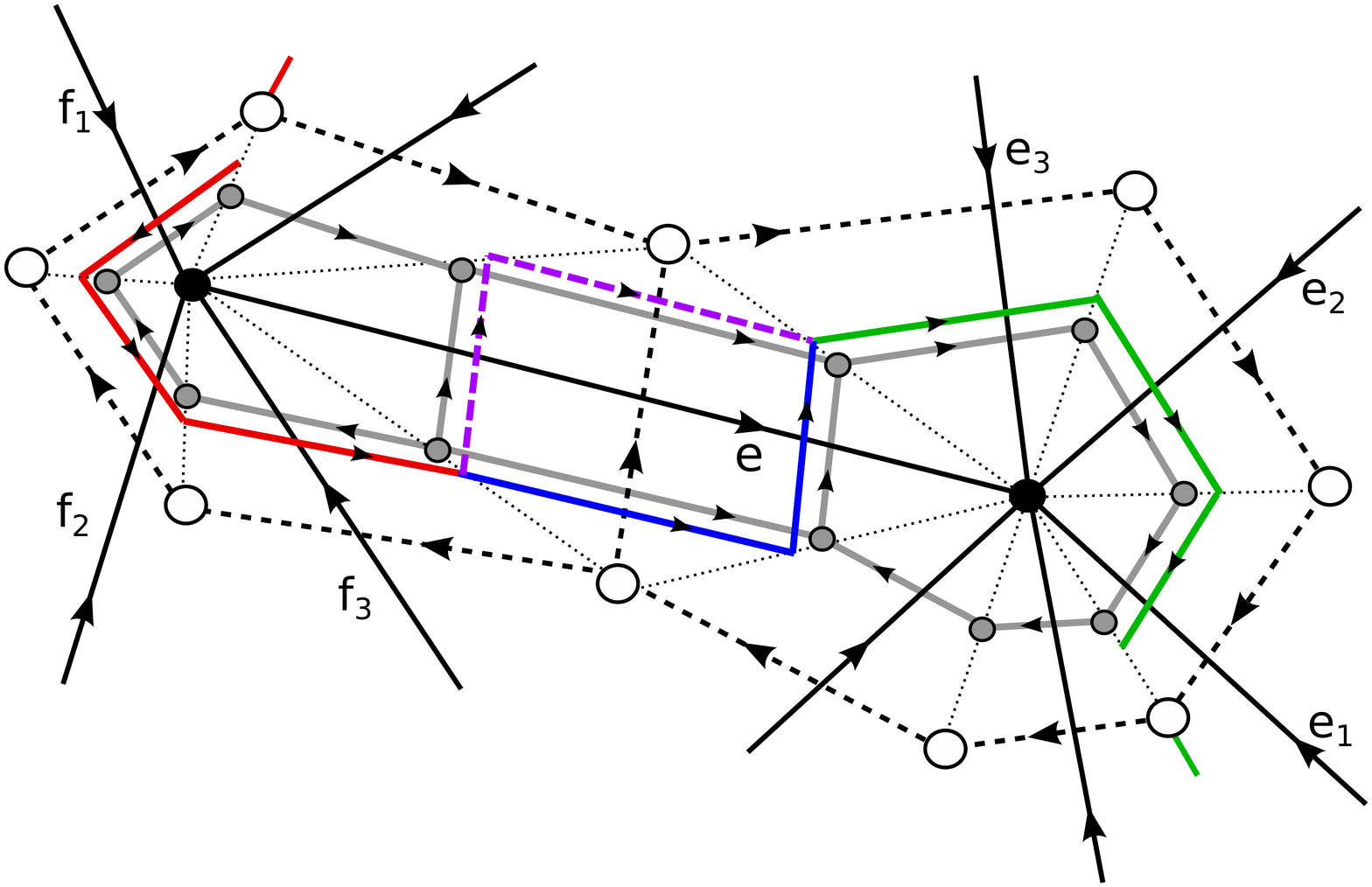}

\vspace{-.6cm}
\caption{The paths 
$p_{e,+}
=p_{t(e)<}\circ r(\bar e)\circ r(e)\circ p_{s(e)}^\inv$ and 
$p_{e,-}
=p_{t(e)<}\circ  l(e)\circ l(\bar e)\circ p_{s(e)}^\inv$  in $\gammad$ for an edge $e$ of $\Gamma$ and the subpaths 
 $p_{s(e)<}$ (red),  $p_{t(e)<}$ (green), $r(\bar e)\circ r (e)$ (blue) and  $l(e)\circ l(\bar e)$  (purple).}
\label{fig:edge_variable2}
\end{figure}

\begin{lemma}\label{lem:holl}  
Let $\Gamma$ be a regular ciliated ribbon graph and $e$ an edge of $\Gamma$ as in Definition \ref{def:edgepath}. Then the holonomies of the paths $p_{e,\pm}$ from \eqref{eq:edgepath} coincide and define an algebra homomorphism $\mathrm{Hol}_{p_{e,\pm}}: \mathcal H(H)\to\mathcal H(H)^{\oo E}$ given by
\begin{align}
&\mathrm{Hol}_{p_{e,\pm}}= \iota_{e_1\ldots e_{i-1}e f_{j-1}\cdots f_{1}}\circ (S^{\tau_1}_D\oo...\oo S^{\tau_{i-1}}_D\oo \id\oo S^{\rho_{j-1}}_D\oo ...\oo S^{\rho_1}_D)\circ (\phi_1^{\oo(i-1)}\oo \phi_1^{op\oo(j-1)})\circ \xi_1.\nonumber
\end{align}
where  $\phi_1,\xi_1:\mathcal H(H)\to\mathcal H(H)\oo \mathcal H(H)$  are the injective algebra morphisms from Lemma \ref{lem:antipodehd}, $S_D$ is  the antipode of $D(H)^*$ and  $2\tau_k=1-\epsilon_k$, $2\rho_l=1-\phi_l$.
\end{lemma}
\begin{proof}  Throughout the proof we denote by $\cdot$ the multiplication of $\mathcal H(H)^{\oo E}$ and by $\cdot'$ the multiplication of $D(H)^{*\oo E}$.
From the definition of the holonomy  and 
Lemma \ref{lem:holprops},  5.~one obtains 
\begin{align*}
\mathrm{Hol}_{p_{e,+}}=&\mathrm{Hol}_{p_{t(e)<}}  \bullet \mathrm{Hol}_{r(\bar e)\circ r(e)}\bullet \mathrm{Hol}_{p^\inv_{s(e)<}}=\mathrm{Hol}_{p_{t(e)<}}  \bullet \mathrm{Hol}_{l( e)\circ l(\bar e)}\bullet \mathrm{Hol}_{p^\inv_{s(e)<}}=\mathrm{Hol}(p_{e,-}).
\end{align*}
where $\bullet$ is the multiplication from \eqref{eq:assalg}.
As $p_{e,+}=p_{t(e)<}\circ r(\bar e)\circ r(e)\circ p_{s(e)<}^\inv$  we obtain with Lemma \ref{lem:functor},  Lemma \ref{lem:holprops}, 2. and 4. and  with the expression for the comultiplication of $D(H)^*$ in \eqref{eq:ddualmult}
\begin{align*}
\mathrm{Hol}_{p_e,\pm}(y\oo\gamma)
&=\Sigma_{i,j}\, \mathrm{Hol}_{p_{t(e)<}}(\low y 1\oo 1)\cdot'  (\low y 2\oo\alpha^i \gamma \alpha^j)_e \cdot' \mathrm{Hol}_{p_{s(e)<}} (S(x_j)S(\low y 3)x_i\oo 1).
\end{align*}
As $\Gamma$ has no multiple edges, we have  $\{e_1,...,e_{i-1}\}\cap\{f_1,...,f_{j-1}\}=\emptyset$. As   $e\notin\{e_1,...,e_{i-1}, f_1,...,f_{j-1}\}$, the three terms  in this product commute with respect to $\cdot'$ and $\cdot$ and their products with respect to $\cdot'$  and to $\cdot$ agree. To compute the holonomies of $p_{t(e)<}$ and $p_{s(e)<}$, we use  again Lemma \ref{lem:holprops}, 4.~together with  Lemma \ref{lem:functor} and 
the identity  $\mathrm{Hol}_{r(\bar f)^\inv}=S_D\circ\mathrm{Hol}_{r(\bar f)}$  from Lemma \ref{lem:functor}. This yields
\begin{align*}
\mathrm{Hol}_{p_{t(e)<}}(y\oo\gamma)
&=\mathrm{Hol}_{r(\bar e_1^{\epsilon_1})}((y\oo\gamma)_{(1)})\cdot' ...\cdot' \mathrm{Hol}_{r(\bar e_{i-1}^{\epsilon_{i-1}})}((y\oo\gamma)_{(i-1)})\\
&=\epsilon(\gamma)\; \mathrm{Hol}_{r(\bar e_1^{\epsilon_1})}(\low y 1\oo 1)\cdot'...\cdot' \mathrm{Hol}_{r(\bar e_{i-1}^{\epsilon_{i-1}})}(\low y {i-1}\oo 1)\\
&=\epsilon(\gamma)\;(S^{\tau_1}_D(\low y 1\oo 1))_{e_1}\cdot' ...\cdot' (S^{\tau_{i-1}}_D(\low y {i-1}\oo 1))_{e_{i-1}}
\end{align*}
and an analogous expression for $p_{s(e)<}$. As all factors in this product  commute  with respect to $\cdot$ and $\cdot'$ and their products with respect to  $\cdot'$ and $\cdot$ agree, we obtain for the holonomies of the paths $p_{e,\pm}$
\begin{align}
\label{eq:holedgepath}
&\mathrm{Hol}_{p_e,\pm}(y\oo\gamma)=
\Sigma_{i,j}\,\mathrm{Hol}_{p_{t(e)<}}(\low y 1\oo 1)\cdot (\low y 2\oo\alpha^i\gamma\alpha^j)_e\cdot \mathrm{Hol}_{p_{s(e)<}}(S(x_i)S(\low y 3)x_j\oo 1)\nonumber\\
\text{with}\quad &\mathrm{Hol}_{p_{t(e)<}}(y\oo\gamma)=\epsilon(\gamma)\;(S^{\tau_1}_D(\low y 1\oo 1))_{e_1}\cdot...\cdot (S^{\tau_{i-1}}_D(\low y {i-1}\oo 1))_{e_{i-1}}\nonumber\\
&\mathrm{Hol}_{p_{s(e)<}}(y\oo\gamma)=\epsilon(\gamma)\; (S_D^{\rho_1}(\low y 1\oo 1))_{f_1}\cdot ...\cdot (S_D^{\rho_{j-1}}(\low y {j-1}\oo 1))_{f_{j-1}}.
\end{align}
Comparing these expressions  with the ones for the algebra morphisms $\phi_1,\xi_1:\mathcal H(H)\to\mathcal H(H)$ from Lemma \ref{lem:antipodehd} and using  that $S_D$ is an algebra homomorphism by Lemma \ref{lem:antipodehd} proves the claim. 
\end{proof}

With the expressions for the holonomies  from Lemma \ref{lem:holl} we can now prove that they  relate  the algebra  of functions of a $D(H)$-valued Hopf algebra gauge theory on $\Gamma$  to the algebra  of triangle operators for the $H$-valued Kitaev model on $\Gamma$.

\begin{theorem} \label{th:hdcomb} Let $\Gamma$ be a regular ciliated ribbon graph. 
Consider the algebra structure $\mathcal A^*_\Gamma$ from Theorem \ref{lem:edge_algebra} for $K=D(H)$ and the $R$-matrix from \eqref{eq:dmult}.   For  edges $e,f$ of $\Gamma$ set $ e<f$ if  $e$ and $f$ have no common vertex  or if $e$ and $f$ share a vertex $v$ with $e<f$ at $v$. 
Then the map $\chi:\mathcal A^*_\Gamma\to \mathcal H(H)^{op \oo E}$
\begin{align*}
&\chi((y\oo\gamma)_e):=\mathrm{Hol}_{p_e,\pm} (y\oo\gamma) & &\forall e\in E\\
&\chi((y^1\oo\gamma^1)_{e_1}\cdot_{\mathcal A^*} ... \cdot_{\mathcal A^*} (y^k\oo\gamma^k)_{e_k}):= \chi((y^k\oo\gamma^k)_{e_k})\cdots \chi((y^1\oo\gamma^1)_{e_1}) & &\text{if } e_1<e_2<...<e_k.
\end{align*}
 is an algebra homomorphism. 
\end{theorem}
\begin{proof}   Throughout the proof we denote by $\cdot$ the multiplication of $\mathcal H(H)^{\oo E}$, by $\cdot'$ the multiplication of $D(H)^{*\oo E}$ and by $\cdot''$ the multiplication of $D(H)^*$. 
As $\Gamma$ is a ribbon graph without loops or multiple edges, we have for any two distinct edges $e,f\in E$ either $e<f$ or $f<e$ or both -  the latter if and only if $e$ and $f$ do not share a vertex. As the elements $(y\oo\gamma)_e$ for $e\in E$ generate $\mathcal A^*_\Gamma$ multiplicatively and the elements $(y\oo\gamma)_e$  and $(z\oo\delta)_f$ commute in $\mathcal A^*_\Gamma$ for any two edges $e,f$ that do not share a vertex, the map $\chi$ is well-defined. To show that it is an algebra homomorphism from $\mathcal A^*_\Gamma$ to $\mathcal H(H)^{op\oo E}$,  we have to verify the multiplication relations in Proposition \ref{prop:multrel}. 

As $\Gamma$ has no loops or multiple edges, any two edges $e,f\in E$ satisfies exactly one of the following: (i) $e=f$, (ii) $e$ and $f$ are not incident at a common vertex, (iii) $e$ and $f$ share exactly one vertex. 
As the paths $p_{e,\pm}$ from Definition \ref{def:edgepath} satisfy the assumptions of Lemma \ref{lem:holprops}, 2. 
and edge reversal in $\mathcal A^*_\Gamma$ corresponds to applying the antipode of $D(H)^*$, it is  sufficient to consider products of holonomies of the following form:
\begin{compactenum}[(i)]
\item $\mathrm{Hol}_{p_e,\pm}\cdot \mathrm{Hol}_{p_e,\pm}$
\item $\mathrm{Hol}_{p_e,\pm}\cdot \mathrm{Hol}_{p_f,\pm}$ for edges $e,f$ without a common vertex
\item $\mathrm{Hol}_{p_e,\pm}\cdot \mathrm{Hol}_{p_f,\pm}$ with $\ta(e)=\ta(f)$, $\st(e)\neq\st(f)$ and $e<f$ at $\ta(e)$.
\end{compactenum}
$\bullet$ {case (i):}  In this case, the claim follows directly from  the fact that $\mathrm{Hol}_{p_e,\pm}: \mathcal H(H)\to\mathcal H(H)^{\oo E}$ is an algebra morphism by Lemma \ref{lem:holl}.This yields for all $y,z\in H$, $\gamma,\delta\in H^*$
\begin{align*}
\mathrm{Hol}_{p_{e,\pm}}(y\oo\gamma)\cdot \mathrm{Hol}_{p_{e,\pm}}(z\oo\delta)
=\langle\low \gamma 1,\low z 2\rangle\; \mathrm{Hol}_{p_{e\pm}}(y\low z 1\oo \low\gamma 2\delta).
\end{align*}
With the expression for the $R$-matrix in \eqref{eq:dmult}  and  the multiplication of  $D(H)^*=H^{op}\oo H$  we obtain
 \begin{align}\label{eq:reval}
&\langle  (y\oo\gamma)_{(1)}\oo (z\oo\delta)_{(1)}, R\rangle\; (y\oo\gamma)_{(2)}\cdot'' (z\oo\delta)_{(2)}
=\langle\low\gamma 1, \low z 2\rangle\; (S(\low z 3)y\low z 1\oo\low\gamma 2) \cdot'' (\low z 4\oo\delta)\nonumber\\
&=\langle\low\gamma 1, \low z 2\rangle\; z_4S(\low z 3)y\low z 1\oo\low\gamma 2\delta=\langle\low\gamma 1, \low z 2\rangle\; y\low z 1\oo\low\gamma 2\delta,
 \end{align}
which allows us to rewrite the product of the holonomies as
\begin{align}\label{eq:holequal}
\mathrm{Hol}_{p_{e,\pm}}(y\oo\gamma)\cdot \mathrm{Hol}_{p_{e,\pm}}(z\oo\delta)= \langle (y\oo\gamma)_{(1)}\oo (z\oo\delta)_{(1)}, R\rangle\; \mathrm{Hol}_{p_{e,\pm}} ( (y\oo\gamma)_{(2)}\cdot'' (z\oo\delta)_{(2)}).
\end{align}
$\bullet$ {case (ii):} besides the factors $r(\bar e)$ and $r(e)$, the  
 expression for $p_{e,+}$ involves only factors $r(\bar g)^{\pm 1}$ (factors $l(\bar g)^{\pm 1}$) for  edges $g\in E$  that are incoming (outgoing) at $\st(e)$ or $\ta(e)$.  It therefore has non-trivial entries  only in the
  copies of $\mathcal H(H)$ in $\mathcal H(H)^{\oo E}$ that are associated with $e$ or with such edges $g$. 
  If $g\in E$ is an edge that is incident at one of the vertices $\ta(e)$ and $\st(e)$ and at one of the vertices  $\ta(f)$ and $\st(f)$, then $g\notin\{e,f\}$ and $g$ contributes a factor $r(\bar g)^{\pm}$ in one of the paths $p_{e,+}$, $p_{f,+}$ and a factor $l(\bar g)^{\pm}$  in the other, because it is incoming at one of these vertices and outgoing at the other.  By  \eqref{eq:hdcommutelr} and \eqref{eq:phirl}, the holonomies of $r(\bar g)^{\pm 1}$ and $l(\bar g)^{\pm 1}$ commute with respect to the multiplication $\cdot'$ of $\mathcal H(H)^{\oo E}$. As the different copies of $\mathcal H(H)$ in $\mathcal H(H)^{\oo E}$ commute as well, 
it follows that the holonomies of $p_{e,+}$ and $p_{f,+}$ commute  if $e$ and $f$ have no common vertex.

$\bullet$ {case  (iii):} To prove (iii) we decompose the  edge paths $p_{e,\pm}$ from \eqref{eq:edgepath} as
\begin{align}\label{eq:decompp}
&p_{e,+}=p_{t(e)\leq }\circ p_{s(e)<}^\inv & &p_{e,-}=p_{t(e)<}\circ p_{s(e)\leq }^\inv\\
\text{with}\quad
&p_{t(e)\leq}=p_{t(e)<}\circ r(\bar e)\circ r(e) & &p_{s(e)\leq}=p_{s(e)<}\circ r(\bar e^\inv)\circ r(e^\inv).\nonumber
\end{align}
Because  $\{e_1,...,e_{i-1}\}\cap\{f_1,...,f_{j-1}\}=\emptyset$ and $e\notin\{e_1,...,e_{i-1}, f_1,...,f_{j-1}\}$, the holonomy of 
 $p_{s(e)<}$ commutes with the holonomies of  $r(\bar e)\circ r(e)$, $p_{t(e)<}$ and $p_{t(e)\leq}$,
 and the holonomy of 
 $p_{t(e)<}$ commutes with the holonomies of  $r(\bar e)\circ r(e)$, $p_{s(e)<}$ and $p_{s(e)\leq}$.
As the holonomies of ${r(\bar e)}$ and ${l(\bar e)}$ commute  due to \eqref{eq:hdcommutelr} and  \eqref{eq:phirl},  the holonomy of $p_{t(e)<}$ (of  $p_{s(e)<}$) also commutes  with the holonomies of all paths $p_{g, \pm}$, for which $g$ is not not incident at  $\ta(e)$ (at $\st(e)$). 

Suppose now that $e,f$ share the vertex $\ta(e)=\ta(f)$. As $\Gamma$ has no loops or multiple edges, this implies $\st(e)\notin\{\st(f),\ta(e),\ta(f)\}$ and $\st(f)\notin\{\st(e),\ta(e),\ta(f)\}$.
Due to  Lemma \ref{lem:holl},  the fact that  $S_D:\mathcal H(H)\to \mathcal H(H)$ and $\phi_1,\xi_1: \mathcal H(H)\to \mathcal H(H)\oo \mathcal H(H)$ are algebra morphisms by Lemma \ref{lem:antipodehd}  and because  $\phi_1$ is coassociative,
 it is sufficient to consider a ciliated ribbon graph $\Gamma$ that consists of a vertex $v$  with two incoming edges $e,f$ such that $e<f$  at $v$ and such that  the vertices $\st(e)$ and $\st(f)$ are distinct and univalent.  In this case one has $p_{e,+}=p_{t(e)\leq}$ and $p_{f,+}=p_{t(f)\leq}$.
If one associates the first copy of $\mathcal H(H)$ in $\mathcal H(H)\oo \mathcal H(H)$ with  $e$ and the second with  $f$, one 
obtains
\begin{align*}
&\mathrm{Hol}_{p_{e,+}}(y\oo\gamma)=
(y\oo\gamma)\oo (1\oo 1)\qquad
\mathrm{Hol}_{p_{f,+}}(y\oo\gamma)
=(\low y 1\oo 1)\oo(\low y 2\oo \gamma),
\end{align*}
and this implies
\begin{align}\label{eq:multrels0}
&\mathrm{Hol}_{p_{e,+}}(y\oo\gamma)\cdot \mathrm{Hol}_{p_{f,+}}(z\oo\delta)
=\langle\low\gamma 1,\low z 2\rangle\; (y\low z 1\oo\low \gamma 2)\oo (\low z 3\oo\delta)\\
&=\langle\low\gamma 1,\low z 2\rangle\;  ((\low z 4\oo 1)\cdot (S(\low z 3)y\low z 1\oo\low \gamma 2))\oo (\low z 5\oo\delta)\nonumber\\
&=\langle \low\gamma 1,\low z 2\rangle\; \mathrm{Hol}_{p_{e,+}}(\low z 4\oo\delta)\cdot \mathrm{Hol}_{p_{f,+}}(S(\low z 3)y\low z 1\oo\low \gamma 2).\nonumber\qquad\qquad\qquad\qquad
\end{align}
With formula \eqref{eq:reval} this yields
\begin{align*}
&\mathrm{Hol}_{p_{e,+}}(y\oo\gamma)\!\cdot\! \mathrm{Hol}_{p_{f,+}}(z\oo\delta)\!=\!\langle  (y\oo\gamma)_{(1)}\!\oo\! (z\oo\delta)_{(1)}, R\rangle \;\mathrm{Hol}_{p_{f,+}}((z\oo\delta)_{(2)})\!\cdot\! \mathrm{Hol}_{p_{e,+}}((y\oo\gamma)_{(2)}).\nonumber
\end{align*}
Comparing this expression and expression \eqref{eq:holequal} with  the multiplication relations in Proposition \ref{prop:multrel} then proves 
that  $\chi:\mathcal A^*_\Gamma\to \mathcal H(H)^{op \oo E}$ is an algebra homomorphism.
\end{proof}

We will now show that the linear map $\chi: \mathcal A^*_\Gamma\to \mathcal H(H)^{\oo E}$ is not only an algebra homomorphism but an algebra {\em isomorphism}.  This is achieved by expressing it in terms of certain  algebra homomorphisms  from the algebra of functions $\mathcal A^*_v$ on each vertex neighbourhood $\Gamma_v$ into  $\otimes_{v\in V}\mathcal H(H)^{\oo |v|}$. 
These algebra homomorphisms are  obtained from  the holonomies of  paths in $\gammad$ that generalise the  paths $p_{e,\pm}$ from Definition \ref{def:edgepath}. 

\begin{definition}\label{def:nbpaths} 
Let $v$ be a ciliated vertex with $n$ incident edges $e_1,...e_n$, numbered according to the ordering at $v$ and such that $e_1^{\epsilon_1}$,..., $e_n^{\epsilon_n}$ are incoming.  Then for $i\in\{1,...,n\}$  we define
\begin{align*}
p_{e_i,0}:
=r(\bar e_1^{\epsilon_1})\circ...\circ r(\bar e_{i-1}^{\epsilon_{i-1}})\circ r(\bar e_i)\circ r(e_i)\qquad p_{e_i,1}:= 
r(\bar e_1^{\epsilon_1})\circ...\circ r(\bar e_{i-1}^{\epsilon_{i-1}})\circ l(e_i).
\end{align*}
\end{definition}

We equip the vector space $(H\oo H^*)^{\oo n}$ with  the algebra structure $\mathcal A^*_v$ from  
Definition \ref{th:vertex_gt} for the Hopf algebra $K=D(H)$, arbitrary parameters $\sigma_i\in\{0,1\}$ and the universal $R$-matrix from \eqref{eq:dmult}. 

\begin{lemma} \label{lem:vertnb} 
The map  $\chi_{\tau,\sigma} : \mathcal A^*_v\to \mathcal H(H)^{op\oo n}$ with
\begin{align*}
&\chi_{\tau,\sigma}((y\oo\gamma)_i)=\mathrm{Hol}_{p_{e_i},\sigma_i}(y\oo\gamma)\\
&\chi_{\tau,\sigma}(y^1\oo\gamma^1\oo...\oo y^n\oo\gamma^n)=\chi_{\tau,\sigma}((y^n\oo \gamma^n)_n)\cdot...\cdot\chi_{\tau,\sigma}((y^1\oo\gamma^1)_1)
\end{align*}
 is an algebra homomorphism. It is an  isomorphism iff $\sigma_i=0$ for all $i\in\{1,...,n\}$ or $\dim_{\mathbb F}(H)=1$.
\end{lemma}
\begin{proof}  

If we identify the $i$th copy of $\mathcal H(H)$ in $\mathcal H(H)^{\oo n}$ with the $i$th edge $e_i$, then by the proof of Lemma \ref{lem:holl}  the holonomies of the paths $p_{e_i,\sigma_i}$ are given by
 \begin{align*}
 &\mathrm{Hol}_{p_{e_i,\sigma_i}}=(S^{\tau_1}_D\oo...\oo S^{\tau_n}_D)\circ \mathrm{Hol}_{p_{e_i,\sigma}}^{\tau=0}\circ S^{\tau_i}_D,\\
 \text{where}\quad &\mathrm{Hol}_{p_{e_i,\sigma_i}}^{\tau=0}(y\oo\gamma)=\begin{cases}
 (\low y 1\oo 1)\oo ...\oo (\low y {i-1}\oo 1)\oo(\low y i\oo\gamma)\oo(1\oo 1)^{\oo (n-i)} & \sigma_i=0\\
(\low y 1\oo 1)\oo ...\oo (\low y {i-1}\oo 1)
\oo S_D(1\oo S(\gamma))\oo(1\oo 1)^{\oo(n-i)} & \sigma_i=1
 \end{cases}
 \end{align*}
 denotes the holonomies for the case where all edges $e_1,...,e_n$ are incoming at $v$, $2\tau_i=1-\epsilon_i$ and $S_D$ is the antipode of $D(H)^*$. 
The map $S_D:\mathcal H(H)\to\mathcal H(H)$ is an algebra isomorphism by Lemma \ref{lem:antipodehd}, and the algebra structure from Definition  \ref{th:vertex_gt} for general edge orientation is defined by the condition that  $S^{\tau_1}_D\oo...\oo S^{\tau_n}_D$ is an algebra morphism to the corresponding  algebra structure for $\tau_1=...=\tau_n=0$. It is therefore sufficient to consider the case $\tau_i=0$ for all $i\in\{1,...,n\}$. 
If $\sigma_i=0$ for all $i\in\{1,....,n\}$ the claim follows from Theorem \ref{th:hdcomb}, applied to a graph that consists of a central vertex $v$ with $n$ incident edges and $n$ univalent vertices.  For general $\sigma$, we have
\begin{align*}
(\chi_{0,\sigma}\circ \iota_i)(y\oo\gamma)=\begin{cases} \phi_1^{ (i-1)}(y\oo\gamma)\oo(1\oo 1)^{\oo (n-i)} & \sigma_i=0\\
\phi_1^{(i-2)}(y\oo 1)\oo S_D(1\oo S(\gamma))\oo(1\oo 1)^{\oo(n-i)} & \sigma_i=1,
\end{cases}
\end{align*}
where $\phi_1: \mathcal H(H)\to\mathcal H(H)\oo
\mathcal H(H)$ is the coassociative algebra homomorphism from Lemma \ref{lem:antipodehd}. It is therefore sufficient to consider the case $n=2$ and to compute the products $(y\oo\gamma)_i\cdot (y\oo\gamma)_i$ for $i\in\{1,2\}$ and $\sigma_i=1$, $(y\oo\gamma)_1\cdot (z\oo\delta)_2$ for $\sigma_1=\sigma_2=1$, $\sigma_1=1-\sigma_2=0$ and $\sigma_1=1-\sigma_2=1$. For the first two cases, we obtain
\begin{align*}
&\chi_{0,\sigma}((y\oo\gamma)_1)\cdot \chi_{0,\sigma}((z\oo\delta)_1)=\epsilon(y)\epsilon(z) \big( S_D(1\oo S(\gamma))\oo (1\oo 1)\big)\cdot \big( S_D(1\oo S(\delta))\oo (1\oo 1)\big)\\
&=\epsilon(yz)\, S_D(1\oo S(\delta\gamma))\oo (1\oo 1)=\chi_{0,\sigma}((yz\oo\delta\gamma)_1)=\chi_{0,\sigma}((z\oo\delta)\cdot''(y\oo\gamma))\\
&\chi_{0,\sigma}((y\oo\gamma)_2)\cdot \chi_{0,\sigma}((z\oo\delta)_2)=\big((\low y 1\oo 1)\oo (S_D(1\oo S(\gamma)))\big)\cdot \big((\low z 1\oo 1)\oo (S_D(1\oo S(\delta))) \big)\quad\\
&=(\low y 1\low z 1\oo 1)\oo (S_D(1\oo S(\low\delta\gamma)))=\chi_{0,\sigma}( yz\oo \delta\gamma)=\chi_{0,\sigma}((z\oo\delta)\cdot''(y\oo\gamma)),\end{align*}
where $\cdot''$ denotes the multiplication of $D(H)^*=H^{op}\oo H^*$.
If $\sigma_1=1-\sigma_2=0$, we obtain
\begin{align*}
&\chi_{0,\sigma}((y\oo\gamma)_1)\cdot \chi_{0,\sigma}((z\oo\delta)_2)=\big((y\oo\gamma)\oo(1\oo 1)\big)\cdot \big((\low z 1\oo 1)\oo S_D(1\oo S(\delta))\big)\\
&=\langle\low\gamma 1,\low z 2\rangle\, (y\low z 1\oo \low\gamma 2)\oo S_D(1\oo S(\delta)) \\
&=\langle \low\gamma 1,\low z 2\rangle\, \big((\low z 4\oo 1)\oo S_D(1\oo S(\delta))\big)\cdot \big((S(\low z 3)y z_1\oo \gamma)\oo (1\oo 1) \big)\\
&=\langle (y\oo \gamma)_{(1)}\oo (z\oo\delta)_{(1)},R\rangle\, \chi_{0,\sigma}(((z\oo\delta)_{(2)})_2)\cdot \chi_{0,\sigma}(((y\oo\gamma)_{(2)})_1),\qquad\qquad\qquad\qquad\qquad
\end{align*}
where we used equation \eqref{eq:reval} in the last step. Similarly,  for $\sigma_1=1-\sigma_2=1$ we have
\begin{align*}
&\chi_{0,\sigma}((y\oo\gamma)_1)\cdot \chi_{0,\sigma}((z\oo\delta)_2)=\epsilon(y)\, \big( S_D(1\oo S(\gamma))\oo(1\oo 1)\big)\cdot \big((\low z 1\oo 1)\oo (\low z 2\oo\delta)\big)\\
&=\epsilon(y)\, \big((x_kS(x_l)\oo\alpha^l\gamma\alpha^k)\oo(1\oo 1)\big)\cdot \big((\low z 1\oo 1)\oo (\low z 2\oo\delta)\big)\\
&=\epsilon(y)\langle\alpha^l_{(1)}\low\gamma 1\alpha^k_{(1)},\low z 2\rangle\, (x_kS(x_l)\low z 1\oo \alpha^l_{(2)}\low\gamma 2\alpha^k_{(2)})\oo (\low z 3\oo\delta) \\
&=\epsilon(y)\langle \low\gamma 1,\low z 1\rangle\, (\low z 2 x_kS(x_l)\oo \alpha^l\low\gamma 2\alpha^k)\oo (\low z 3\oo\delta) \\
&=\epsilon(S(\low z 3)y \low z 1)\langle \low\gamma 1,\low z 2\rangle\,  \big((\low z 4\oo 1)\oo (\low z 5\oo\delta)\big)\cdot \big((x_kS(x_l)\oo\alpha^l\low\gamma 2\alpha^k)\oo (1\oo 1)\big)\\
&=\langle (y\oo \gamma)_{(1)}\oo (z\oo\delta)_{(1)}, R\rangle\, \chi_{0,\sigma}(((z\oo\delta)_{(2)})_2)\cdot \chi_{0,\sigma}(((y\oo\gamma)_{(2)})_1),\qquad\;\;\;
\end{align*}
and an analogous computation yields the same identity for $\sigma_1=\sigma_2=1$. Comparing these expressions with the multiplication relations in Theorem \ref{rem:flipalg}  proves that $\chi_{0,\sigma}$ is an algebra homomorphism.
Clearly, $\chi_{\tau,\sigma}$ is bijective if $\dim_{\mathbb F}(H)=1$. 
To show that  for $\dim_{\mathbb F}(H)>1$ the map  $\chi_{\tau,\sigma}$ is bijective if and only if $\sigma_i=0$ for all $i\in\{1,...,n\}$, it is again sufficient to consider the case $\tau_1=...=\tau_n=0$, because the map $S^{\tau_1}_D\oo...\oo S^{\tau_n}_D$ is an involution.  For $\sigma_1=...=\sigma_n=0$  one has
\begin{align*}
&\chi_{0,0}((y^1\oo\gamma^1)\oo...\oo(y^n\oo\gamma^n))\\
&=(y^n_{(1)}\cdots y^2_{(1)}y^1\oo\gamma^1)\oo (y^n_{(2)}\cdots y^2_{(2)}\oo\gamma^2)\oo \ldots\oo (y^n_{(n-1)}y^{n-1}_{(n-1)}\oo\gamma^{n-1})\oo (y^n_{(n)}\oo\gamma^n),
\end{align*}
and hence $\chi_{0,0}$ is invertible with inverse
 \begin{align*}
&\chi_{0,0}^\inv((z^1\oo\delta^1)\oo...\oo (z^n\oo\delta^n))\\
&=(S(z^2_{(1)})z^1\oo \delta^1)\oo (S(z^3_{(1)})z^2_{(2)}\oo\delta^2)\oo...\oo (S(z^n_{(1)})z^{n-1}_{(2)}\oo\delta^{n-1})\oo (z^n_{(2)}\oo\delta^n).\qquad
 \end{align*}
 If  $k=\mathrm{min}\{i:\sigma_i=1\}\in \{1,...,n\}$ and $\dim_{\mathbb F}(H)>1$, one has 
  \begin{align*}
 &\chi_{0,\sigma}((S(\low y 1)\oo 1)_{k-1}\cdot_{\mathcal A^*} (\low y 2\oo 1)_k)
 =\phi_1^{(k-2)}(\low y 2S(\low y 1)\oo 1)\oo(1\oo 1)^{\oo(n-k+1)}=\epsilon(y)\,(1\oo 1)^{\oo n}.
 \end{align*}
As $\dim_{\mathbb F}(H)>1$  implies $\mathrm{ker}(\epsilon)\neq \{0\}$, it  follows  that $\mathrm{ker}(\chi_{0,\sigma})\neq \{0\}$.
\end{proof}

The  paths $p_{e_i,\sigma_i}$ in the vertex neighbourhoods that define the algebra homomorphism $\chi_{\sigma,\tau}$ give a geometrical interpretation to the parameters $\sigma_i\in\{0,1\}$ that arise in the definition of a Hopf algebra gauge theory  in Theorem \ref{rem:flipalg}. While it  was shown in \cite{MW} that these parameters are necessary to combine the algebra structures on the different vertex neighbourhoods into an algebra, i.~e.~to ensure that the image of the map $G^*$ from \eqref{eq:dualemb} is a subalgebra of $\oo_{v\in V}\mathcal A^*_v$, 
they were introduced there by purely algebraic considerations.
From the perspective of the Kitaev model they have a geometrical meaning. 
Passing from the ciliated ribbon graph $\Gamma$ to its vertex neighbourhoods $\Gamma_v$  involves a subdivision of each edge $e$ of $\Gamma$ into two edge ends $s(e)$ and $t(e)$. The  paths $p_{e,\pm}$  from Definition \ref{def:edgepath} then correspond to the paths
\begin{align*}
&p'_{e,+}=p_{t(e)<}\circ r(\overline{t(e)})\circ r(t(e))\circ r(s(e))\circ p^\inv_{s(e)<}=p_{t(e), 0}\circ p_{s(e), 1}^\inv\\
&p'_{e,-}=p_{t(e)<}\circ l(t(e))\circ l(s(e)) \circ  l(\overline{s(e)})\circ p^\inv_{s(e)<})=p_{t(e), 1}\circ p_{s(e), 0}^\inv
\end{align*}  in the thickening of the subdivided graph, which split naturally into paths $p_{t(e), \sigma}$ and $p_{s(e),1-\sigma}$ in the vertex neighbourhoods $\Gamma_{\ta(e)}$ and $\Gamma_{\st(e)}$, as shown in Figure \ref{fig:edge_split}.
 These are precisely the paths from Definition \ref{def:nbpaths} whose holonomies  define the algebra homomorphism in Lemma \ref{lem:vertnb}. Hence, the
parameters $\sigma_i\in\{0,1\}$ in Theorem \ref{rem:flipalg} describe the splitting of the associated paths 
 $p'_{e,\pm}$ in the thickened graph into two paths  $p_{t(e),\sigma}$ and $p_{s(e), 1-\sigma}$ in the vertex neighbourhoods $\Gamma_{\ta(e)}$ and $\Gamma_{\st(e)}$. 

\begin{figure}
\centering
\includegraphics[scale=0.32]{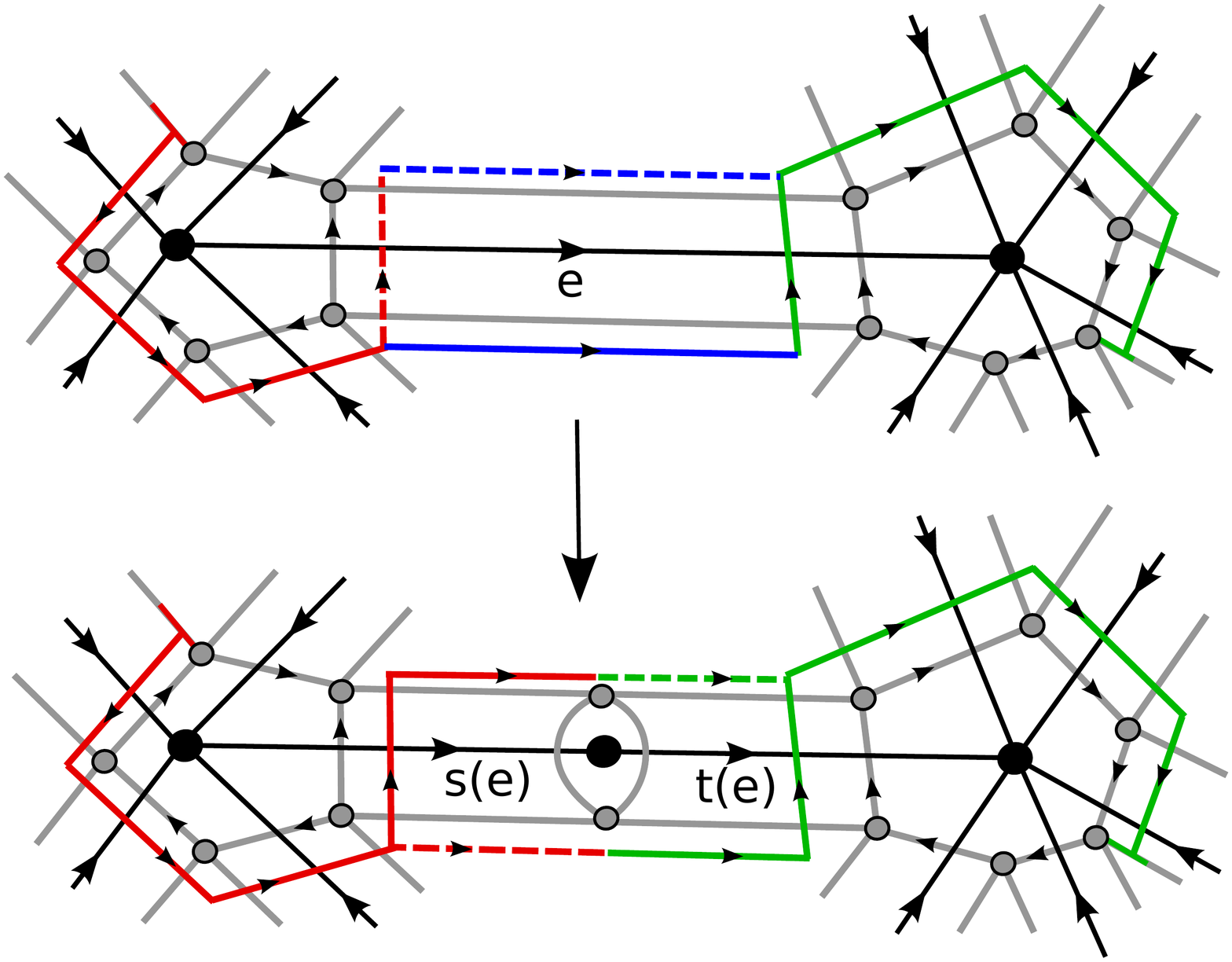}

\vspace{-.3cm}
\caption{The paths $p_{e,\pm}$ in $\gammad$ and the splitting of the associated paths $p'_{e,\pm}$  in the thickening of the subdivision into  paths $p_{t(e),0}$, $p_{t(e),1}$  (green) and $p_{s(e),0}$, $p_{s(e),1}$ (red) in the vertex neighbourhoods of $t(e)$ and $s(e)$.
}
\label{fig:edge_split}
\end{figure}


\begin{theorem} \label{th:chiiso}For any regular ciliated ribbon graph $\Gamma$, 
 the map $\chi:\mathcal A^*_\Gamma\to \mathcal H(H)^{op \oo E}$ from Theorem \ref{th:hdcomb} is an algebra isomorphism.
\end{theorem}
\begin{proof}  As $\mathcal A^*_\Gamma\cong (H\oo H^*)^{\oo E}\cong \mathcal H(H)^{\oo E}$ as vector spaces, it is sufficient to show that $\chi$ is injective. We relate $\chi$ to the maps from Lemma \ref{lem:vertnb} via the injective linear map $G^*: \mathcal A^*_\Gamma\to \oo_{v\in V}\mathcal A^*_v$ from \eqref{eq:dualemb}. For this, define  $\rhop:E(\Gamma) \to \cup_{v\in V}E(\Gamma_v)$  by setting $\rhop(e)=t(e)$ for each edge $e\in E(\Gamma)$. For each vertex neighbourhood $\Gamma_v$ and edge end $f\in E(\Gamma_v)$ we set $\tau_f=\sigma_f=0$ if  $f$ is incoming at $v$ and $\tau_f=\sigma_f=1$ else and equip $\Gamma_v$ with the algebra structure $\mathcal A^*_v$ from Definition \ref{th:vertex_gt}. Then by 
Theorem \ref{lem:edge_algebra} the algebra $\mathcal A^*_\Gamma$ is isomorphic to the image
of  $G^*$.  
Denote  by $\chi_v=\chi_{\tau_v,\sigma_v}: \mathcal A^*_v\to\mathcal H(H)^{\oo |v|}$ the linear map from Lemma \ref{lem:vertnb} for the vertex neighbourhood $\Gamma_v$. Then one has
\begin{align*}
&\chi=M^*\circ (\oo_{v\in V}\chi_v)\circ G^*\quad\text{with}\\
&M^*: \oo_{v\in V}\mathcal H(H)^{\oo |v|}\to \mathcal H(H)^{\oo E},\quad (y\oo \gamma\oo z\oo\delta)_{t(e)s(e)}\mapsto \epsilon(\delta)\,(yz\oo\gamma)_e. 
\end{align*}
From the definition of  $\chi_v=\chi_{\sigma_v\tau_v}$ in Lemma \ref{lem:vertnb} it follows  that the restriction of $M^*$ to the image of $(\oo_{v\in V}\chi_v)\circ G^*_e$ is injective. Hence, it is sufficient to prove that  $ (\oo_{v\in V}\chi_v)\circ G^*:\mathcal A^*_\Gamma\to \oo_{v\in V}\mathcal H(H)^{\oo |v|}$ is injective. For this we assign to a vertex $v$ of $\Gamma$  of valence $|v|=n$  the map
\begin{align*}
&\mu_v=(S_D^{\tau_{1}}\oo ... \oo S_D^{\tau_{n}})\circ \mu_{v,0}\circ(S_D^{\tau_{1}}\oo ... \oo S_D^{\tau_{n}}): \mathcal H(H)^{\oo n }\to \mathcal A^*_v\\
&\mu_{v,0}(z^1\oo\delta^1\oo...\oo z^n\oo\delta^n)=(S(z^2_{(1)})z^1\oo\delta^1)\oo(S(z^3_{(1)})z^2_{(2)}\oo\delta^2)\oo ....\oo (z^n_{(2)}\oo\delta^n),
\end{align*}
as in the proof of Lemma \ref{lem:vertnb}. Then  $\mu_v\circ \chi_v=(S_D^{\tau_{1}}\oo ... \oo S_D^{\tau_{n}})\circ \mu_{v,0}\circ \chi_{v,0}\circ (S_D^{\tau_{1}}\oo ... \oo S_D^{\tau_{n}})$ and  
\begin{align*}
&\mu_{v,0}\circ \chi_{v,0}(y^1\oo \gamma^1\oo ...\oo y^n\oo\gamma^n)\\
&=m_1(y^1\oo y^2_{(1)}\oo\gamma^1)\oo m_2(y^2_{(2)}\oo y^3_{(1)}\oo\gamma^2)\oo ...\oo m_{n-1}(y^{n-1}_{(2)}\oo y^n_{(1)}\oo \gamma^{n-1})\oo m_n(y^n_{(2)}\oo 1\oo \gamma^n)\\
&\text{with}\quad m_i(y\oo z\oo\delta)=\begin{cases} \epsilon(z)\, y\oo\delta & \sigma_i=\sigma_{i+1}=0\\
\epsilon(y)\, \Sigma_{r,s}\, zx_rS(x_s) \oo \alpha^s\delta\alpha^r    & \sigma_i=\sigma_{i+1}=1\\
\epsilon(y)\epsilon(z) \Sigma_{r,s}\, x_rS(x_s) \oo \alpha^s\delta\alpha^r & \sigma_i=1-\sigma_{i+1}=1\\
zy\oo\delta &  \sigma_i=1-\sigma_{i+1}=0,
\end{cases}
\end{align*}
where we set $\sigma_{n+1}=0$. As $(\epsilon\oo\id)\circ\Delta=(\id\oo\epsilon)\circ\Delta=\id$ and for each edge $e$ of $\Gamma$ one has $\sigma_{s(e)}=1-\sigma_{t(e)}=0$,  it follows  that   $(\oo_{v\in V}\mu_v)\circ(\oo_{v\in V}\chi_v)\circ G^*:\mathcal A^*_\Gamma\to\oo_{v\in V}\mathcal H(H)^{\oo |v|}$ is injective unless  there is a finite sequence of edges $e_1$, $e_2$,...., $e_n$ such that $\ta(e_i)=\st(e_{i+1})$ for  $i\in\{1,...,n-1\}$, $\ta(e_n)=\st(e_1)$,  $e_i$ and $e_{i+1}$ are adjacent at $\st(e_{i+1})=\ta(e_i)$ with $e_{i+1}<e_i$ and $e_1$ and $e_n$ are adjacent at $\st(e_1)=\ta(e_n)$ with $e_1<e_n$. However, this implies that the path $f=e_1^\inv\circ ...\circ e_n^\inv$ is a ciliated face that represents a face of $\Gamma$  that does not contain a cilium, a contradiction to  the assumption that $\Gamma$ is regular. Hence $(\oo_{v\in V}\mu_v)\circ(\oo_{v\in V}\chi_v)\circ G^*:\mathcal A^*_\Gamma\to\oo_{v\in V}\mathcal H(H)^{\oo |v|}$ and  $(\oo_{v\in V}\chi_v)\circ G^*$ are injective, and  the claim follows.
\end{proof}

\section{The quantum moduli algebra and the protected space}
\label{sec:qmod}

The last section proves that for each regular ciliated ribbon graph $\Gamma$,  the algebra $\mathcal H(H)^{op\oo E}$ of triangle operators for a $H$-valued Kitaev model on $\Gamma$ is isomorphic to the algebra of functions $\mathcal A^*_\Gamma$  of a  $D(H)$-valued Hopf algebra gauge theory. In this section we show, that the isomorphism $\chi:\mathcal A^*_\Gamma\to\mathcal H(H)^{op \oo E}$ is compatible with  the action of gauge symmetries and the curvatures in the two models. We start by proving that it is a module morphism with respect to the $D(H)^{\oo V}$-module structures on $\mathcal A^*_\Gamma$ and on $\mathcal H(H)^{op\oo E}$. We then show  that it sends the holonomies of faces of $\Gamma$ that are based at the cilium of $\Gamma$ to the product of the associated vertex and face operator. The final result  is an algebra isomorphism between the quantum moduli algebra of the Hopf algebra gauge theory from Theorem \ref{th:flatinv} and the  subalgebra of Kitaev's triangle operator algebra from Lemma \ref{lem:hdproj2}.

To show that the map $\chi:\mathcal A^*_\Gamma\to\mathcal H(H)^{op \oo E}$ from Theorem \ref{th:hdcomb} is not only an algebra isomorphism but also a module morphism with respect to gauge transformations,  we consider the $D(H)^{\oo V}$-right module algebra structure  on the algebra  $\mathcal A^*_\Gamma$   from Theorem \ref{lem:edge_algebra} and Proposition \ref{prop:multrel}  and  the $D(H)^{\oo V}$-right module algebra structure on the  algebra $\mathcal H(H)^{op\oo E}$ from Theorem \ref{lem:hdmodule}. 

\begin{theorem} \label{lem:actedge}Let $\Gamma$ be a regular ciliated ribbon graph.  Then for all edges $e$ of $\Gamma$ one has
\begin{align*}
\mathrm{Hol}_{p_{e,\pm}}(y\oo\gamma)\lhd (\delta\oo z)_v
&=\begin{cases}
\epsilon(z)\epsilon(\delta)\,\mathrm{Hol}_{p_{e,\pm}}(y\oo\gamma) & v\notin\{\st(e),\ta(e)\}\\
\langle  \delta\oo z, (y\oo\gamma)_{(1)}\rangle\; \mathrm{Hol}_{p_{e,\pm}}((y\oo\gamma)_{(2)}) & v=\ta(e)\\
\langle  \delta\oo z, S_D((y\oo\gamma)_{(2)})\rangle\; \mathrm{Hol}_{p_{e,\pm}}((y\oo\gamma)_{(1)}) & v=\st(e).
\end{cases}
\end{align*}
 The map  $\chi: \mathcal A^*_\Gamma\to\mathcal H(H)^{op \oo E }$  from Theorem \ref{th:hdcomb} is an isomorphism of $D(H)^{\oo V }$-right module algebras and restricts to an algebra isomorphism $\chi:\mathcal A^*_{\Gamma\,inv}\to \mathcal H(H)^{op\oo E }_{inv}$.
\end{theorem}
\begin{proof}  1.~Denote for each vertex $v$ by $f(v)$ the ciliated face of $\Gamma$ that starts and ends at the cilium at $v$ and by $p_v$ and $p_{f(v)}$ the associated vertex and face loop in $\gammad$.  
To prove the first equation in the theorem, we consider a vertex $v\notin\{\st(e),\ta(e)\}$ and show that the holonomies of $p_v$ and $p_{f(v)}$ commute with the holonomy of $p_{e,\pm}$. The equation then follows from the formula for the $D(H)^{\oo V}$-module structure in Theorem \ref{lem:hdmodule}.
For this, note that the vertex loop $p_v$ is composed of paths $r(\bar g)$ and $l(\bar g)^{-1}$ for edges $g\in E(\Gamma)$ incident at $v$. The path $p_{e,\pm}$ is composed of  paths  $r(\bar h)^{\pm 1}$, $l(\bar h)^{\pm 1}$, $r(h)^{\pm 1}$, $l( h)^{\pm 1}$ for $h\in E(\Gamma)$ incident at $\st(e)$ or $\ta(e)$. It follows that the holonomy of $p_v$ commutes with the holonomy of $p_{e,\pm}$ if $v\notin\{\st(e),\ta(e)\}$.

To prove the corresponding statement for the face loop $p_{f(v)}$, consider an edge $e$ as in Definition \ref{def:edgepath} and let  $p_{e,\pm}$ be the associated edge path from \eqref{eq:edgepath}. 
As $S_D: \mathcal H(H)\to\mathcal H(H)$ is an algebra morphism, we  can suppose  that all  edges $e_1$,...,$e_{i-1}$ and $f_1$,...,$f_{j-1}$ are incoming. If $v\notin \{\st(e),\ta(e)\}$  the regularity of $\Gamma$ implies that  $f(v)$ does no traverse the cilium at $\st(e)$ or $\ta(e)$ and 
  the edges $e_1,...,e_{i-1}, f_1,...,f_{j-1}, e$ and their inverses are not  the first or last edge in $f(v)$. Hence, the face loop $p_{f(v)}$
 can be decomposed into 
\begin{compactenum}[(i)]
\item subpaths which do not contain $r(g)^{\pm 1}$ or $l(g)^{\pm 1}$ for $g\in \{e_1,...,e_{i-1}, e, f_1,..,f_{j-1}\}$,
\item subpaths of the form $l(e_{a+1})^\inv\circ r(e_{a})$ or $l(f_{b+1})^\inv\circ r(f_{b})$ with $1\leq a\leq i-1$, $1\leq b\leq j-1$,
\item subpaths of the form $r(e)\circ r(f_{j-1})$ or  $l(e)^\inv\circ r(e_{i-1})$,
\end{compactenum}
as shown in Figure \ref{fig:reidemeister}. The  holonomies of the paths in (i) commute with the holonomy of $p_{e,\pm}$ by definition of  $\mathcal H(H)^{\oo E}$, the holonomies of the paths in (ii)  and the holonomy of  $l(e)^\inv\circ r(e_{i-1})$ in (iii) commute with 
the holonomy of $p_{e,\pm}$ by Lemma \ref{lem:commhelp}. 
The  same holds for the holonomy  of  $r(e)\circ r(f_{j-1})$  in (iii), since  $r(e)\circ r(f_{j-1})=l(e)^\inv\circ r(f_{j-1})$  and $\mathrm{Hol}_{p_{e,\pm }^\inv}=\mathrm{Hol}_{p_{e^\inv,\pm }}$ by Lemma \ref{lem:holl}.
This shows that the holonomy of $p_{f(v)}$ commutes with the holonomies of all paths $p_{e,\pm}$ with $v\notin\{\st(e),\ta(e)\}$ and completes the proof of  the first equation.  

2.~As $\mathrm{Hol}_{p_{e^\inv,\pm}}=\mathrm{Hol}_{p_{e,\pm}^\inv}=\mathrm{Hol}_{p_{e,\pm}}\circ S_D$ and $\Delta\circ S_D=(S_D\oo S_D)\circ \Delta^{op}$ in $D(H)^*$, the third equation follows from the second equation. To prove the second equation, we compute the commutation relations of the holonomies of $p_v$ and $p_{f(v)}$
 with the holonomy of $p_{e,\pm}$ for $v=\ta(e)$. Due to the formula in Lemma \ref{lem:holl}
and because $S_D:\mathcal H(H)\to\mathcal H(H)$   is an algebra morphism by Lemma \ref{lem:antipodehd},
it is sufficient to consider a vertex $v$ at which all edges are incoming. 
 Suppose that the incident edges at $v$ are  
 $e_1,...,e_n$, numbered according to the ordering at $v$.

 2(a).~We start by computing the commutation relations of the holonomy of the vertex loop $p_{v}$  with the holonomies of the  paths $p_{e_i,\pm}$. As the edges $e_1,..,e_n$ are incoming at $v$, one has
 $p_v=r(\bar e_1)\circ \ldots\circ r(\bar e_n)$, and the paths $p_{e_i,+}$ can be decomposed as $p_{e_i,+}=p_{t(e_i)\leq}\circ p_{s(e_i)<}^\inv$, where 
 $p_{t(e_i)\leq}=p_{t(e_i)<}\circ r(\bar{e_i})\circ r(e_i)$. Note that $p_{s(e_i)<}$ does not contain the factors $r(\bar e_i)$, $l(\bar e_i)$, $r(e_i)$ and $l(e_i)$  or their inverses since $\Gamma$ has no loops or multiple edges. It follows that the holonomy of $p_v$ commutes with the holonomy of $p_{s(e_i)<}$ for all $i\in\{1,...,n\}$.  As  $\mathrm{Hol}_{p_v}(z\oo \delta)=\epsilon(\delta)\,\mathrm{Hol}_{p_{t(e_n)\leq}}(z\oo 1)$, equation \eqref{eq:multrels0} 
implies for the paths $p_{t(e_i)\leq}$ with $i\in\{1,...,n-1\}$ 
\begin{align*}
& \mathrm{Hol}_{p_{t(e_i)\leq}}(y\oo\gamma)\cdot \mathrm{Hol}_{p_v}(z\oo 1)=\langle \low\gamma 1,\low z 2\rangle \; \mathrm{Hol}_{p_v}(\low z 4\oo 1)\cdot  \mathrm{Hol}_{p_{t(e_i)\leq} } (S(\low z 3)y \low z 1\oo\gamma).\qquad\qquad
\end{align*}
For the paths $p_{e_n,\pm}$ we obtain the same identity  using the fact that $\mathrm{Hol}_{p_{e_n,\pm}}: \mathcal H(H)\to\mathcal H(H)^{\oo E}$ is an algebra morphism by Lemma \ref{lem:antipodehd}:
\begin{align*}
&\mathrm{Hol}_{p_{t(e_n)\leq}}(y\oo\gamma)\cdot \mathrm{Hol}_{p_v}(z\oo 1)
=\mathrm{Hol}_{p_{t(e_n)\leq}}(y\oo\gamma)\cdot \mathrm{Hol}_{p_{t(e_n)\leq}}(z\oo 1)\\
&=\langle \low\gamma 1, \low z 2\rangle\; \mathrm{Hol}_{p_{t(e_n)\leq}}(\low z 1y\oo\low \gamma 2)
=\langle \low\gamma 1, \low z 2\rangle\; \mathrm{Hol}_{p_{t(e_n)\leq}}(\low z 4 S(\low z 3)y \low z 1y\oo\low \gamma 2)\qquad\qquad\quad\;
\\
&=\langle \low\gamma 1,\low z 2\rangle \; \mathrm{Hol}_{p_{t(e_n)\leq}}(\low z 4\oo 1)\cdot  \mathrm{Hol}_{p_{t(e_n)\leq}} (S(\low z 3)y \low z 1\oo\gamma)
\\
&=\langle \low\gamma 1,\low z 2\rangle \; \mathrm{Hol}_{p_v}(\low z 4\oo 1)\cdot  \mathrm{Hol}_{p_{t(e_n)\leq}} (S(\low z 3)y \low z 1\oo\gamma).\end{align*}

This implies for all $i\in\{1,...,n\}$ 
\begin{align}\label{eq:helpacomm}
&\mathrm{Hol}_{p_{e_i,\pm}}(y\oo\gamma)\lhd (1\oo z)_v\\
&=\mathrm{Hol}_{p_v}(S(\low z 2)\oo 1)\cdot  \mathrm{Hol}_{p_{t(e_i)\leq}}((y\oo\gamma)_{(1)})\cdot \mathrm{Hol}_{p_{s(e_i)<}^\inv}((y\oo\gamma)_{(2)})\cdot  \mathrm{Hol}_{p_v}(\low z 1\oo 1)\nonumber\\
&=\Sigma_{k,l}\,\mathrm{Hol}_{p_v}(S(\low z 2)\oo 1)\cdot  \mathrm{Hol}_{p_{t(e_i)\leq}}( \low y 1\oo \alpha^k\gamma\alpha^l)\cdot  \mathrm{Hol}_{p_v}(\low z 1\oo 1) \cdot \mathrm{Hol}_{p_{s(e_i)<}^\inv}( S( x_l) \low y 2 x_k\oo 1)\nonumber\\
&=\Sigma_{k,l}\,\langle \alpha^k_{(1)}\low\gamma 1\alpha^l_{(1)}, \low z 2\rangle\; 
\mathrm{Hol}_{p_{t(e_i)\leq}}( S(\low z 3) \low y 1\low z 1\oo \alpha^k_{(2)}\low \gamma 2\alpha^l_{(2)})\cdot
\mathrm{Hol}_{p_{s(e_i)<}^\inv}(S( x_l) \low y 2 x_k\oo 1)\nonumber\\
&=\Sigma_{k,l}\,\langle \low\gamma 1, \low z 3\rangle\; 
\mathrm{Hol}_{p_{t(e_i)\leq}}( S(\low z 5) \low y 1\low z 1\oo \alpha^k\low \gamma 2\alpha^l)\cdot
\mathrm{Hol}_{p_{s(e_i)<}^\inv}( S( x_l) S(\low z 4)\low y 2 \low z 2 x_k\oo 1)\nonumber\\
&=\langle \low\gamma 1, \low z 2\rangle\; \mathrm{Hol}_{p_{e_i,\pm}}(S(\low z 3)  y \low z 1\oo\low\gamma 2)\nonumber
\end{align}
and proves the second equation for  $z\in H$ and $\delta=1$.

2(b).~To prove the second equation for $z=1$ and   $\delta\in H^*$, we compute the commutation relations of the holonomy of the face loop $f(v)$ with the holonomies of the paths $p_{e_i,\pm}$. 
As $p_{f(v)}$ starts and ends at the cilium at $v$, it can be decomposed as $p_{f(v)}=q\circ l(e_1)^\inv$, where
$q=r(e_n)\circ t\in\mathcal G(\gammad)$ is a path from $\st(e_1)$ to $v$ that turns maximally right at each vertex, traverses each edge at most once and does not traverse any cilia, as shown  in Figure \ref{fig:faceact}.
The holonomy of $p_{f(v)}$  is given by
 \begin{align*}
 \mathrm{Hol}_{p_{f(v)}}(y\oo\gamma)&=\epsilon(y)\; \mathrm{Hol}_{q}(1\oo\low\gamma 1)\cdot \mathrm{Hol}_{l(e_1)^\inv}(1\oo\low\gamma 2).
 \end{align*}
By 1.~and by Lemma  \ref{lem:commhelp},  the holonomy of $q$ commutes with the holonomies of  the paths $p_{e_i,\pm}$ for all $i\in\{1,...,n-1\}$ and with the holonomy of $p_{e_n,-}$. This is apparent from Figure \ref{fig:faceact}. As none of the edges in  $p_{s(e_i)<}$ are the first or last edge in $p_{f(v)}$,  by 1.~the holonomies of $p_{f(v)}$, $l(e_1)^\inv$ and $q$  commute with the holonomy of $p_{s(e_i)<}$ for all $i\in\{1,...,n-1\}$. 
This implies for all  $i\in\{1,...,n\}$
\begin{align}\label{eq:helpcom}
&\mathrm{Hol}_{p_{e_i,\pm}}(y\oo\gamma)\lhd(\delta\oo 1)_v=\mathrm{Hol}_{p_{f(v)}}(1\oo S(\low \delta 1))\cdot \mathrm{Hol}_{p_{e_i,\pm}}(y\oo\gamma)\cdot \mathrm{Hol}_{p_{f(v)}}(1\oo\low \delta 2)\\
&=\mathrm{Hol}_{l(e_1)}(1\oo \low\delta 1)\cdot \mathrm{Hol}_{p_{t(e_i)\leq }}((y\oo\gamma)_{(1)}) 
\cdot \mathrm{Hol}_{l(e_1)}(1\oo S(\low\delta 2)) 
\cdot \mathrm{Hol}_{p_{s(e_i)<}}( S((y\oo\gamma)_{(2)})).\qquad\qquad\qquad\qquad\nonumber
\end{align}
To evaluate this expression further, we use the identities
\begin{align*}
&\mathrm{Hol}_{l(e_1)}(1\oo \delta  )\cdot ( y \oo \gamma)_{e_1}
=\langle \low\delta 1, \low y 1\rangle\; (\low y 2 \oo \gamma)_{e_1}\cdot \mathrm{Hol}_{l(e_1)}(1\oo \low \delta 2) \\[+1ex]
&\mathrm{Hol}_{l(e_1)}(1\oo  \delta )\cdot ( y \oo \gamma)_{e_i}=( y \oo \gamma)_{e_i}\cdot \mathrm{Hol}_{l(e_1)}(1\oo  \delta )\quad\qquad\qquad\qquad\qquad\qquad\qquad \forall i\in\{2,...,n\},
\end{align*}
 which follow from  the algebra structure of the Heisenberg double and  Lemma \ref{lem:functor}.  
 Inserting them into \eqref{eq:helpcom} and using the identity  $\mathrm{Hol}_{l(e_1)}(1\oo\gamma)\cdot \mathrm{Hol}_{l(e_1)}(1\oo\delta)=\mathrm{Hol}_{l(e_1)}(1\oo\gamma\delta)$, which follows from  Lemma \ref{lem:functor}, we obtain
\begin{align}\label{eq:secc}
&\mathrm{Hol}_{p_{e_i,\pm}}(y\oo\gamma)\lhd(1\oo\delta)_v\\
&=\Sigma_{k,l}\mathrm{Hol}_{l(e_1)}(1\oo \low\delta 1)\cdot \mathrm{Hol}_{p_{t(e_i)\leq }}(\low y 1\oo\alpha^k\low\gamma 1\alpha^l)
\cdot \mathrm{Hol}_{l(e_1)}(1\oo S(\low\delta 2)) 
\cdot \mathrm{Hol}_{p_{s(e_i)<}^\inv}( S(x_l)\low y 2 x_l\oo \low\gamma 2)\nonumber\\
&=\Sigma_{k,l}\langle \low\delta 1,\low y 1\rangle\;  \mathrm{Hol}_{p_{t(e_i)\leq }}(\low y 2\oo\alpha^k\low\gamma 1\alpha^l)
\cdot \mathrm{Hol}_{l(e_1)}(1\oo \low\delta 2S(\low\delta 3)) 
\cdot \mathrm{Hol}_{p_{s(e_i)<}^\inv}(S(x_l)\low y 3 x_k\oo \low\gamma 2)\nonumber\\
&=\Sigma_{k,l}\langle \delta ,\low y 1\rangle\;  \mathrm{Hol}_{p_{t(e_i)\leq }}(\low y 2\oo\alpha^k\low\gamma 1\alpha^l)
\cdot \mathrm{Hol}_{p_{s(e_i)<}^\inv}(S(x_l)\low y 3 x_k\oo \low\gamma 2)
=\langle \delta ,\low y 1\rangle\; \mathrm{Hol}_{p_{e_i,\pm}}(\low y 2\oo\gamma).\nonumber
\end{align}
This proves the second equation for  $\delta\in H^*$ and $z=1$. 

2(c).~Combining \eqref{eq:secc} and \eqref{eq:helpacomm} and using the  fact that the linear map $\tau_v:D(H)\to \mathcal H(H)^{\oo E}$, $\delta\oo z\mapsto \mathrm{Hol}_{p_{f(v)}}(1\oo\delta)\cdot \mathrm{Hol}_{p_v}(z\oo 1)$ from \eqref{eq:tauv} is an algebra homomorphism,  we obtain the second equation  for general $y,z\in H$ and $\gamma,\delta\in H^*$
\begin{align*}
 &\mathrm{Hol}_{p_{e_i,\pm}}(y\oo\gamma)\lhd (\delta\oo z)_v=(\mathrm{Hol}_{p_{e_i,\pm}}\lhd (\delta\oo 1)_v)\lhd (1\oo z)_v
=\langle \delta ,\low y 1\rangle\;    \mathrm{Hol}_{p_{e_i,\pm}}(\low y 2\oo\gamma)\lhd (1\oo z)_v\\
&=\langle \delta ,\low y 1\rangle\;  \langle \low\gamma 1, \low z 2\rangle\; \mathrm{Hol}_{p_{e_i,\pm}}(S(\low z 3)\low y 2\low z 1\oo\low \gamma 2)\\
&=\Sigma_{k,l}\langle \delta ,\low y 1\rangle\;  \langle \alpha^k\low\gamma 1\alpha^l, z \rangle\;
\mathrm{Hol}_{p_{e,\pm}}(S(x_l)\low y 2 x_k\oo\low \gamma 2)\\
&=\Sigma_{k,l}\langle \delta\oo z, \low y 1 \oo \alpha^k\low\gamma 1\alpha^l \rangle\; \mathrm{Hol}_{p_{e,\pm}}(S(x_l)\low y 2x_k\oo\low\gamma 2)=\langle \delta \oo z, (y\oo\gamma)_{(1)}\rangle\; \mathrm{Hol}_{p_{e,\pm}}((y\oo\gamma)_{(2)}).\qquad
\end{align*}
3.~That $\chi$ is a morphism of $D(H)^{\oo V}$-right modules 
 follows directly by comparing the formulas  in this theorem with the expressions for the $D(H)^{\oo V}$-module structure on $\mathcal A^*_\Gamma$ in Proposition \ref{prop:multrel}. The subalgebras $\mathcal A^*_{\Gamma\, inv}\subset\mathcal A^*_\Gamma$ and $\mathcal H(H)^{op\oo E}_{inv}\subset\mathcal H(H)^{op \oo E}$
are the images of the projectors
$$P_{inv}:\mathcal A^*_\Gamma\to\mathcal A^*_\Gamma, \; X\mapsto X\lhd(\eta\oo \ell)^{\oo V}\qquad Q_{inv}: \mathcal H(H)^{op\oo E}\to\mathcal H(H)^{op\oo E}, \; X\mapsto X\lhd(\eta\oo \ell)^{\oo V}$$
from Theorem \ref{th:ginvsub} and  Lemma \ref{lem:hdproj}.
As $\chi$ is an isomorphism of $D(H)^{\oo V}$-right module algebras, it satisfies $\chi\circ P_{inv}=Q_{inv}\circ \chi$ and hence induces an algebra isomorphism  $\chi:\mathcal A^*_{inv}\to \mathcal H(H)^{op\oo E}_{inv}$.
\end{proof}

\begin{figure}
\centering
\includegraphics[scale=0.27]{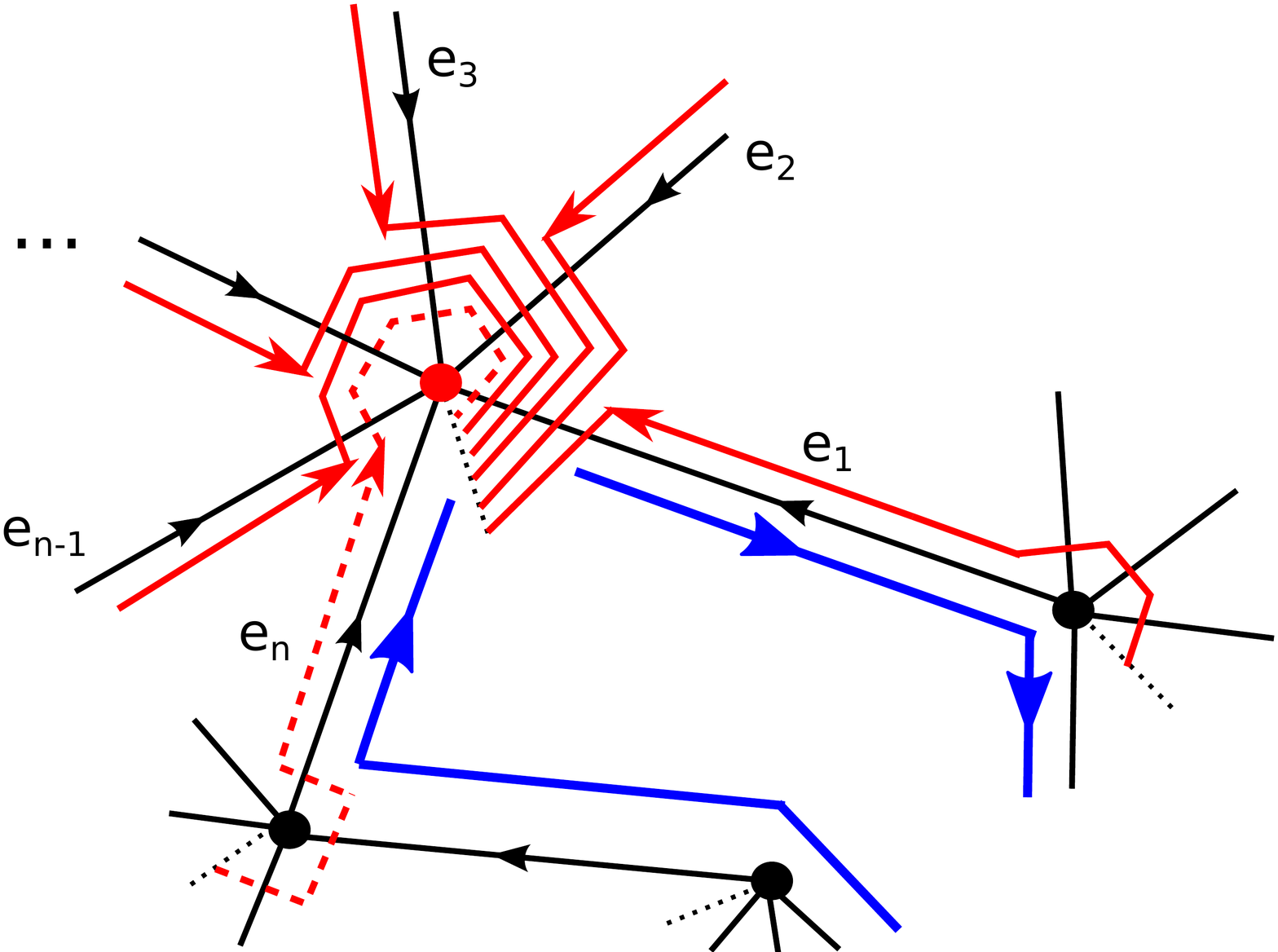}
\caption{A face loop (blue)  based at a ciliated vertex  of $\Gamma$ (red),  the paths $p_{e_1,+}$, ..., $p_{e_{n-1},+}$ (solid red) and  the path $p_{e_n,-}$ (dashed red) for the incident edges $e_1,...,e_n$. }
\label{fig:faceact}
\end{figure}

After clarifying the relation between gauge transformations in the two models, we  can now relate  curvatures in  a Hopf algebra gauge theory  on $\Gamma$ to the curvatures  in the associated  Kitaev model.  As discussed in Section \ref{sec:gtheory}, the curvatures of a Hopf algebra gauge theory on a regular ciliated ribbon graph $\Gamma$ are given by the holonomies of ciliated faces  that start and end at a cilia of $\Gamma$. Similarly, it was shown in  Section \ref{subsec:flatkit} that the curvatures in the Kitaev models are given by the holonomies of the associated vertex loops and face loops  in its thickening $\gammad$ that start and end at these cilia.  
The isomorphism  $\chi:\mathcal A^*_\Gamma\to\mathcal H(H)^{op \oo E}$  from Theorem \ref{th:hdcomb} relates the former to the latter.

\begin{proposition} \label{lem:facetransfer}Let  $\Gamma$ be a regular ciliated ribbon graph, $v$  a vertex of $\Gamma$, $f(v)$ the ciliated face that starts and ends at the cilium at $v$ and $p_v$, $p_{f(v)}$ the associated vertex and face loops. Then
$$\chi\circ \mathrm{Hol}^\Gamma_{f(v)}(y\oo\gamma)=\mathrm{Hol}^{\gammad}_{p_v}(y\oo 1)\cdot \mathrm{Hol}^{\gammad}_{p_{f(v)}}(1\oo\gamma)\qquad \forall y\in H, \gamma\in H^*,$$
and the projectors $P_f: \mathcal A^*_\Gamma\to\mathcal A^*_\Gamma$ and   $Q_v:\mathcal H(H)^{\oo E}\to\mathcal H(H)^{\oo E}$ from Lemma \ref{lem:faceprpr}  and \ref{lem:hdproj2} satisfy
$$\chi\circ P_{f(v)}=Q_v\circ \chi.$$
\end{proposition}
\begin{proof} Denote by $\cdot$ the multiplication of $\mathcal H(H)^{\oo E}$ 
and by $\cdot_{\mathcal A^*}$ the multiplication of $\mathcal A^*_\Gamma$. Suppose the ciliated face $f(v)\in \mathcal G(\Gamma)$ is given by the reduced word $f(v)=e_1^{\epsilon_1}\circ ...\circ e_n^{\epsilon_n}$ with $e_i\in E(\Gamma)$ and $\epsilon_i\in\{\pm 1\}$. As $f(v)$ traverses each edge $e_i$ at most once, its  holonomy is given by
\begin{align}
\mathrm{Hol}^\Gamma_{f(v)}(y\oo\gamma)&=
(S^{\tau_1}_D((y\oo\gamma)_{(1)})\oo S^{\tau_2}_D((y\oo\gamma)_{(2)})\oo ...\oo S^{\tau_n}_D((y\oo\gamma)_{(n)})    )_{e_1...e_n},\nonumber
\end{align}
where $2\tau_i=1-\epsilon_i$. 
As $f(v)$  turns maximally right at each vertex and does not traverse any cilia, the edges $e_i$ and $e_{i+1}$ are adjacent at the vertex $\ta(e_{i+1}^{\epsilon_{i+1}})=\st(e_{i}^{\epsilon_{i}})$ with $e_{i+1}<e_{i}$ for all $i\in\{1,...,n-1\}$.  As $f(v)$ starts and ends at the cilium at $v$,  the edge $e_n$ is the edge of lowest  and the edge $e_1$  the edge of highest order at $v$,  as shown in Figure \ref{fig:faceedgepath}. Because $\Gamma$ has no loops or multiple edges, it follows from Proposition \ref{prop:multrel} that 
$(y\oo\gamma\oo z\oo \delta)_{e_ie_{j}}=(z\oo\delta)_{e_{j}}\cdot_{\mathcal A^*}(y\oo\gamma)_{e_i}$ for all $1\leq i<j\leq n$. Because $f(v)$ traverses each edge of $\Gamma$ at most once, this implies
\begin{align}
\mathrm{Hol}^\Gamma_{f(y)}(y\oo\gamma)=(S^{\tau_n}_D((y\oo\gamma)_{(n)}))_{e_n}\cdot_{\mathcal A^*} ...\cdot_{\mathcal A^*} (S^{\tau_2}_D((y\oo\gamma)_{(2)}))_{e_2}\cdot_{\mathcal A^*} (S^{\tau_1}_D((y\oo\gamma)_{(1)}))_{e_1}.\nonumber
\end{align}
As $\chi:\mathcal A^*_\Gamma\to \mathcal H(H)^{op \oo E}$ is an algebra isomorphism with $\chi\circ \iota_e=\mathrm{Hol}^\gammad_{p_{e,\pm}}$ and  $\mathrm{Hol}^{\gammad}_{p_{e,\pm}}\circ S_D=\mathrm{Hol}^{\gammad}_{p_{e^\inv,\pm}}$ for all edges $e$ of $\Gamma$, this implies
\begin{align}\label{eq:chillep}
\chi\circ \mathrm{Hol}^{\Gamma}_{f(v)}(y\oo\gamma)
=\mathrm{Hol}^{\gammad}_{p_{e_1^{\epsilon_1},\pm }}((y\oo\gamma)_{(1)})\cdot ... \cdot \mathrm{Hol}^{\gammad}_{p_{e_n^{\epsilon_n},\pm }}((y\oo\gamma)_{(n)}),
\end{align}
where  $p_{e,\pm}$  are the paths from Definition \ref{def:edgepath}:
$$p_{e_i^{\epsilon_i}+}=p_{t(e_i^{\epsilon_i})<}\circ r(\bar e_i^{\epsilon_i})\circ r(e_i^{\epsilon_i})\circ p_{s(e_i^{\epsilon_i})<}^\inv=:{t_i}\circ r(e_i^{\epsilon_i})\circ s_i^\inv.$$
  As $e_i$ and $e_{i+1}$ are adjacent at  the vertex $\ta(e_{i+1}^{\epsilon_{i+1}})=\st(e_{i}^{\epsilon_{i}})$ with $e_{i+1}<e_{i}$ for  $i\in\{1,...,n-1\}$ and $e_n$ and $e_1$ are  the edges of lowest and highest order at $v$, respectively, one has $s_i=t_{i+1}$ for all $i\in\{1,...,n-1\}$, 
  ${s_n}=\o_v$, $t_1=p_v$ and $p_{f(v)}=r(e_1^{\epsilon_1})\circ ...\circ r(e_n^{\epsilon_n})$,
 as shown in Figure \ref{fig:faceedgepath}. 
Inserting these identities into \eqref{eq:chillep} one obtains
\begin{align*}
\chi\circ \mathrm{Hol}^\Gamma_{f(v)}(y\oo\gamma)
&=\mathrm{Hol}^\gammad_{t_1}((y\oo\gamma)_{(1)(1)})\cdot  \mathrm{Hol}^{\gammad}_{r(e_1^{\epsilon_1})}((y\oo\gamma)_{(1)(2)})\cdot \mathrm{Hol}^\gammad_{s_1^\inv}((y\oo\gamma)_{(1)(3)})\\
&\;\;\cdot \mathrm{Hol}^\gammad_{t_2}((y\oo\gamma)_{(2)(1)})\cdot  \mathrm{Hol}^{\gammad}_{r(e_2^{\epsilon_2})}((y\oo\gamma)_{(2)(2)})\cdot \mathrm{Hol}^\gammad_{s_2^\inv}((y\oo\gamma)_{(2)(3)})\\
&\;\;\quad\vdots\\
&\;\;\cdot\mathrm{Hol}^\gammad_{t_n}((y\oo\gamma)_{(n)(1)})\cdot  \mathrm{Hol}^{\gammad}_{r(e_n^{\epsilon_n})}((y\oo\gamma)_{(n)(2)})\cdot \mathrm{Hol}^\gammad_{s_n^\inv}((y\oo\gamma)_{(2)(n)})\\
&=\mathrm{Hol}^{\gammad}_{p_v}(y\oo 1)\cdot \mathrm{Hol}^{\gammad}_{p_{f(v)}}(1\oo\gamma),
\end{align*}
where we used that the last holonomy in each line cancels with the first  in the next one, because $s_i=t_{i+1}$,  that the first holonomy is equal to the holonomy of $p_{v}$ and that the last one is trivial. The remaining holonomies combine to form the holonomy of $p_{f(v)}$.
This implies for all $X\in \mathcal A^*_\Gamma$
\begin{align*}
\chi\circ P_{f(v)}(X)&=\chi(\mathrm{Hol}^\Gamma_{f(v)}(\ell\oo\eta)\cdot X)
=\chi(X)\cdot \mathrm{Hol}^{\gammad}_{p_v}(\ell\oo 1)\cdot \mathrm{Hol}^{\gammad}_{p_{f(v)}}(1\oo\eta)=Q_v\circ\chi(X).\\[-9ex]
\end{align*}
\end{proof}

\begin{figure}
\centering
\includegraphics[scale=0.35]{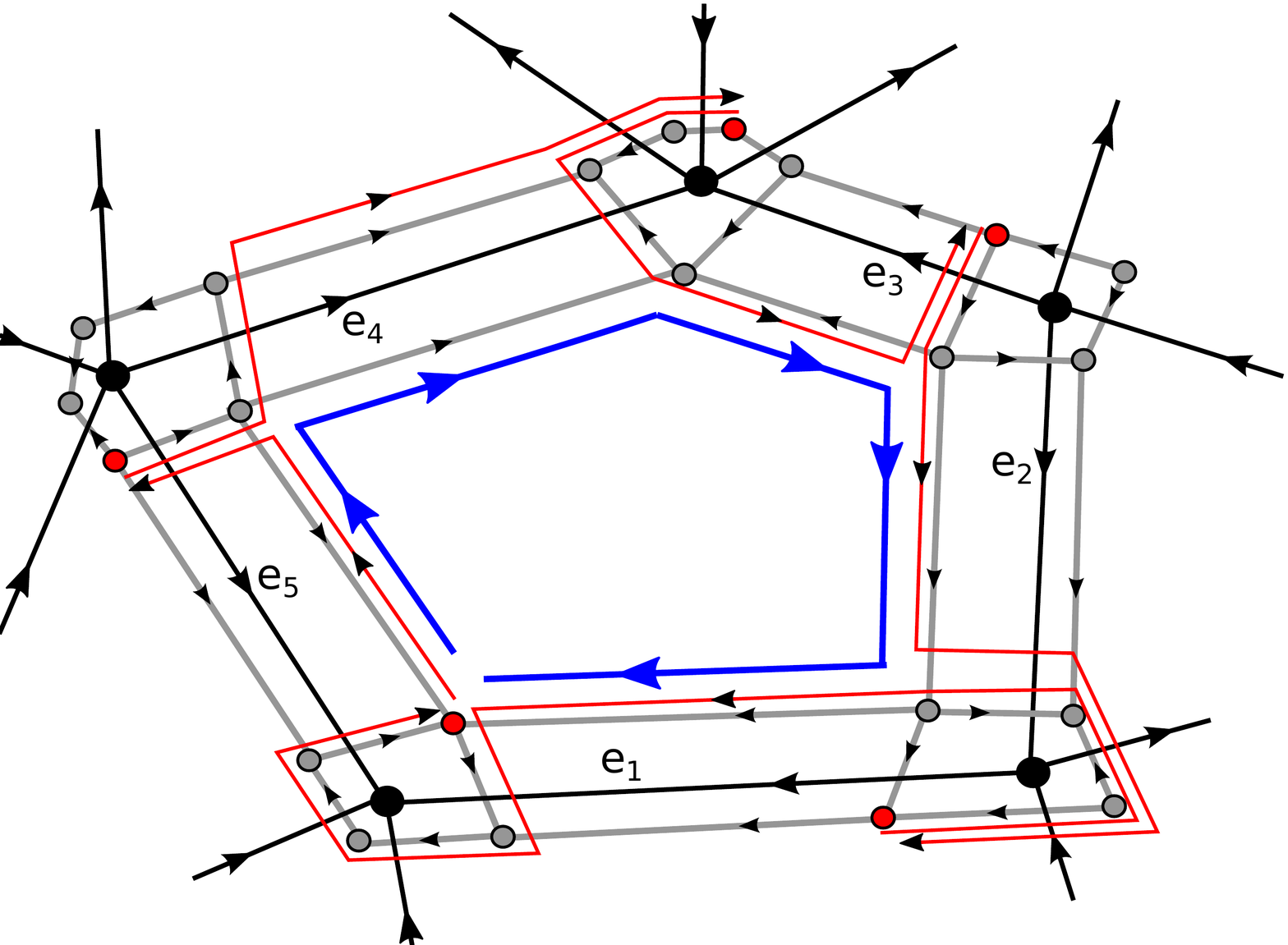}
\caption{The  paths $p_{e_i,\pm}$ in $\gammad$ for the ciliated face $f(v)=e_1\circ e_2\circ e_3^\inv\circ e_4\circ e_5^\inv$ .}
\label{fig:faceedgepath}
\end{figure}

In particular, Proposition \ref{lem:facetransfer} shows that projecting out the curvature of the ciliated face $f(v)$ in the Hopf algebra gauge theory amounts to projecting out  the curvatures of both, the ciliated faces $p_v$ and $p_{f(v)}$, in the Kitaev model. This allows one to relate  the projectors on the subalgebras of gauge invariant functions on flat gauge fields in the two models. 

By composing the projectors $P_{f(v)}:\mathcal A^*_{\Gamma}\to\mathcal A^*_{\Gamma}$ from Lemma \ref{lem:faceprpr} for all vertices $v$ of a regular ribbon graph $\Gamma$ and restricting them to the gauge invariant subalgebra $\mathcal A^*_{\Gamma\,inv}\subset \mathcal A^*_\Gamma$, one obtains the projector $P_{flat}:\mathcal A^*_{\Gamma\,inv}\to\mathcal A^*_{\Gamma\, inv}$ from Theorem \ref{th:flatinv}. 
Its image is the quantum moduli algebra $\mathcal M_\Gamma\subset \mathcal A^*_\Gamma$, which  is a topological invariant  and can be interpreted as the algebra of gauge invariant functions on the linear subspace of flat gauge fields.

Similarly, composing the projectors $Q_v:\mathcal H(H)^{\oo E}\to \mathcal H(H)^{\oo E}$ from Lemma \ref{lem:hdproj2} for all vertices $v$ of $\Gamma$  and restricting them to the gauge invariant subalgebra $\mathcal H(H)^{\oo E}_{inv}$ yields the projector $Q_{flat}:\mathcal H(H)^{\oo E}_{inv}\to\mathcal H(H)^{\oo E}_{inv}$  from Lemma \ref{lem:hdproj2}.  Its image, the subalgebra  $\mathcal H(H)^{\oo E}_{flat}\subset \mathcal H(H)^{\oo E}_{inv}$, can be interpreted as the algebra of operators acting on the protected space and contains those  elements $X\in \mathcal H(H)^{\oo E}$ with $\ham\cdot X\cdot \ham=X$. As the protected space is a topological invariant, the same holds for the algebra  $\mathcal H(H)^{\oo E}_{flat}$. 
With Proposition \ref{lem:facetransfer}
it then follows that $\chi$ induces an algebra isomorphism between this algebra and the quantum moduli algebra of the Hopf algebra gauge theory and hence relates the topological invariants of the two models.

\begin{theorem}\label{th:modth} Let $\Gamma$ be a regular ciliated ribbon graph. Then   $\chi: \mathcal A^{*}_\Gamma\to\mathcal H(H)^{op \oo E }$   induces an algebra isomorphism
$\chi: \mathcal M_\Gamma\to \mathcal H(H)^{op \oo E}_{flat}$.
\end{theorem}
\begin{proof} 
By Theorem \ref{th:flatinv} and Lemma \ref{lem:hdproj2}, the subalgebras $\mathcal M_\Gamma\subset \mathcal A^*_{\Gamma\,inv}$  and  $\mathcal H(H)^{op\oo E}_{flat}\subset \mathcal H(H)^{op\oo E}_{inv}$ are the images  of the projectors 
$$P_{flat}=\Pi_{f\in F} P_f: \mathcal A^*_{\Gamma\,inv}\to \mathcal A^*_{\Gamma\,inv}\qquad 
Q_{flat}=\Pi_{v\in V} Q_v: \mathcal H(H)^{op\oo E}_{inv}\to\mathcal H(H)^{op\oo E}_{inv},$$
and all projectors $P_f$ and $Q_v$ commute. 
By Theorem \ref{lem:actedge} the map $\chi:\mathcal A^*_{\Gamma\, inv}\to\mathcal H(H)^{op\oo E}_{inv}$ is an algebra isomorphism, and by 
Proposition \ref{lem:facetransfer} one has $\chi\circ P_{f(v)}=Q_v\circ\chi$ for all  vertices $v\in V$. As $\Gamma$ is regular, the faces of $\Gamma$ are in bijection with the ciliated faces $f(v)$  based at the cilia of the vertices  $v\in V$. This implies  $\chi\circ P_{flat}=Q_{flat}\circ \chi$, and 
 the claim follows.
\end{proof}

Theorems \ref{lem:actedge} and \ref{th:modth} also allow one to understand excitations in the Kitaev model from the viewpoint of Hopf algebra gauge theory. Creating an excitation in a Kitaev model at a site $(v,f(v))$ corresponds to removing the associated vertex operator $A_v^\ell$ and face operator $B^\eta_{f(v)}$ from the Hamiltonian $\ham$.
On the level of triangle operators, this amounts to removing the gauge transformations at the vertex $v$  from the projector $Q_{inv}$ in Lemma \ref{lem:hdproj} and the projector $Q_v$ from the projector $Q_{flat}$ in Lemma \ref{lem:hdproj2}. 
  By Theorem  \ref{lem:actedge} and Theorem  \ref{th:modth}, this is equivalent  to removing the  projector $P_v$ from the projector  $P_{inv}$ in Theorem \ref{th:ginvsub} and removing the projector $P_{f(v)}$ 
 from the projector $P_{flat}$ in Theorem \ref{th:flatinv}.
Hence, from the viewpoint of  Hopf algebra gauge theory, excitations are created by relaxing gauge invariance at a vertex $v$ of $\Gamma$
and flatness at the associated face $f(v)$.

Together, Theorems \ref{th:chiiso},  \ref{lem:actedge} and  \ref{th:modth} establish the full equivalence of a $D(H)$-valued Hopf algebra gauge theory on a regular ribbon graph  $\Gamma$ and the associated Kitaev model for $H$.  Theorem \ref{th:chiiso} 
shows that there is an algebra isomorphism 
between the  algebra of functions  of the Hopf algebra gauge theory and the algebra of triangle operators in the Kitaev model.  Theorem \ref{lem:actedge} shows that  this algebra isomorphism  is compatible with the action of gauge transformations at the vertices of $\Gamma$ and hence induces an isomorphism between the  subalgebras of gauge invariant functions in the two models.
Proposition \ref{lem:facetransfer}
 then establishes that this isomorphism maps curvatures in the Hopf algebra gauge theory to curvatures in the Kitaev model.  
As the curvatures in the Hopf algebra gauge theory define a $C(D(H))^{\oo F}$-module structure on the gauge invariant subalgebra by Lemma \ref{lem:facecentral} and one has  $C(D(H))\cong Z(D(H))$ and $|F|=|V|$, this defines $Z(D(H))^{\oo V}$-module  structures on the gauge invariant subalgebras of the two models.  From Theorem \ref{th:modth} we then obtain  
 an algebra isomorphism between the associated subalgebras of invariants. These are the topological invariants of the models, the quantum moduli algebra $\mathcal M_\Gamma$ of the Hopf algebra gauge theory and the algebra $\mathcal H(H)^{op\oo E}_{flat}$ of operators on the protected space in the Kitaev model. 
 
Denoting by $\mathcal G=D(H)^{\oo V}$ the Hopf algebra of gauge transformations and by  $\mathcal C=Z(D(H))^{\oo V}$ the algebra of curvatures, we can describe the relation between Hopf algebra gauge theory and Kitaev models by the following commuting diagram that summarises the results of this article:

 \begin{align*}
\xymatrix{
\mathcal A^*_{\Gamma}\oo \mathcal G  
\ar[rd]^{\lhd}  \ar[rdd]_{P_{inv}\oo\epsilon^{\oo V}}  \ar[rrr]^{\chi\oo\id} & & & 
\mathcal H(H)^{op\oo E}\oo \mathcal G 
\ar[ldd]^{Q_{inv}\oo\epsilon^{\oo V}} \ar[ld]_{\lhd} & 
{\begin{array}{l}\text{action of}\\ \text{gauge transformations}\end{array}}\quad
 \\
& \mathcal A^*_\Gamma \ar[r]^{\chi} \ar[d]^{P_{inv}} & 
\mathcal H(H)^{op\oo E} \ar[d]_{Q_{inv}}&  & \text{functions on gauge fields}
\\
& \mathcal A^*_{\Gamma\, inv} \ar[r]^{\chi} \ar[d]^{P_{flat}} 
& \mathcal H(H)^{op\oo E}_{inv}  \ar[d]_{Q_{flat}} & &\text{gauge invariant functions}
\\
 & 
 \mathcal M_\Gamma \ar[r]^{\chi}& 
 \mathcal H(H)^{op\oo E}_{flat} &  & {\begin{array}{l}\text{gauge invariant functions}\\ \text{on flat gauge fields}\end{array}}
\\
  \mathcal A^*_{\Gamma\, inv}\oo \mathcal C
  \ar[ruu]^{\lhd} \ar[ru]_{ P_{flat}\oo\epsilon^{\oo V}}  \ar[rrr]_{\chi\oo\id}  & & &  \mathcal H(H)^{op\oo E}_{inv}\oo\mathcal C \ar[luu]_{\lhd} \ar[lu]^{Q_{flat}\oo\epsilon^{\oo V}\quad} & 
{\begin{array}{l}\text{action of}\\ \text{curvatures}\end{array}}\qquad\qquad\qquad\;\;
}
\end{align*}

\vfill
{\large \bf Acknowledgements}.  I thank  Derek Wise and John Baez for discussions. This work was supported by 
 the Action MP1405 QSPACE from the European Cooperation in Science and Technology (COST).



\end{document}